\documentclass{amsart}
\usepackage{amsmath}
\pdfoutput=1

\usepackage{graphicx}
\usepackage{latexsym}
\usepackage{color}
\usepackage{amscd}
\usepackage[all]{xy}
\usepackage{enumerate}
\usepackage{enumitem}
\usepackage{hyperref}
\usepackage{subfigure}
\usepackage{soul}
\usepackage{comment}
\usepackage{epsfig}
\usepackage{epstopdf}
\usepackage{marginnote}
\usepackage{multirow}

\newtheorem{thm}{Theorem}

\newtheorem{prop}[thm]{Proposition}

\newtheorem{theorem}[thm]{Theorem}

\theoremstyle{definition}

\newtheorem*{definition*}{Definition}

\newtheorem{remark}[thm]{Remark}

\newtheorem{example}[thm]{Example}

\newtheorem*{thma}{Theorem~A}

\newtheorem*{thmd}{Theorem~D}

\newtheorem*{corb}{Corollary\,B}
\newtheorem*{corc}{Corollary\,C}

\newcommand{\CPb}{\overline{\mathbb{CP}}{}^{2}}
\newcommand{\CP}{{\mathbb{CP}}{}^{2}}
\newcommand{\CPo}{{\mathbb{CP}}{}^{1}}

\newcommand{\RP}{{\mathbb{RP}}{}^{2}}

\newcommand{\C}{\mathbb{C}}

\newcommand{\Z}{\mathbb{Z}}

\newcommand{\K}{{\rm K3}}

\newcommand{\twprod}{\mathbin{\mathchoice%
    {\ooalign{\raise1.15ex\hbox{$\scriptstyle\sim$}\cr\hidewidth$\times$\hidewidth\cr}}%
    {\ooalign{\raise1.15ex\hbox{$\scriptstyle\sim$}\cr\hidewidth$\times$\hidewidth\cr}}%
    {\ooalign{\raise.85ex\hbox{$\scriptscriptstyle\sim$}\cr\hidewidth$\scriptstyle\times$\hidewidth\cr}}%
    {\ooalign{\raise.65ex\hbox{$\scriptscriptstyle\sim$}\cr\hidewidth$\scriptscriptstyle\times$\hidewidth\cr}}%
    }}

\newcommand{\QED}{\hfill \ensuremath{\Box}}
\newcommand{\M}{\operatorname{Mod}}

\usepackage{scalerel} 
\newcommand{\DT}[1]{t_{\scaleto{\mathstrut #1}{9pt}}}
\newcommand{\LDT}[1]{t_{\scaleto{\mathstrut #1}{9pt}}^{-1}}
\def \x {\times}
\def \eu{{\text{e}}}

\allowdisplaybreaks[4]

\begin{document}

\title[Lefschetz fibrations with arbitrary signature] 
{Lefschetz fibrations with arbitrary signature}

\author[R. \.{I}. Baykur]{R. \.{I}nan\c{c} Baykur}
\address{Department of Mathematics and Statistics, University of Massachusetts, Amherst, MA 01003, USA}
\email{baykur@math.umass.edu}

\author[N. Hamada]{Noriyuki Hamada}
\address{Department of Mathematics and Statistics, University of Massachusetts, Amherst, MA 01003, USA}
\email{hamada@math.umass.edu}

\begin{abstract} 
We develop techniques to construct explicit symplectic Lefschetz fibrations over the $2$--sphere with any prescribed signature $\sigma$ and any spin type when $\sigma$ is divisible by $16$. \linebreak This solves a long-standing conjecture on the existence of such fibrations with positive signature. As applications, we produce symplectic $4$--manifolds that are homeomorphic but not diffeomorphic to connected sums of $S^2 \times S^2$, with the smallest topology known to date, as well as larger examples as symplectic Lefschetz fibrations. 
\end{abstract}

\maketitle

\setcounter{secnumdepth}{2}
\setcounter{section}{0}




\vspace{0.15in}
\section{Introduction} 

An immense literature has been dedicated to the study of symplectic Lefschetz fibrations since the works of Donaldson and Gompf \cite{Donaldson, GompfStipsicz} have established Lefschetz fibrations over the \mbox{$2$--sphere} as topological counter-parts of closed symplectic $4$--manifolds, after blow-ups. However, \linebreak the possible values of one of the most fundamental invariants of \mbox{$4$--manifolds,} the \emph{signature} of a Lefschetz fibration over the $2$--sphere, has not been quite understood, despite many effective ways of calculating it being available, due to the works of Endo, Ozbagci and others \cite{Matsumoto, Endo, EndoNagami, Ozbagci, Smith, CengelKarakurt}. The aim of our article is to resolve this problem:

\smallskip
\begin{thma} \label{thma}
There exist infinitely many relatively minimal symplectic Lefschetz fibrations over the $2$--sphere, whose total spaces have any prescribed signature $\sigma \in \Z$ and any spin type when  $\sigma$ is divisible by $16$. All  the examples can be chosen to be simply-connected, minimal, and with fiber genus as small as $9$, as well as arbitrarily high. 
\end{thma}

\noindent Fiber summing with trivial fibrations over orientable surfaces, possibly with boundary, the main statements of the theorem  carry over to Lefschetz fibrations over \emph{any} compact orientable surface. 

It was often conjectured that Lefschetz fibrations over the $2$--sphere with positive signatures did not exist,  which constituted a long-standing open problem; see e.g. \cite{OzbagciThesis, Ozbagci}, \cite{StipsiczTalk}, \cite[Problems 6.3, 6.4]{StipsiczSurvey}, \cite[Problems 7.4, 7.5]{KorkmazStipsicz}. Our examples settle this conjecture in the negative. In contrast, in the literature there are plenty of examples of Lefschetz fibrations over the $2$--sphere with negative signatures, many coming from algebraic geometry. To the best of our knowledge, even in this case it was not known that every negative signature could be realized; see Remark~\ref{SignatureHistory}.

We will prove the theorem by explicitly describing the Lefschetz fibrations in terms of their monodromy factorizations, which correspond to positive Dehn twist factorizations in the mapping class group of an orientable surface.  Our constructions will heavily use variations of the \emph{breeding technique} \cite{BaykurGenus3, Hamada} for building Lefschetz pencils and fibrations out of lower genera ones. A great deal of our efforts will be spent on building  monodromy factorizations for \emph{spin} Lefschetz fibrations. \linebreak We should note that, even though we effectively use the breeding technique to derive new symplectic $4$--manifolds from smaller ones, this is not an inherently symplectic operation; we build pencils/fibrations with locally non-complex nodal singularities in intermediate steps, but then we match all the locally non-complex nodes with locally complex ones and remove these pairs at the end (which corresponds to a $5$--dimensional $3$--handle attachment, and at times to removing an $S^2 \times S^2$ summand from a connected sum); see Remark~\ref{breeding}.

We should also note that more flexible variations of this \emph{signature realization problem} are easier to address, even in the holomorphic category: For Lefschetz \emph{pencils}, where one allows base points, examples with prescribed signatures are easy to obtain using very ample line bundles on suitable compact complex surfaces. Similarly, there are many compact complex surfaces admitting semi-stable fibrations over non-rational complex curves; in fact, when one allows the base surface of the fibration to be of higher genera, there are even smooth fibrations, often called Kodaira fibrations, whose total spaces have positive signatures. While the essential challenge is to describe positive signature Lefschetz fibrations \emph{over the $2$--sphere}, we shall point out that most, and possibly all, of our examples in Theorem~A are not holomorphic; see Remark~\ref{Holomorphic}. Since there are no separating vanishing cycles in our fibrations, by Endo's signature formula \cite{Endo}, none of our examples with non-negative signatures can be hyperelliptic either.

Recall that two $4$--manifolds are said to be \emph{stably diffeomorphic} if they become diffeomorphic after taking connected sums of each with copies of $S^2 \times S^2$. By a classical result of C. T. C. Wall, we have the following immediate corollary to our theorem:

\begin{corb} 
Any closed smooth oriented simply-connected $4$--manifold is stably diffeomorphic to a symplectic Lefschetz fibration.
\end{corb}

Furthermore, by crafting the monodromies of our fibrations more carefully and invoking Freedman's celebrated work, we can show that any closed  smooth oriented simply-connected $4$--manifold $X$ with signature $\sigma$, provided its holomorphic Euler charactersitic \mbox{$\chi_h(X):=\frac{1}{4}(\eu(X)+ \sigma (X))$} is integral and sufficiently large (depending solely on $\sigma$), is \emph{homeomorphic} to a symplectic Lefschetz fibration. (Doing this for \emph{all} integral $\chi_h(X)$ greater than a constant that depends on $\sigma$ requires more work.) So we get symplectic Lefschetz fibrations as exotic copies of standard $4$--manifolds that are connected sums of copies of $S^2 \times S^2$, $\CP$ and the $\K$ surface, taken with either orientations. It is a very interesting problem to determine the smallest $\chi_h$ (or $b_2$) one can take for a given $\sigma$. For example, when $\sigma=0$, we can show that:

\smallskip
\begin{corc} \label{corc}
There exist symplectic Lefschetz fibrations, whose total spaces are pairwise non-diffeomorphic, but homeomorphic to $\#_m (S^2 \x S^2)$, where the examples can be chosen so that  $m$ is as small as $127$, or is any odd $m \geq 415$. 
\end{corc}

These are the first exampes of Lefschetz fibrations in the given homeomorphism classes. \linebreak The spin examples are particularly interesting, and have been sought after for quite some time, in connection with the existence of exotically knotted \emph{orientable} surfaces in the $4$--sphere. Examples of exotically knotted \emph{non-orientable} surfaces in the $4$--sphere were constructed by Finashin, Kreck and Viro \cite{FinashinKreckViro}, and further examples were given by Finashin \cite{Finashin}, all using involutions on genus--$1$ Lefschetz fibrations. In contrast, it is still unknown whether there are any examples of exotically knotted orientable surfaces in the $4$--sphere, whereas double covers along topologically but not smoothly trivial ones would yield exotic $\#_m (S^2 \x S^2)$; see \mbox{Remark \ref{KnottedSurfaces}.}  

The motivation for the Theorem~A in part comes from the \emph{symplectic geography problem}, pioneered by Gompf in \cite{Gompf}, which asks to determine the pairs of integers that can be realized as $\chi_h$ and $c_1^2=2 \eu + 3\sigma$ of minimal symplectic $4$--manifolds for a given spin type and fundamental group, akin to the geography problem for compact complex surfaces \cite{Persson, ZChen, PerssonEtal}.
This is because of the extensive use of Lefschetz fibrations in the past four decades as building blocks for new complex, symplectic and smooth $4$--manifolds. Employing symplectic surgeries, which do not preserve the fibration structure, we can get sharper results and realize new pairs of values as $(\chi_h, c_1^2)$ of simply-connected spin symplectic $4$--manifolds:

\begin{thmd} \label{thmd}
There exist symplectic Lefschetz fibrations over the $2$--torus, which are equivalent via Luttinger surgeries to infinitely many, pairwise non-diffeomorphic symplectic $4$--manifolds homeomorphic to $\#_m (S^2 \times S^2)$, where $m$ can be chosen as small as $23$, as well as arbitrarily large. They further yield examples of infinitely many, pairwise non-diffeomorphic symplectic $4$--manifolds homeomorphic to $\#_m (S^2 \times S^2)$ for \emph{any} odd $m \geq 23$.
\end{thmd}


To the best of our knowledge, Theorem~C contains examples with the smallest topology among the exotic $\#_m (S^2 \times S^2)$ discovered to date.  Symplectic $4$--manifolds homeomorphic but not diffeomorphic to $\#_m (S^2 \times S^2)$ were first constructed by \mbox{J. Park} in \cite{Park} for unspecified large $m$, and this result was improved dramatically by Akhmedov and \mbox{D. Park} to $m \geq 275$ \cite{APspin1} and later to $m \geq 175$ \cite{APspin2}.
(Recent talks by Sakall{\i}  announce further improvement to  $m \geq 27$ in joint work  with the same authors \cite{APS}, but this preprint is not publicly available at the time of writing.) All these previous works use compact complex surfaces with positive signatures built by algebraic geometers as essential ingredients. Our main ingredients, the symplectic Lefschetz fibrations over the $2$--torus we build,  are not homotopy equivalent to any complex surfaces.

We discuss how to describe spin structures on Lefschetz pencils and fibrations using their monodromy factorizations in Section~2. There we present a way to calculate the divisibility of the fiber class from the monodromy factorization, and produce a very handy criterion for equipping the fibration with a spin structure, building on Stipsicz's work in \cite{StipsiczSpin}; see Proposition~\ref{divisibility} and Theorem~\ref{SpinLF2}. Both are employed repeatedly in our constructions of spin Lefschetz fibrations to follow. In the proof of our main theorem, Theorem~A, the signature zero genus--$9$ Lefschetz fibration over the $2$--sphere singled out in Theorem ~\ref{KeyLF} will play a key role. The construction of this fibration, which spans the entire Section~3, is the most technically involved one in our paper, so we chose to present it in multiple steps, in each step producing positive factorizations for Lefschetz pencils that might be of particular interest themselves.
Even more effort is spent on establishing whether this fibration has a primitive fiber, since we were not able to  identify any sections of this fibration. (We plan to examine the existence of (multi)sections of this  fibration elsewhere.) We will finish the proof of Theorem~A in Section~4, and prove Corollary~C and Theorem~D in Section~5, using the explicit monodromies of our signature zero Lefschetz fibrations, while now paying special attention to killing the fundamental group efficiently so as to produce symplectic $4$--manifolds in the desired homeomorphism classes. 

\bigskip
\noindent \textit{Conventions:} Any Lefschetz fibration in this article is assumed to have non-empty critical set, and any Lefschetz pencil has non-empty base locus. All are assumed to be relatively minimal, i.e. no exceptional spheres contained in the fibers. Unless explicitly specified otherwise, the base of our Lefschetz fibrations is the $2$--sphere. We denote by $\Sigma_{g,k}^b$ a compact orientable surface of genus $g$ with $b$ boundary components and $k$ marked points in its interior, and we drop $b$ or $k$ from the notation when they are zero. The mapping class group $\M(\Sigma_{g,k}^b)$ then consists of orientation-preserving diffeomorphisms of $\Sigma_{g,k}^b$ which fix all the boundary points and marked points, modulo isotopies of the same type.  A right-handed or positive Dehn twist along a simple closed curve or loop $c$ on a surface $\Sigma$ is denoted by $t_c$ in $\M(\Sigma)$. Our products of mapping classes, and in particular of Dehn twists, act on curves starting with the rightmost factor.  We express a  monodromy factorization of a genus--$g$ Lefschetz fibration by $t_{c_1} \cdots t_{c_n} =1$ in $\M(\Sigma_g)$, and that of a genus--$g$ Lefschetz pencil with $b$ base points as $t_{c_1} \cdots t_{c_n} = t_{\delta_1} \cdots t_{\delta_b}$ in $\M(\Sigma_g^b)$, where $\{\delta_i$\} are boundary parallel curves along distinct boundary components of $\Sigma_g^b$. We often refer to these as positive factorizations (of identity or boundary multi-twist). The reader can turn to  \cite{GompfStipsicz} for general background on Lefschetz fibrations and pencils, monodromy factorizations and symplectic $4$--manifolds, and  to \cite{FarbMargalit, EndoNagami} for more on relations in the mapping class group and signature calculations (while keeping in mind that in \cite{FarbMargalit} the authors' convention is to use left-handed Dehn twists instead).

\vspace{0.2in}
\noindent \textit{Acknowledgements. } The first author was  supported by  a Simons Foundation Grant (634309) and an NSF grant (DMS-200532).

\clearpage
\section{Spin structures on Lefschetz pencils and fibrations}  \label{sec:spin}

In this section, we will first review some preliminary results on spin structures on Lefschetz pencils and fibrations, and present a few quintessential examples that will be used in our later constructions. We will then discuss how to calculate the divisibility of the homology class of the regular fiber of a Lefschetz fibration, and present a practical way to build a spin structure on its total space, solely using monodromy factorizations. The reader can  turn to \cite{Kirby, GompfStipsicz, StipsiczSpin} for basic definitions and background results on spin structures on $3$-- and $4$--manifolds, and to \cite{Browder, Johnson, Masbaum} for spin structures on surfaces, quadratic forms and spin mapping class groups.  While we will focus on Lefschetz pencils and fibrations, we note that \emph{everything} we discuss is applicable, mutatis mutandis, to achiral pencils and fibrations. 

Given a Lefschetz pencil $(X,f)$, we can easily determine whether $X$ admits a spin structure by studying the $\Z_2$--homology classes of the Dehn twist curves in the monodromy factorization of $f$:

\begin{theorem}[Baykur-Hayano-Monden \cite{BaykurHayanoMonden}] \label{SpinLP}
Let $(X,f)$ be a Lefschetz pencil with monodromy factorization \ $t_{c_1}\cdots t_{c_n}=t_{\delta_1}\cdots t_{\delta_b}$ in $\M(\Sigma_g^b)$, $b \geq 1$. Then $X$ admits a spin structure if and only if there is a quadratic form $q\colon H_1(\Sigma_g^b;\Z_2) \to \Z_2$ with respect to the $\Z_2$--intersection pairing, such that $q(c_i)=1$ for all i, and $q(\delta_j)= 1$ for some $j$. 
\end{theorem}

Given a Lefschetz fibration $(X,f)$, we can also determine whether $X$ admits a spin structure or not from the monodromy factorization of $f$, provided the integral homology class of the regular fiber $F$ is primitive and there is information available on the self-intersection of an algebraic dual $S$ of $F$  in $H_2(X)$, i.e. $F \cdot S =1$. (These extra conditions will be the focus of our discussions to follow.) A characterization in this case was given by Stipsicz in \cite{StipsiczSpin}, which motivates and predates Theorem~\ref{SpinLP}. We rephrase it as:

\begin{theorem}[Stipsicz \cite{StipsiczSpin}]  \label{SpinLF}
Let $(X,f)$ be a genus--$g$ Lefschetz fibration with monodromy factorization \,$t_{c_1} \cdots t_{c_n}=1$. Assume that the homology class of the regular fiber has an algebraic \mbox{dual $S$ in $H_2(X)$.} Then $X$ admits a spin structure if and only if  $S$ has even self-intersection, and there is a quadratic form $q\colon H_1(\Sigma_g;\Z_2) \to \Z_2$ with respect to the $\Z_2$--intersection pairing, such that $q(c_i)=1$ for all i.
\end{theorem}

In our revised statement of Theorem~\ref{SpinLF}, we replaced the original algebraic condition for the vanishing cycles $\{c_i\}$ given in \cite{StipsiczSpin}, with an equivalent condition in terms of a quadratic form, which we find to be handier for our calculations. (See below for an explanation of this condition.) Further, we have stated the existence of an algebraic dual as a hypothesis, which always holds when the homology class of the regular fiber is primitive in $H_2(X)$. In \cite{StipsiczSpin, StipsiczPrimitive} it was claimed that every Lefschetz fibration has a primitive fiber class, but there is a mistake in the proof of this claim (and there are counter-examples in the achiral case, like Matsumoto's genus--$1$ achiral Lefschetz fibration on $S^4$); see e.g. \cite[Appendix]{BaykurMinimality} for an explanation. In the case of a pencil, the induced handle decomposition removes the need for additional assumptions on the homological dual of the fiber class. Note that when $X$ is spin, the condition on the vanishing cycles is implied regardless of the divisibility of the fiber class. 

Recall that a function $q\colon H_1(\Sigma_g;\Z_2) \to \Z_2$ describes a quadratic form with respect to the \mbox{$\Z_2$--intersection} pairing if \,$q(a+b)=q(a)+q(b)+ a \cdot b$ (mod 2) for every $a, b \in H_1(\Sigma_g;\Z_2)$.\linebreak There is a bijection between the set of isomorphism classes of spin structures $\mathrm{Spin}(\Sigma_g)$  and the set of such quadratic forms on $H_1(\Sigma_g;\Z_2)$ \cite{Browder, Johnson} ---and a similar correspondence holds for $\Sigma_g^b$, where the quadratic form is no longer non-singular.  From the proof of the above theorems, one can in fact deduce that \emph{any} spin structure on $X$ comes from one on $\Sigma_g^b$ or $\Sigma_g$, for which the monodromy curves satisfy the spin property given in the statements.

An important invariant of a quadratic form $q$ on $H_1(\Sigma_g;\Z_2)$ is its $\Z_2$--valued \emph{Arf invariant}  $\mathrm{Arf}(q)$, which can be most easily calculated as  $\mathrm{Arf}(q) = \sum_{i=1}^q \, q(\alpha_i) q(\beta_i)$, where $\{\alpha_i, \beta_i\}$ is \emph{any} symplectic basis for $H_1(\Sigma_g;\Z_2)$. For $q_s$, $q_{s'}$ the quadratic forms corresponding to $s, s' \in \mathrm{Spin}(\Sigma_g)$ respectively, there exists a spin diffeomorphism $\phi\colon (\Sigma_g, s) \to (\Sigma_g, s')$, or equivalently $\phi^*(q_{s'}) = q_{s}$, \linebreak if and only if $\mathrm{Arf}(q_s)= \mathrm{Arf}(q_{s'})$ \cite{Arf, Johnson}.

For a fixed spin structure $s \in \mathrm{Spin}(\Sigma_g)$, let $\M(\Sigma_g, s)$ be the \emph{spin mapping class group}, which consists of $\phi \in \M(\Sigma_g)$ such that $\phi^*q=q$, where $q$ is the quadratic form corresponding to $s$. \linebreak 
Since $q$ respects the $\Z_2$--intersections, it follows from the Picard-Lefschetz formula that a Dehn twist $t_c$ is in $\M(\Sigma_g, s)$ if and only if $q(c)=1$. Now, if $Y$ is a  $\Sigma_g$--bundle over $S^1$ with \mbox{monodromy $\phi$,} then it admits a spin structure $\mathfrak{s}$ that restricts to a spin structure $s$ on the fibers if and only if $\phi \in \M(\Sigma_g, s)$. When this holds, attaching a Lefschetz $2$--handle to $Y \times I$ along a loop $c$ on a fiber prescribes a fibered cobordism from $Y$ to $Y'$, another $\Sigma_g$--bundle with monodromy $\phi'=t_c  \circ \phi$, if and only if $t_c \in \M(\Sigma_g, s)$ \cite{Radosevich, StipsiczSpin}. With these in mind, we will next illustrate how Dehn twists in  $\M(\Sigma_g, s)$ play a key role in building spin Lefschetz fibrations.

Let $(X,f)$ be a Lefschetz fibration with monodromy factorization \,$t_{c_1} \cdots t_{c_n}=1$, and assume that there is a quadratic form $q$ satisfying the hypothesis of Theorem~\ref{SpinLF}. The Lefschetz fibration $f$ on $X$ induces a decomposition $X= (D^2 \times \Sigma_g) \, \cup \, W \, \cup \, (D^2 \times \Sigma_g)$, where 
$W = (S^1 \x \Sigma_g \x I) \, \cup \,\sum_{i=1}^n \, h_i$
consists of Lefschetz $2$--handles $h_i$ attached along the Dehn twist curves $c_i$ on $\Sigma_g$. For $s \in \mathrm{Spin}(\Sigma_g)$ corresponding to the quadratic form $q$, take the spin structure on the first copy of $D^2 \times \Sigma_g$ which is the product of  the unique spin structure on $D^2$ with $s$. Induced spin structure $\mathfrak{s}$ on the boundary $\partial D^2 \x \Sigma_g$ is obviously the product of the bounding spin structure on $S^1$ and $s$ \cite{StipsiczSpin}. Now, by our discussion in the previous paragraph, $W$ yields a spin cobordism from $\partial D^2 \times \Sigma_g$  to $\partial X'$, for $X' =( D^2 \times \Sigma_g) \cup W$. As $t_{c_1} \cdots t_{c_n}=1$, there is a diffeomorphism $\partial X' \to  S^1 \times \Sigma_g$, which maps the induced spin structure $\mathfrak{s'}$ on $\partial X'$ to a spin structure $\mathfrak{s''}$ on $S^1 \times \Sigma_g$. However, for $\mathfrak{s''}$ to extend over the remaining $D^2 \times \Sigma_g$, it should be coming from a product spin structure where the spin structure on the $S^1$ factor is the bounding one.

We note that when we double the monodromy factorization, the spin condition for the vanishing cycles alone is sufficient to conclude that the total space of the new fibration is spin:

\begin{prop} \label{doubling}
Let $(X,f)$ be a Lefschetz fibration with monodromy factorization $(t_{c_1} \cdots t_{c_n})^2=1$, where \,$t_{c_1} \cdots t_{c_n}=1$  in $\M(\Sigma_g)$. Then $X$ admits a spin structure if and only if there is a quadratic form $q\colon H_1(\Sigma_g;\Z_2) \to \Z_2$ with respect to the $\Z_2$--intersection pairing, such that $q(c_i)=1$ for all i.
\end{prop}

\begin{proof}
In this case, the Lefschetz fibration $f$ with doubled monodromy induces a decomposition 
\[ X= (D^2 \times \Sigma_g) \, \cup \, W \, \cup \, W \, \cup \, (D^2 \times \Sigma_g) = X' \cup X' \, , \]
where, $W = (S^1 \x \Sigma_g \x I) \, \cup \sum_{i=1}^n h_i$, with $2$--handles $h_i$ attached along $c_i$ on $\Sigma_g$, is a spin cobordism as we argued above. So $X' = (D^2 \times \Sigma_g) \, \cup \, W$ admits a spin structure per our hypothesis. As before, $t_{c_1} \cdots t_{c_n}=1$ induces an identification $\Phi\colon \partial X' \to  S^1 \times \Sigma_g$, and for $\Psi$ the orientation-reversing diffeomorphism on $S^1 \times \Sigma_g$ which is the product of the complex conjugation on $S^1$ and identity on $\Sigma_g$, we have $X=X' \cup_{\Phi^{-1} \Psi \Phi} X'$. Clearly,  this gluing matches the induced spin structures $\mathfrak{s'}$ on both copies of $\partial X'$, and the spin structures we have on the two copies of $X'$ then extend to a spin structure on the whole of $X$. (So we see that if the cobordism $W$ from $S^1 \times \Sigma_g$ to itself happens to flip the spin structure on the $S^1$ factor, attaching $W$ for a second time, we still get back the bounding spin structure on the $S^1$ factor.) 
\end{proof}

The doubled monodromy in Proposition~\ref{doubling} obviously amounts to taking an untwisted fiber sum $(X,f)=(X',f')\, \# \,(X',f')$, where $f'$ has monodromy $t_{c_1} \cdots t_{c_n}=1$. In \cite{StipsiczFiberSum} Stipsicz argues that even if $f'$ does not have a section, $f$ always has one with even self-intersection. One can then invoke Theorem~\ref{SpinLF} to obtain another proof of the proposition. 

More generally, the essential information on an algebraic dual of the fiber in Theorem~\ref{SpinLF} would be readily available if the fibration $(X,f)$ admits a section $S$. While this is equivalent to having a lift of the monodromy $t_{c_1} \cdots t_{c_n}=1$ in $\M(\Sigma_g)$ to $t_{c'_1} \cdots t_{c'_n}=1$ in $\M(\Sigma_{g,1}$), finding such a lift ---which might require isotoping the original $\{c_i\}$ so they now have a lot more geometric intersections--- is a challenging job; see e.g. \cite{KorkmazOzbagci, Onaran, Tanaka, Hamada}. Nonetheless, once we have this lift, determining the self-intersection of the section is fairly easy. We have the short exact sequence
\[ 1 \rightarrow \langle t_\delta \rangle \rightarrow \M(\Sigma_g^1) \rightarrow \M(\Sigma_{g,1}) \rightarrow 1 \]
where the epimorphism is induced by capping the boundary component $\delta$ of $\Sigma_g^1$ by a disk with a marked point. So there is always a further lift $t_{c''_1} \cdots t_{c''_n}=t_\delta^k$ in $\M(\Sigma_g^1)$, which one can easily obtain by removing a disk neighborhood of the marked point and then calculating the number of boundary twists needed to get back to identity, by looking at the action of the monodromy on an essential arc on $\Sigma_g^1$. The power $k$ of the boundary-twist $t_\delta$ then determines the negative of the self-intersection number of $S$. This can be easily seen as follows:  The disk neighborhood of the marked point corresponds to a tubular neighborhood of $S$, and by picking a fixed point on the boundary of the disk, we get a  push-off of $S$, so now by the action of $t_{\delta}^k$, the intersection number of the two is $\pm k$. \linebreak Running this calculation in a single example (like the ones below), for once and all we determine the sign, and conclude that the self-intersection of $S$ is $-k$.

\smallskip
We continue with some quintessential examples:

\begin{example} \label{oddchain}
The odd chain relation in $\M(\Sigma_g^2)$, with $c_i$ as in Figure~\ref{fig:spinexamples_oddchain}
\begin{equation} \label{oddchainrelation}
 (t_{c_1} t_{c_2} \cdots t_{c_{2g+1}})^{2g+2}= t_{\delta_1} t_{\delta_2} \, ,
\end{equation}
is the monodromy factorization of a genus--$g$ pencil on a simply-connected K\"{a}hler surface, which is the $\K$ surface when $g=2$ \cite{GompfStipsicz}.  Since the $\Z_2$--homology classes of $c_1, \ldots, c_{2g+1}$ generate $H_1(\Sigma_g^2; \Z_2)$, and the sum of all odd indexed $c_i$ is homologous to each $\delta_j$, we see that there is in fact a unique quadratic form $q\colon H_1(\Sigma_g^2; \Z_2) \to \Z_2$ for which the monodromy curves satisfy the spin condition (corresponding to the unique spin structure on the $4$--manifold) if and only if \mbox{$g$ is even.} 
\end{example}

\begin{example} \label{evenchain}
The even chain relation  in $\M(\Sigma_g^1)$, with $c_i$ as in Figure~\ref{fig:spinexamples_evenchain}
\begin{equation} \label{evenchainrelation}
(t_{c_1} t_{c_2} \cdots t_{c_{2g}})^{4g+2}= t_\delta \, , 
\end{equation}
is the monodromy factorization of a genus--$g$ pencil on a simply-connected K\"{a}hler surface. Since the $\Z_2$--homology classes of the curves $c_1, \ldots, c_{2g}$ generate $H_1(\Sigma_g^1; \Z_2)$, 
there is in fact a unique quadratic form $q\colon H_1(\Sigma_g^1; \Z_2) \to \Z_2$ for which $q(c_i)=1$ for all $i$. To check the remaining condition on the boundary components,  take $c_{2g+1}$ to be any curve extending the $2g$--chain $c_1, \ldots, c_{2g}$ to a $(2g+1)$--chain on $\Sigma_g^1$. Now we see that all the odd indexed $c_i$ cobound a $\Sigma_0^{g+1}$, and together with $\delta$, they also cobound a $\Sigma_0^{g+2}$. So we have $q(c_{2g+1}) \equiv \sum_{i=1}^g q(c_{2i-1})$ (mod $2$), and  $q(\delta)\equiv q(c_{2g+1})+ \sum_{i=1}^g q(c_{2i-1}) \equiv 2 q(c_{2g+1}) \equiv 0$ (mod $2$). Hence, these $4$--manifolds are \emph{never} spin. 
\end{example}

\begin{figure}[htbp]
	\centering
	\subfigure[Odd chain. \label{fig:spinexamples_oddchain}]
	{\includegraphics[height=55pt]{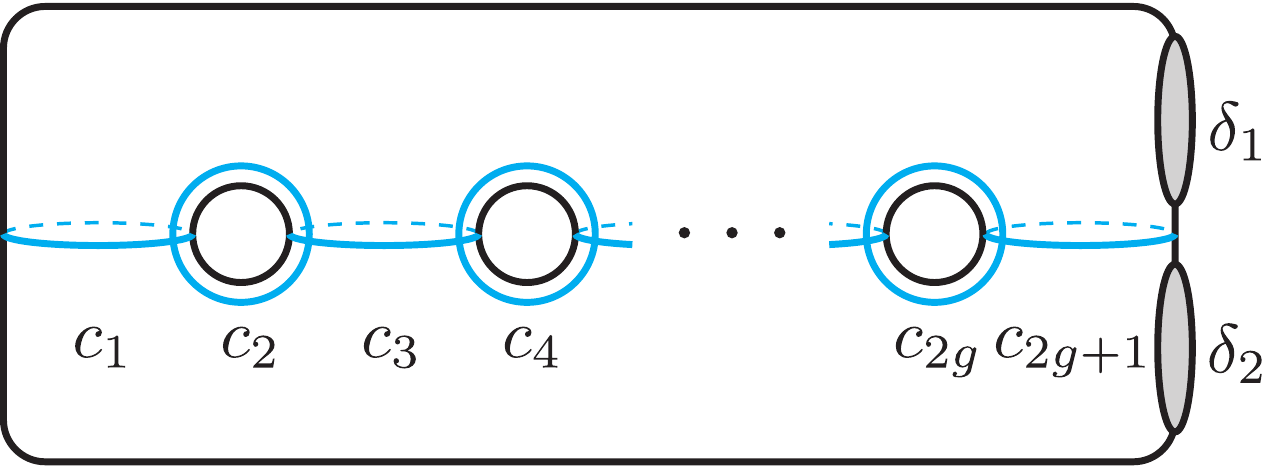}} 
	\hspace{10pt}
	\subfigure[Even chain. \label{fig:spinexamples_evenchain}]
	{\includegraphics[height=55pt]{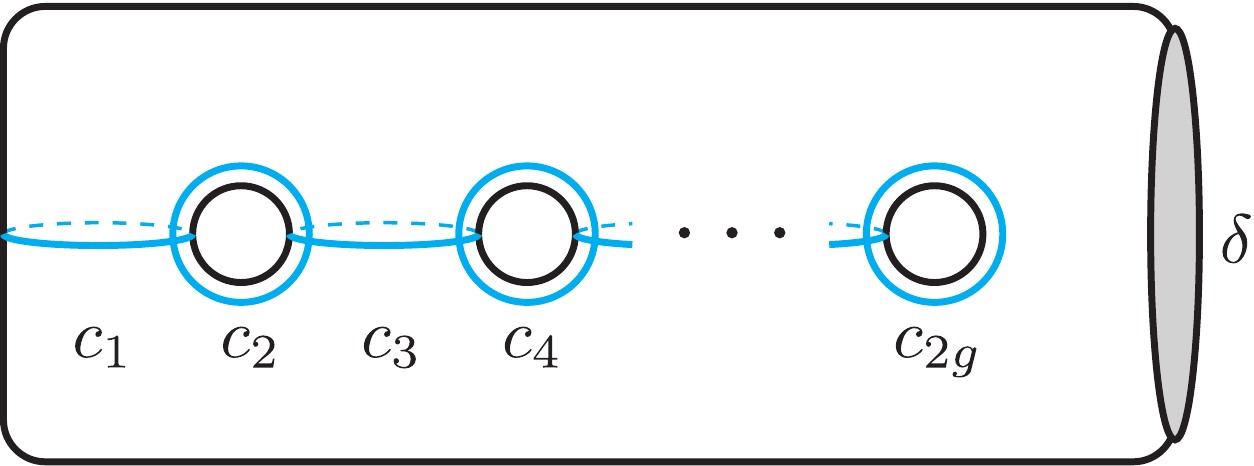}} 
	\subfigure[Hyperelliptic relation. \label{fig:spinexamples_hyperelliptic}]
	{\includegraphics[height=55pt]{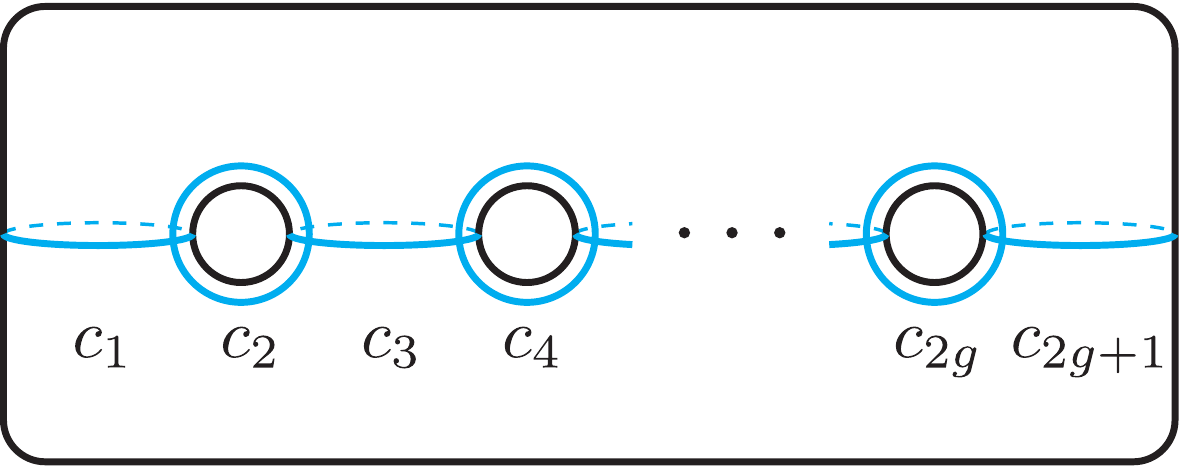}} 
	\caption{Dehn twist curves for standard relations.} 	
	\label{fig:spinexamples}
\end{figure}

\begin{example} \label{hyperelliptic}
The hyperelliptic relation  in $\M(\Sigma_g)$, with $c_i$ as in Figure~\ref{fig:spinexamples_hyperelliptic}
\begin{equation} \label{hyperellipticrelation}
(t_{c_1} t_{c_2} \cdots t_{c_{2g}} t_{c_{2g+1}}^2 t_{c_{2g}} \cdots t_{c_2} t_{c_1} )^2= 1 \, , 
\end{equation}
is the monodromy factorization of a genus--$g$ fibration on the rational surface \mbox{$\CP \# (4g+5) \CPb$,} and the double of this factorization is known to yield a genus--$g$ Lefschetz fibration on the elliptic surface \mbox{$E(g+1)$} \cite{GompfStipsicz}. Once again, the $\Z_2$--homology classes of the curves $c_1, \ldots, c_{2g+1}$ generate $H_1(\Sigma_g; \Z_2)$, and  because the odd indexed $c_i$ all together bound a subsurface, there is a unique quadratic form 
$q\colon H_1(\Sigma_g; \Z_2) \to \Z_2$ for which $q(c_i)=1$ for all $i$, if and only if $g$ is odd. 

Yet, even when $g$ is odd, the total space $\CP \# (4g+5) \CPb$ for the positive factorization~\eqref{hyperellipticrelation} is obviously not spin, despite the fiber class being primitive (which is immediate by the particular case of Proposition~\ref{divisibility} below, or by the fibration admitting sections). Indeed, this fibration arises as a blow-up of a pencil on a Hirzebruch surface, and admits as many as $4g+4$ exceptional sections, each one of which can be taken as the dual $S$. As an instance of Proposition~\ref{doubling} however, when we double the monodromy factorization~\eqref{hyperellipticrelation}, we get a spin Lefschetz fibration, which yields the unique spin structure on $E(g+1)$, for $g$ odd. This is also evident from Theorem~\ref{SpinLF} and the fact that by matching the exceptional sections, we see that the latter fibration has a section $S':=S \# S$ with self-intersection $-2$. Similarly, if we cap off the two boundary components in Example~\ref{oddchain}, we get a positive factorization of the identity, and doubling it, we get a monodromy factorization for a spin Lefschetz fibration, provided $g$ is odd. Note that in both examples, we have the same spin structure on $\Sigma_g$, and the corresponding quadratic form has Arf invariant  $1$ when $g \equiv 1$ (mod 4) and $0$ when $g \equiv 3$ (mod 4). 

On the other hand, if we cap off the boundary component in Example~\ref{evenchain}, the monodromy curves, which are in this case linearly independent in homology, already satisfy the spin condition for \emph{any} genus $g$, and doubling it yields a spin Lefschetz fibration. The quadratic form for this unique spin structure on $\Sigma_g$ now has Arf invariant $1$ when $g \equiv 1$ or $2$ (mod 4) and $0$ when $g \equiv 3$ or $4$ (mod 4). 
\end{example}

\begin{example} \label{yun}
A less known relation in $\M(\Sigma_g)$ discovered by K.-H. Yun \cite{Yun}, which is a simultaneous generalization of the hyperelliptic relation and the Matsumoto-Cadavid-Korkmaz relation \cite{Matsumoto, Korkmaz} (all coming from different involutions), is as follows: 
\begin{equation} \label{yunrelation}
(t_{A_{2n-2}} t_{A_{2n-3}} \cdots t_{A_2} t_{A_1}^2  t_{A_2} \cdots t_{A_{2n-3}} t_{A_{2n-2}} \,  t_{B_0} \cdots t_{B_{2m}} t_{A_{2n-1}} )^2 = 1 \, , 
\end{equation}
where $A_i, B_j$ are as in Figure~\ref{fig:yunrelation}, and $m, n >0$. This is the monodromy factorization of a genus $g=2m+n-1$ fibration on the ruled surface $(\Sigma_m \times S^2) \, \# \, 4n \CPb$, and importantly, a twisted fiber sum of two copies of this fibration is known to yield a genus--$g$ Lefschetz fibration with several $(-2$)--sections, the total space of which is a knot surgered elliptic surface $E(n)_K$, where $K$ is a fibered knot of genus $m$ \cite{FSLF, Yun}. Since $E(n)_K$ are homeomorphic to $E(n)$ \cite{FSKnotSurgery}, in particular we deduce that there is a twisted fiber sum of two copies of this fibration that is simply-connected, and moreover spin with signature $-8n$ when $n$ is even. 

Let $\{\alpha_i, \beta_i\}_{i=1}^{n-1} \cup \{\alpha'_j, \beta'_j\}_{j=1}^{2m}$ be a symplectic basis of $H_1(\Sigma_g; \Z_2)$, represented by the same labeled curves in Figure~\ref{fig:yunrelation}. We calculate the $\Z_2$--homology classes of the vanishing cycles as
\begin{align*}
A_1 &= \beta_1, \\
A_{2i-1} &= \beta_{i-1} + \beta_i \quad (i=2,\ldots,n-1), \\
A_{2n-1} &= \beta_{n-1}, \\
A_{2i} &= \alpha_i \quad (i=1,\ldots,n-1), \\
B_0 &= \alpha_1^\prime + \cdots + \alpha_{2m}^\prime + \beta_{n-1}, \\
B_{2j} &= \alpha_{j+1}^\prime + \cdots + \alpha_{2m-j}^\prime + \beta_j^\prime + \beta_{2m+1-j}^\prime + \beta_{n-1} \quad (j=1,\ldots, m-1), \\
B_{2m} &= \beta_m^\prime + \beta_{m+1}^\prime + \beta_{n-1}, \\
B_{2j-1} &= \alpha_j^\prime + \cdots + \alpha_{2m+1-j}^\prime + \beta_j^\prime + \beta_{2m+1-j}^\prime + \beta_{n-1} \quad (j=1,\ldots,m).
\end{align*}
If $n$ is odd then no quadratic form $q\colon H_1(\Sigma_g; \Z_2) \to \Z_2$ satisfies $q(A_i)=1$ for all odd $i$ since such $A_i$ cobound a subsurface $\Sigma_m^n$.
Suppose that $n$ is even.
Let $q(\alpha_i)=q(\alpha'_j)= q(\beta'_j)=1$ for all $i, j$, and set $q(\beta_{i})=1$ for all odd $1 \leq i \leq n-1$ and $q(\beta_{i})=0$ for all even $2 \leq i \leq n-2$. This defines a quadratic form on $H_1(\Sigma_g; \Z_2)$ for which $q(A_i)=q(B_j)=1$ for all $i, j$, and the Arf invariant of this quadratic form is $0$ when $n \equiv 0$ (mod 4) and $1$ when $n \equiv 2$ (mod 4).

\begin{figure}[htbp]
	\centering
	{\includegraphics[height=160pt]{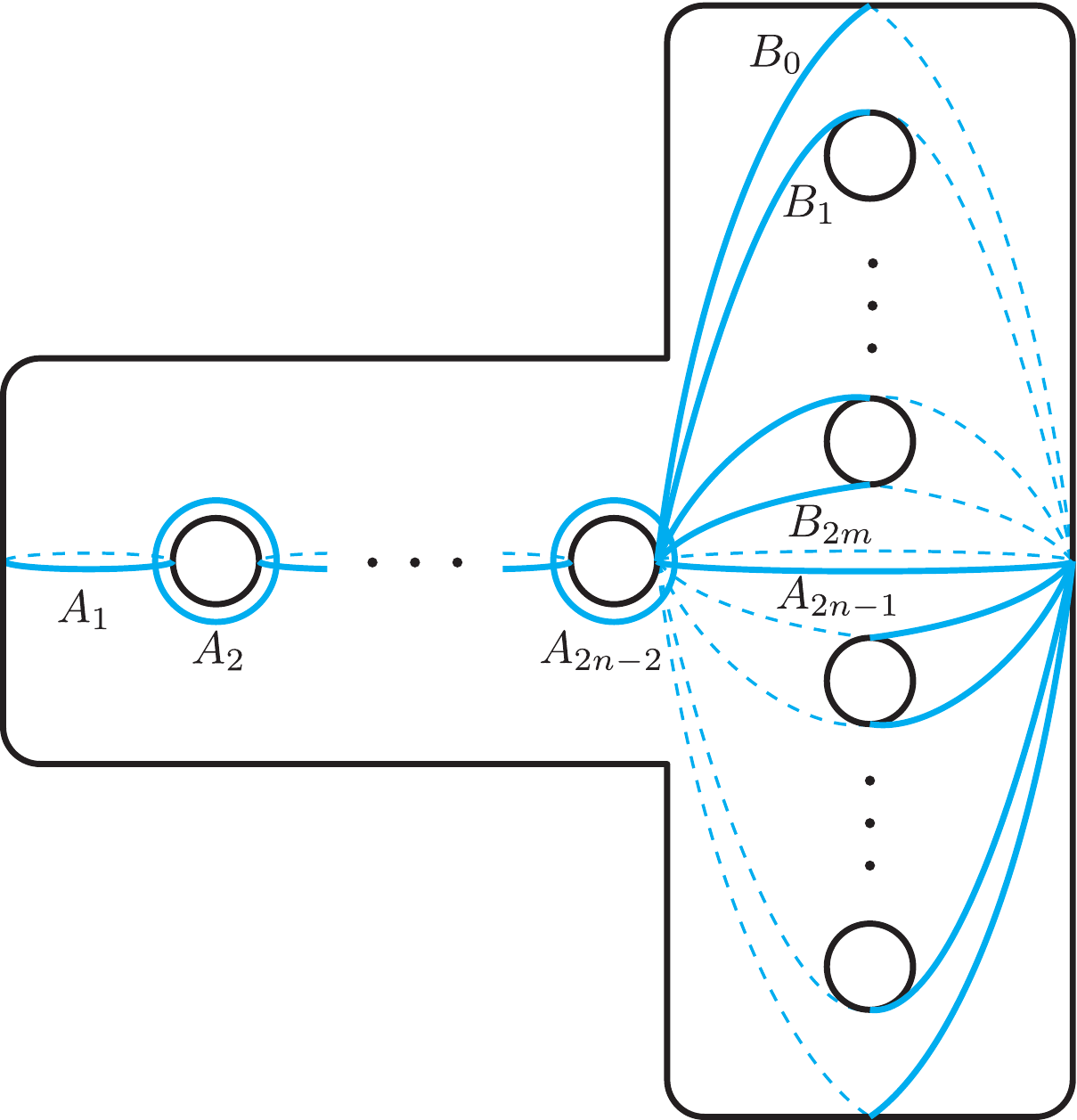}} 
	\hspace{10pt}
	{\includegraphics[height=160pt]{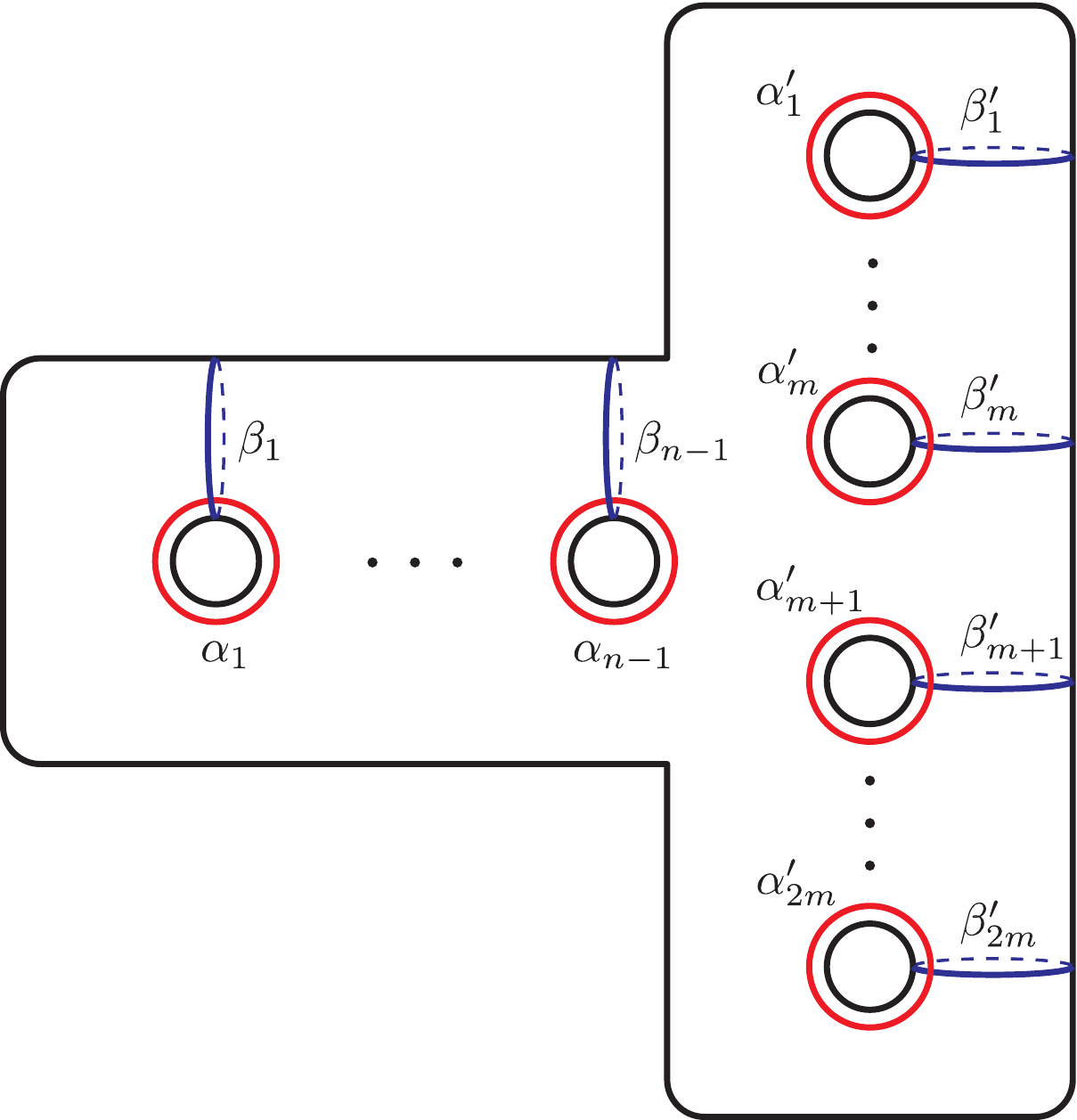}} 
	\caption{Dehn twist curves of Yun's relation and the homology generators.} 	
	\label{fig:yunrelation}
\end{figure}
\end{example}

We will next show, when there is no information available for a section, how, using the monodromy, we can still determine whether the fiber is primitive and thus has an algebraic dual as required in Theorem~\ref{SpinLF}. Since the kernel of the forgetful homomorphism $\M(\Sigma_{g,1}) \to \M(\Sigma_g)$ is generated by point-pushing maps \cite{FarbMargalit}, if we take an arbitrary marked point on $\Sigma_g$ avoiding all the $c_i$, we still get a lift \, $t_{c'_1} \cdots t_{c'_n}= \mathcal{P}_{\vec{\alpha}_1} \cdots \mathcal{P}_{\vec{\alpha}_m}  \, $ in $\M(\Sigma_{g,1})$, where on the right-hand side we now have point-pushing maps along oriented loops $\vec{\alpha}_i$. This product of point-pushing maps can be expressed as one point-pushing map along a possibly immersed oriented loop $\vec{\alpha}$, and when $\alpha$ is null-homotopic, the marked point we picked in fact gives an honest section of the fibration. 

\smallskip
\begin{prop} \label{divisibility}
Let $(X,f)$ be a genus--$g$ Lefschetz fibration, which has a monodromy factorization \mbox{$t_{c_1}\cdots t_{c_n}=1$,} with the following lift in $\M(\Sigma_{g,1})$:
\[ t_{c'_1}\cdots t_{c'_n}= \mathcal{P}_{\vec{\alpha}_1} \cdots \mathcal{P}_{\vec{\alpha}_m}  \, .\]
Let $\gamma_i$ be the homology class of $c_i$ taken with either orientation for $i=1, \ldots, n$, and let $\alpha$ be the sum of the homology classes of all $\vec{\alpha}_j$ in $H_1(\Sigma_g)$. 
The divisibility of the homology class of the regular fiber $F$ in $H_2(X)$ is the smallest positive integer $d$ such that the image of $d \alpha$ under the homomorphism $H_1(\Sigma_g) \to H_1(\Sigma_g)\, / \langle \gamma_1, \ldots, \gamma_n \rangle$ is trivial. In particular, if $\{\gamma_i\}_{i=1}^n$ generate $H_1(\Sigma_g)$, then $[F]$ is primitive.  
\end{prop}

\begin{proof}
We have $X= X' \cup (D^2 \times \Sigma_g)$, where $X'= (D^2 \times \Sigma_g) \cup \, \sum_{i=1}^n h_i$ and $\{h_i\}$ are the Lefschetz \mbox{$2$--handles} attached along $c_i$ on $\Sigma_g$. So $X'$ contains all the $2$--handles of $X$ but one, which comes from the second copy of $D^2 \times \Sigma_g$ in the decomposition of $X$. The isotopy, which takes the product $t_{c_1} \cdots t_{c_n}$ to $1$ in $\mathrm{Diff}^+(\Sigma_g)$, prescribes a fiber--preserving diffeomorphism $\Phi$ from the boundary fibration $f|_{\partial X'}$ to the trivial \mbox{$\Sigma_g$--fibration} on $\partial D^2 \times \Sigma_g$. The attaching circle $c'$ of the last $2$--handle $h'$ is a section $s':=\Phi^{-1}(\partial D^2 \times p)$ of $f|_{\partial X'}$, for some point $p \in \Sigma_g$. 

In the standard handle diagram of $X'$, the attaching circle $c'$ of $h'$ links the \mbox{$2$--handle} of the fiber  geometrically once, and possibly goes over the $1$--handles coming from $\Sigma_g$, and links with the attaching circles $\{c_i\}$ of the Lefschetz $2$--handles $\{h_i\}$. While the section $s'$ of $f|_{\partial X'}$ carries the first information, i.e. how $c'$ would go over the $1$--handles as it traverses around $\partial D^2 \times p$ once, there are still many ways  $h'$ can be attached along a circle $c$ in the complement of the attaching circles $\{c_i\}$ ---where $c$, even if it is not isotopic to $c'$ in the diagram, is still isotopic to $s'$ on ${\partial X'}$. Equivalently, there are several ways to isotope all the attaching circles $\{c_i\}$ in the diagram to curves  that will project on to $\Sigma_g$ still as embedded curves $\{c'_i\}$, while now avoiding a marked point corresponding to $c$. Each one of them prescribes a lift of the identity $t_{c_1} \cdots t_{c_n}=1$ in $\M(\Sigma_g)$ to a factorization $t_{c'_1}\cdots t_{c'_n}= \mathcal{P}_{\vec{\alpha}}$ in $\M(\Sigma_{g,1})$, where  $\vec{\alpha}$ is an oriented curve on $\Sigma_g$. However, $c$ and $c'$ only differ by handle slides over $h_i$, and the homology classes of their attaching circles on $\Sigma_g$ will only differ by relations generated by the homology classes of $\{c_i\}$ in $H_1(\Sigma_g)$. Thus, they will all map to the same homology class under the projection $H_1(\Sigma_g) \to  H_1(\Sigma_g)\, / \langle \gamma_1, \ldots, \gamma_n \rangle$.

Corresponding to a given lift $t_{c'_1}\cdots t_{c'_n}= \mathcal{P}_{\vec{\alpha}_1} \cdots \mathcal{P}_{\vec{\alpha}_m}$ in $\M(\Sigma_{g,1})$ is such a choice $c$ as above. To see this, let $\pi\colon S^1 \times \Sigma_g \to \Sigma_g$ be the projection, and let $t$ parametrize $S^1 = [0, 1] \, / 0\,{\small \sim}\,1$. Take a partition $0=t_0 < t_1 < \ldots < t_i < \ldots <t_m = 1$. For each $\vec{\alpha}_i$, we can find an oriented arc $\beta_i$ in [$t_{i-1}, t_i] \times \Sigma_g$, running from the point $t_{i-1} \times p$ to the point $t_i \times p$, so that the $\pi|_{\beta_i} \colon \beta_i \to \vec{\alpha}_i$ is an immersion. Concatenating all of them, we get an oriented arc $\beta$ in $[0,1] \times \Sigma_g$. Identifying the end points of either $+\beta$ or $-\beta$, we  obtain $c$. Just like our reading of the self-intersection number of a section from the corresponding monodromy factorization, the correct sign of $\beta$ here can be settled for once and all after calculating it in one example, but for our homology calculations to follow, the sign will not matter.

From the decomposition $X= X' \,\cup \, (D^2 \times \Sigma_g)$, it is clear that there is no class with a representative in $X'$ which has non-trivial pairing with the homology class of the regular fiber $F$ in $H_2(X)$.  By our discussion above, up to relations generated by the homology classes $\gamma_i$ of $c_i$ on $H_1(\Sigma_g)$, the attaching circle $c'$ of the last $2$--handle $h'$ is homologous to $c$ we build from a given lift, which in turn is homologous to a concatenation of $S^1 \times p$ and all $\vec{\alpha}_j $s  in  $H_1(S^1 \times \Sigma_g)=H_1(S^1) \times H_1(\Sigma_g)$. Since the $S^1 \times p$ curve already bounds $D^2 \times p$, when we attach $h'$, its attaching circle becomes homologous to $\pm \alpha$ in $H_1(S^1 \times \Sigma_g)$, for $\alpha = \sum_{j=1}^m [\vec{\alpha}_j]$ in $H_1(\Sigma_g)$. It follows that if the image of $\alpha$ is trivial in the quotient $H_1(\Sigma_g)\, / \langle \gamma_1, \ldots, \gamma_n \rangle \cong H_1(X')$, then the attaching circle of $h'$ bounds some surface $S'$ in $X'$. We can then form a closed surface $S$ in $X$ by gluing $S'$ and the core of $h'$ to obtain a surface which intersects the fiber $F$ once,  and conclude that $[F]$ is primitive. 

In general, if the image of $d\,\alpha$ is trivial in $H_1(\Sigma_g)\, / \langle \gamma_1, \ldots, \gamma_n \rangle$, we get a surface $S'$ in $X'$ that bounds $d$ parallel copies of the attaching circle of $h'$. So we can build a closed surface $S$ in $X$, which satisfies $S \cdot F=d$, by gluing $S'$ and $d$ parallel copies of the core of the $2$--handle $h'$. Conversely, if $[F]=d F_0$ for a primitive class $F_0$, let $S$ be a surface in $X$ representing the dual of $F_0$, so $S \cdot F= d$. If necessary, by tubing ---along $F$--- between any pair of positive and negative intersections of $S$ with $F$, we can replace $S$ with a higher genus surface in the same homology class, which now intersects $F$ geometrically $d$ times. We can then isotope this $S$ so that in a small neighborhood $N(F)$ of $F$ identified as $N(F) \cong D^2 \times \Sigma_g$, 
we have $S \cap N(F) \cong D^2 \times d \textrm{ points}$.
Identifying $N(F)$ with $D^2 \times \Sigma_g$ in the decomposition $X = X' \cup \, (D^2 \times \Sigma_g)$, we conclude that any class dual to the primitive root $F_0$ of $[F]$ can in fact be represented by a surface $S$ that splits into a surface $S'$ in $X'$ and $d$ parallel copies of the core of $h'$. Hence, the smallest positive integer $d$, for which  $d\,\alpha$ is trivial in the quotient, gives the divisibility of the fiber class $[F]$ in $H_2(X)$. 
\end{proof}

\clearpage

In the case of pencils, the induced handle decomposition nullifies the need for the additional information on the attaching of the last $2$--handle, and one can determine the divisibility of the fiber class of a genus--$g$ Lefschetz \emph{pencil} with $b$ base points directly using a lift of the monodromy in $\M(\Sigma_g^b)$; see \cite[Appendix]{HamadaHayano}.

We will call a surface $S$ a  \emph{pseudosection} of $(X,f)$ if it intersects a regular fiber $F$ once. As seen from our proof of Proposition~\ref{divisibility}, when the lift of the monodromy factorization to $\M(\Sigma_{g,1})$ involves non-trivial point-pushing maps, one gets a pseudosection (possibly after some handle slides over the Lefschetz $2$--handles), provided the class $\alpha$ determined by the point-pushing curves becomes trivial under the quotient homomorphism $H_1(\Sigma_g) \to H_1(\Sigma_g)\, / \langle \gamma_1, \ldots, \gamma_n \rangle$, and conversely any pseudosection yields such a lift.  However, unlike  in the case of a section, now a further lift to $\M(\Sigma_g^1)$ does not allow us to directly determine the self-intersection number of $S$.\footnote{In principle, one can of course attempt to access all this information by first drawing the standard handle diagram for $(X', f|_{X'})$ , then finding a sequence of Kirby moves from the induced diagram on $\partial X'$ to the standard handle diagram of $S^1 \times \Sigma_g$ with $2g$ $1$--handles and a $2$--handle, and finally dragging back a $0$--framed meridian to the $2$--handle in the latter diagram to a $2$--handle in the former, so as to calculate its framing, and so on. However, this not only proves to be a difficult exercise even for Kirby calculus aficionados, but also departs from our general approach to read off all the essential information from monodromy factorizations.}

We will show that, under favorable conditions that apply to all the examples we will deal with in this article, knowing that the fiber class is primitive will be enough to build the spin structure:

\begin{theorem}  \label{SpinLF2}
Let $(X,f)$ be a genus--$g$ Lefschetz fibration with a monodromy factorization \linebreak $t_{c_1} \cdots t_{c_n}=1$, a primitive fiber class in $H_2(X)$, and signature $\sigma(X)\equiv 0$ (mod 16). \linebreak Assume that there is a quadratic form $q\colon H_1(\Sigma_g;\Z_2) \to \Z_2$ with respect to the $\Z_2$--intersection pairing, such that $q(c_i)=1$ for all $i$, and  $\textrm{Arf}\, (q)=1$. Then \emph{any} spin structure $s'$ corresponding to a quadratic form $q^\prime$ on $H_1(\Sigma_g; \Z_2)$ with $q^\prime(c_i)=1$ for all $i$, induces a spin structure on $X$.
\end{theorem}

\begin{proof}
Let us first assume that there exists a genus--$g$ \emph{achiral} Lefschetz fibration $(Y_g, f_g)$, with a monodromy factorization $t_{d_1}^{\epsilon_1} \cdots t_{d_m}^{\epsilon_m}=1$ in $\M(\Sigma_g)$ (for some $\epsilon_i = \pm 1$),  which satisfies the following four properties: 
\begin{enumerate}[label=(\roman*)]
\item $\sigma(Y_g) \equiv 8$ (mod $16$), 
\item $f_g$ admits a section $S_g$ of (possibly positive) odd self-intersection,
\item there is an $s_g \in \mathrm{Spin}(\Sigma_g)$, with a quadratic form $q_g$ such that $q_g(d_i)=1$ for all $i$,
\item $\textrm{Arf}(q_g)=1$. 
\end{enumerate}
We will show in a bit that there are such model fibrations for every genus $g \geq 1$.

Since $(X,f)$ has a primitive fiber class, by our discussion above, it has a pseudosection $S$  intersecting a regular fiber $F$ geometrically once. Let $s \in \mathrm{Spin}(\Sigma_g)$ correspond to the quadratic form $q$ in the hypothesis. 

Removing a small neighborhood $N(F)$ of $F$, where $S$ intersects $F$ once, and $N(F_g)$ of the regular fiber $F_g$ of $(Y_g, f_g)$, we can take a fiber sum $(\hat{X}, \hat{f})=(X,f) \# (Y_g, f_g)$, so that under the gluing $\partial (X \setminus N(F))  \to \partial(Y \setminus N(F_g))$ we have  $S  \cap  \, \partial N(F)$ sent to  $S_g \cap \, \partial N(F_g)$, which we can always achieve after an isotopy for any given fiber-preserving diffeomorphism. So this fibration has a pseudosection $\hat{S}=(S \setminus S \cap N(F)) \cup (S_g \setminus S_g \cap N(F_g)) = S \# S_g$, with self-intersection $\hat{S} \cdot \hat{S} = S \cdot S + S_g \cdot S_g \equiv  S \cdot S + 1$ (mod $2$), as we assumed $S_g$ has odd self-intersection.

Identifying the boundaries $\partial N(F)$ and $\partial N(F_g)$ with $S^1 \times \Sigma_g$, the gluing diffeomorphism is prescribed by a self-diffeomorphism $\Phi$ of $S^1 \times \Sigma_g$, which is the product  of the complex conjugation of the unit circle $S^1 \subset \C$, and some $\phi \in \textrm{Diff}^+(\Sigma_g)$. Since both spin structures $s$ and $s_g$ have the same Arf invariant, we can choose $\phi$ to be a spin diffeomorphism $\phi\colon (\Sigma_g, s) \to (\Sigma_g, s_g)$. Now, the fibration $(\hat{X}, \hat{f})$ has monodromy curves $\{c_i, \phi^{-1}(d_j)\}$, with $q(c_i)=1$ and $q(\phi^{-1}(d_j))=q_g(d_j)=1$ for all $i, j$. If the pseudosection $S$ of $(X,f)$ had odd self-intersection, then $\hat{S}$ would have even self-intersection, and by Theorem~\ref{SpinLF}, $\hat{X}$ would be spin. However, by Novikov additivity, the signature $\sigma(\hat{X}) = \sigma(X) + \sigma(Y_g) \equiv 8$ (mod $16$), and Rokhlin's theorem implies that $\hat{X}$ cannot be spin. So we deduce that $S$ has even self-intersection, and thus we can invoke Theorem~\ref{SpinLF} to conclude that our original Lefschetz fibration $(X,f)$ is spin, in fact not only for $q$, but for any quadratic form $q^\prime$ with $q^\prime(c_i)=1$ for all $i$. 

We are left with building the models $(Y_g, f_g)$. Note that by Proposition~\ref{doubling}, any achiral Lefschetz fibration satisfying (iii) should necessarily have signature $\sigma(Y_g) \equiv 0$ (mod 8). Some of these models can be derived immediately from Examples~\ref{oddchain}--\ref{hyperelliptic}. When we cap off the boundary components, all these relations yield hyperelliptic Lefschetz fibrations, so using Endo's signature formula  for hyperelliptic fibrations \cite{Endo}, we can easily see that the even-chain relation yields Lefschetz fibrations with the desired properties when $g \equiv 2$ (mod 4), whereas all three relations yield such examples when $g \equiv 1$ (mod 4). 

To get models for all $g$ however, we will go about it a little differently, and build achiral fibrations out of the smallest chain relations instead ---while remembering that everything we have discussed so far also applies to achiral fibrations. Below, let $\{\alpha_i, \beta_i\}$ be a symplectic basis for $H_1(\Sigma_g; \Z_2)$ generated by the same labeled curves given in Figure~\ref{fig:symplecticgenerators}. 

\begin{figure}[htbp]
	\centering
	\includegraphics[height=60pt]{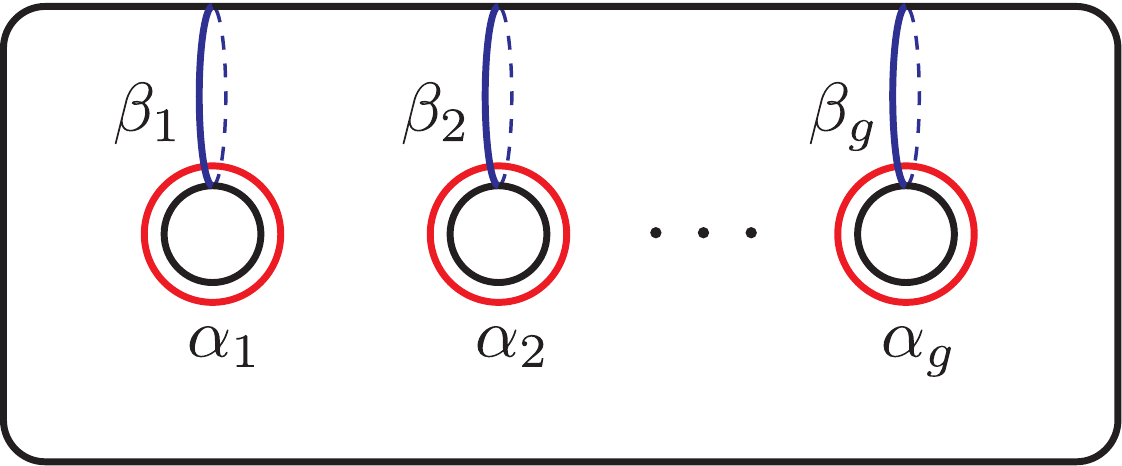}
	\caption{The symplectic basis for $H_1(\Sigma_g; \Z_2)$.} 
	\label{fig:symplecticgenerators}
\end{figure}

\noindent \underline{$g=1$}: Capped off even chain, odd chain and hyperelliptic relations all give Hurwitz equivalent positive factorizations, the monodromy curves of which evaluate as $1$ under the quadratic form prescribed by $q_1(\alpha_1)=q_1(\beta_1)=1$. Its signature is $-8$. 

\noindent \underline{$g=2$}: Capped off even chain relation gives a positive factorization, the monodromy curves of which evaluate as $1$ under the quadratic form prescribed by $q_2(\alpha_1)=q_2(\beta_1)=q_2(\alpha_2)=1$ and $q_2(\beta_2)=0$. Its signature is $-24$.

\noindent \underline{\textit{Every odd $g \geq 3$}}: We have a decomposition $\Sigma_g = \Sigma_1^1 \cup \Sigma_1^2 \cup \ldots \cup \Sigma_1^2 \cup \Sigma_1^1$, with $2k-1$ copies of $\Sigma_1^2$ for $g=2k+1$. Let $A$ denote the $2$--chain relation $(t_{c_1}t_{c_2})^6 t_{\delta}^{-1}$ and $B$ denote the $3$--chain relation $(t_{c_1}t_{c_2}t_{c_3})^4 t_{\delta_1}^{-1} t_{\delta_2}^{-1}$. Embed $A$ in both copies of $\Sigma_1^1$, and embed $B^{-1}$ into the first copy of $\Sigma_1^2$, $B$ into the second copy, and keep alternating the embeddings in the same fashion for all the copies of $\Sigma_1^2$. The embeddings of the boundary twists $t_{\delta}, t_{\delta_1}, t_{\delta_2}$ all cancel out with each other. So we get an achiral fibration $(Y_g, f_g)$, for $g=2k+1$, with only non-separating monodromy curves as in Figure~\ref{fig:achiralmodels}(a). Take the quadratic form $q$ which maps $q_g(\alpha_i)=q_g(\beta_i)=1$ for all $i$, so $q$ maps all the monodromy curves  to $1$. By \cite{EndoNagami}, taking the algebraic sum of the signatures of the relations $\sigma(A)=-7$ and  $\sigma(B)=-6$, we calculate that $\sigma(Y_g)= 2(-7)+ k(6) +(k-1)(-6)= -8$. 

\smallskip
\noindent \underline{\textit{Every even $g \geq 4$}}: This time we have a decomposition $\Sigma_g = \Sigma_1^1 \cup \Sigma_1^2 \cup \ldots \cup \Sigma_1^2 \cup \Sigma_2^1$, with $2k-1$ copies of $\Sigma_1^2$ for $g=2k+2$. Let $A$ and $B$ denote the $2$--chain and $3$--chain relations as above, and in addition, let $C$ denote the $4$--chain relation $(t_{c_1}t_{c_2}t_{c_3}t_{c_4})^{10} \, t_{\delta}^{-1}$. Embed $A$ into $\Sigma_1^1$, then $B^{-1}$ and $B$ into the copies of $\Sigma_1^2$ in the same alternating fashion as above. This time we finish with embedding $C$ into the last $\Sigma_2^1$ piece. Once again, the embeddings of the boundary twists $t_{\delta}, t_{\delta_1}, t_{\delta_2}$ all cancel out with each other, and we  get an achiral fibration $(Y_g, f_g)$, for $g=2k+2$, with non-separating monodromy curves as in Figure~\ref{fig:achiralmodels}(b). Take the quadratic form $q_g$ which maps all $q_g(\alpha_i)=q_g(\beta_j)=1$ except for $q_g(\beta_{g-1})=0$. One can see that all the monodromy curves are mapped to $1$ under $q_g$. Using \cite{EndoNagami} again, by taking the algebraic sum of the signatures of the relations $\sigma(A)=-7$, $\sigma(B)=-6$ and $\sigma(C)=-23$, we calculate that $\sigma(Y_g)= (-7)+ k(6) +(k-1)(-6) +(-23)= -24$. 

It is easy to check that the quadratic forms we described for each fibration $(Y_g, f_g)$ has Arf invariant $1$. (In fact there is a unique quadratic form for each one, since the monodromy curves of each $f_g$ generate $H_1(\Sigma_g; \Z_2)$.) Finally, note that every $(Y_g, f_g)$ has exceptional sections: This is obvious for $g=1,2$, and the higher genera ones always have an exceptional section supported in the $\Sigma_1^1$ piece (for instance since the $2$--chain is obtained from the $3$--chain by capping off a boundary component which has a single Dehn twist). So each $(Y_g, f_g)$ has a section $S_g$ of odd self-intersection (in fact a lot of them).
\end{proof}

\begin{figure}[htbp]
	\centering
	\subfigure[Odd genus model.] 
	{\includegraphics[height=70pt]{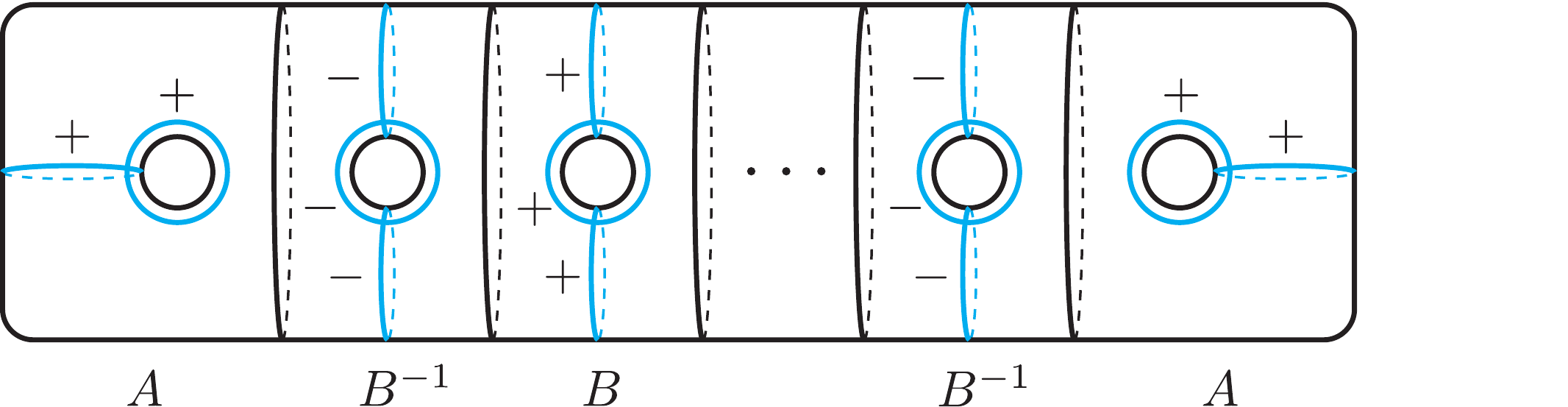}} 
	\hspace{10pt}
	\subfigure[Even genus model.] 
	{\includegraphics[height=70pt]{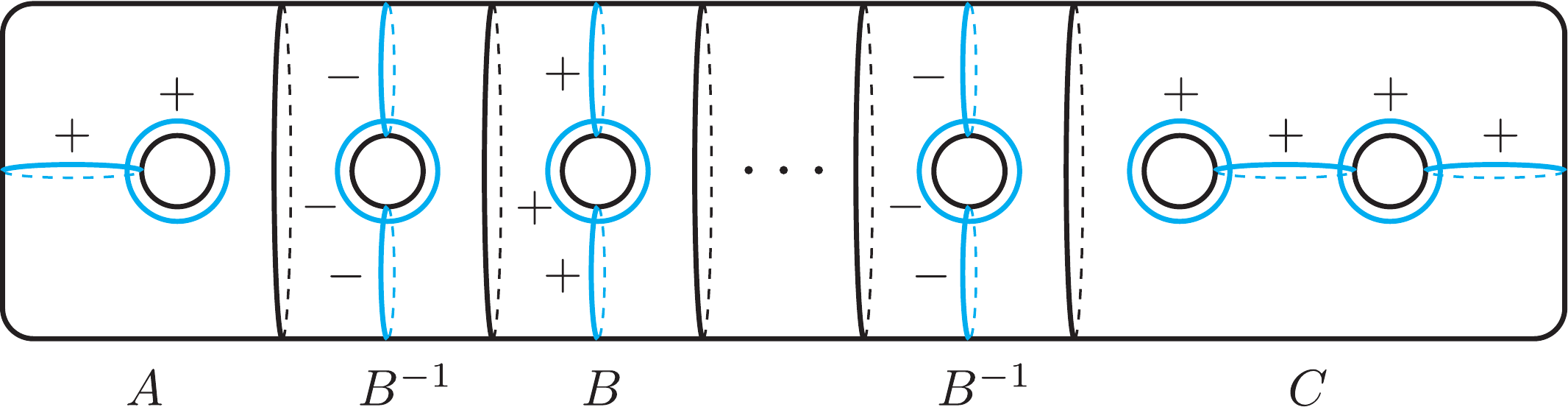}} 
	\caption{Model spin achiral Lefschetz fibrations.} 		
	\label{fig:achiralmodels}
\end{figure}

\smallskip
\begin{remark}
The converse to the statement of Theorem~\ref{SpinLF2} is not true. For instance, when we double the capped off even-chain relation on $\Sigma_g$ with $g \equiv 0$ or $3$ (mod 4), we get a spin Lefschetz fibration with signature divisible by $16$, but the only possible quadratic form on $\Sigma_g$, under which all the monodromy curves are mapped to $1$, has Arf invariant $0$. On the other hand, when $\textrm{rank}(H^1(X; \Z_2))>0$, one often finds spin structures on $(X,f)$ coming from quadratic forms on the fiber $\Sigma_g$ with either Arf invariant, which will be the case for our key example given in Theorem~\ref{KeyLF}. (While this is often the case, it is not always true either; the double of the genus $g=2m+1$ Matsumoto-Cadavid-Korkmaz fibration on $(\Sigma_m \times S^2) \, \#8 \CPb$ has $2^{2m}$ distinct spin structures, but a calculation similar to ours in Example~\ref{yun} shows that \emph{all} of these spin structures come from a quadratic form on $\Sigma_g$ with Arf invariant $1$!)
\end{remark}

\begin{remark}
While the achiral Lefschetz fibrations $(Y_g, f_g)$ we built in the proof of Theorem~\ref{SpinLF2} might be of some interest, the $4$--manifolds $Y_g$ themselves, as well as their spin doubles can all be seen to decompose into a connected sum of standard simply-connected $4$--manifolds $S^2 \times S^2$, $\CP$ and $\K$--surface, taken with either orientations. To see this, first observe that they are all fiber sums of $4$--manifolds that are such connected sums themselves, along embedded spheres of opposite self-intersections (coming from canceling boundary twists), then apply classical cobordism arguments due to Mandelbaum and Moishezon, as in \cite{MandelbaumMoishezon, GompfElliptic, BaykurKnotSurgery}. In particular, none of them are symplectic when $g>2$. 
\end{remark}

\medskip
\section{Lefschetz fibrations with signature zero}  

We will construct our examples of signature zero Lefschetz fibrations in three steps, split into the next three subsections herein, where we employ the \mbox{breeding} technique \cite{BaykurGenus3} in increasing complexity, to build signature zero Lefschetz pencils and fibrations out of lower genera pencils. This is done through careful embeddings of the corresponding positive factorizations into the mapping class group of higher genera surfaces, so that one can cancel all the negative Dehn twists against positive ones. Our $3$-step construction will result in the signature zero genus--$9$ Lefschetz fibration $(X,f)$ of Theorem~\ref{KeyLF}.

The interested reader can see how the key signature zero example $(X,f)$ comes to life, without getting bogged down in more technical details. Here is the outline of our construction: We first build the relation~\eqref{eq:MatsumotoLP_IIA_MarkedPoint} in $\M(\Sigma_2^4)$ for a genus--$2$ pencil, using the $2$--chain and several lantern relations. We embed two copies of this particular relation into $\M(\Sigma_3^4)$, as shown in Figure~\ref{fig:Genus3Breeding}, in order to obtain a new relation for a genus--$3$ pencil. Importantly, this relation \eqref{eq:Genus3LP_4BoundingPairs} contains positive Dehn twists along four bounding pairs given in Figure~\ref{fig:Genus3LP_4BoundingPairs}, each one of which cobound two copies of $\Sigma_1^4$ with a pair of negative Dehn twist curves corresponding to the base points of the pencil. We then embed two copies of this special relation into $\M(\Sigma_9)$ as in Figure~\ref{fig:Genus9_Genus3Embedding1and2}. At this point we have four pairs of negative Dehn twists, which we will cancel one pair at a time, by carefully embedding four more copies of the same genus--$3$ relation, so that each time pairs of negative boundary twists match and cancel with positive bounding pair twists of the first two original  genus--$3$ relations, while a positive bounding pair  matches and cancels with a negative pair. Here one simply needs to see how the positive bounding pairs from the top and the bottom halves of $\Sigma_9$ in Figure~\ref{fig:Genus9_Genus3Embedding1and2} cobound a $\Sigma_3^4$, split into two copies of $\Sigma_1^4$, cobounded by these positive pairs and a negative pair from the first two embeddings of the genus--$3$ relation. 
Following the ingenious work of Endo and Nagami \cite{EndoNagami}, which allows one to calculate the signature of a Lefschetz fibration via elementary relations in the mapping class group, a simple algebraic count of the total number of $2$--chain and lantern relations we employed (as cancellations and braid relations do not affect the signature) confirms that the genus--$2$ and genus--$3$ pencils, and the genus--$9$ Lefschetz fibration we built at the end, all have signature zero.

During our $3$--step construction we will also chase around an additional marked point on the fiber, the information on which we will need only later when calculating the divisibility of the fiber class of $(X,f)$. The reader who is not interested in this particular calculation can safely ignore the additional point-pushing maps that appear in our monodromy factorizations. In fact, while the breeding technique plays an innovative role in the construction of $(X,f)$, because we present this fibration with an explicit positive factorization, the more conservative reader may choose to skip the next three subsections and verify its monodromy given in Theorem~\ref{KeyLF} in a straightforward fashion, using the Alexander method.  (This is a tedious but still manageable calculation, as we have observed so while doing our due diligence to test our monodromy using the same method.)

\smallskip
\subsection{A signature zero genus--$2$  pencil}  \label{SecGenus2} \

We begin with describing a genus--$2$ Lefschetz pencil, whose topology and monodromy both have special features we need for the later steps of our construction. Namely, we would like the total space to have signature zero and be spin, and the monodromy curves to contain a bounding pair and two disjoint separating curves. (The need for all these properties will become clear as we move on to next steps.) 

In order to make our construction as self-contained as possible, we will derive  all our relations from elementary relations that are known to generate the mapping class group, namely, the \mbox{$2$--chain} relation, the lantern relation, and the braid relation, along with commutativity, cancellation and conjugation. In the computations we will freely perform Hurwitz moves and cyclic permutations without stating explicitly.

\smallskip
\noindent \textit{\underline{A variation of the $6$--holed torus relation}}:
We first present a variation of what is known as the $6$--holed torus relation. The curves involved in our construction to follow are given in Figure~\ref{fig:6HoledTorus_Construction}.

\begin{figure}[htbp]
	\centering
	\subfigure[The $2$-chain relation and lantern relations. \label{fig:6HoledTorus_Construction}]
	{\includegraphics[height=120pt]{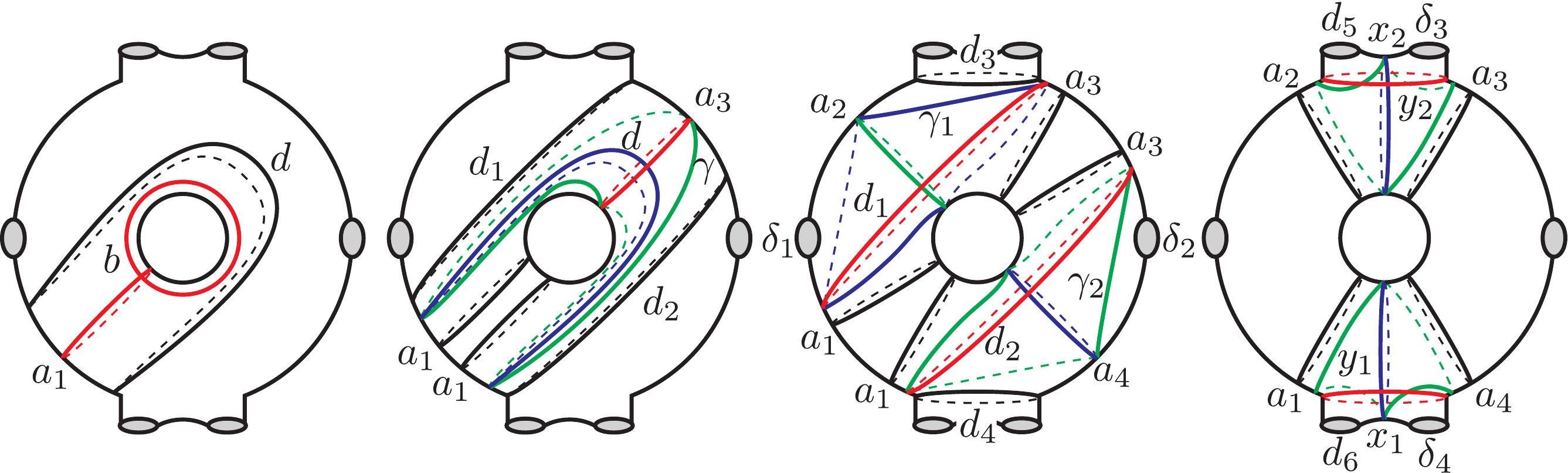}} 
	\hspace{10pt}
	\subfigure[Rearrangement of the boundary components. \label{fig:6HoledTorus_Rearrangement}]
	{\includegraphics[height=100pt]{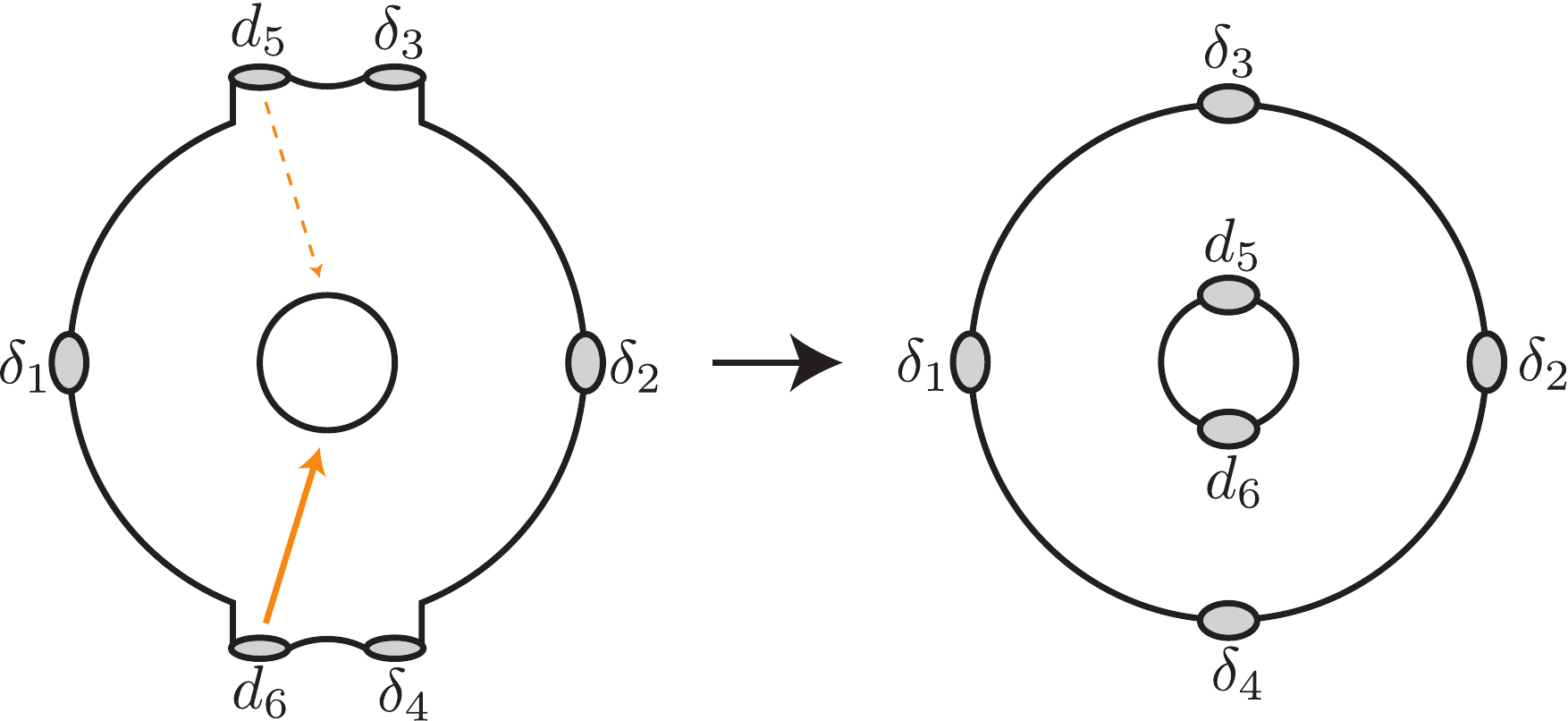}}  
	\caption{Construction of a $6$--holed torus relation.} 
	\label{fig:ConstructionOf6holedtorusrelation}
\end{figure}

We begin with the following $2$--chain relation and lantern relation:
\begin{align*}
(\DT{a_1} \DT{b})^6 &= \DT{d}, \\
\DT{a_3} \DT{d} \DT{\gamma} &= \DT{a_1} \DT{a_1} \DT{d_1} \DT{d_2}.
\end{align*}
We combine them as follows. (Here we underline the parts that we modify in that very step.) 
\begin{align*}
1 &= \DT{a_1} \DT{b} \DT{a_1} \DT{b} \DT{a_1} \DT{b} \DT{a_1} \DT{b} \DT{a_1} \DT{b} \DT{a_1} \DT{b} \LDT{d} \cdot \DT{d} \DT{\gamma} \DT{a_3} \LDT{a_1} \LDT{a_1} \LDT{d_1} \LDT{d_2} \\
&= \DT{a_1} \DT{b} \DT{a_1} \DT{b} \DT{a_1} \DT{b} \DT{a_1} \DT{b} \LDT{d} \cdot \DT{d} \DT{\gamma} \DT{a_3} \LDT{a_1} \LDT{a_1} \LDT{d_1} \LDT{d_2} \cdot \DT{a_1} \DT{b} \DT{a_1} \DT{b} \\
&= \DT{a_1} \DT{b} \DT{a_1} \DT{b} \DT{a_1} \DT{b} \DT{a_1} \DT{b} \LDT{d} \cdot \DT{d} \DT{\gamma} \DT{a_3} \LDT{a_1} \LDT{a_1} \cdot \DT{a_1} \DT{b} \DT{a_1} \DT{b} \cdot \LDT{d_1} \LDT{d_2} \\
&= \DT{a_1} \DT{b} \DT{a_1} \DT{b} \DT{a_1} \DT{b} \DT{a_1} \DT{b} \DT{\gamma} \DT{a_3} \LDT{a_1} \underline{\DT{b} \DT{a_1} \DT{b}} \cdot \LDT{d_1} \LDT{d_2} \\
&= \DT{a_1} \DT{b} \DT{a_1} \DT{b} \DT{a_1} \DT{b} \DT{a_1} \DT{b} \DT{\gamma} \DT{a_3} \LDT{a_1} \DT{a_1} \DT{b} \DT{a_1} \cdot \LDT{d_1} \LDT{d_2} \\
&= \DT{a_1} \DT{b} \DT{a_1} \underline{\DT{b} \DT{a_1} \DT{b} \DT{a_1} \DT{b} \DT{\gamma}} \DT{a_3} \DT{b} \DT{a_1} \cdot \LDT{d_1} \LDT{d_2} \\
&= \DT{a_1} \DT{b} \DT{a_1} \DT{a_3} \DT{b} \DT{a_1} \DT{b} \DT{a_1} \DT{b} \DT{a_3} \DT{b} \DT{a_1} \cdot \LDT{d_1} \LDT{d_2} 
\intertext{(Here we used the simple observation that $\DT{b}\DT{a_1}\DT{b}\DT{a_1}\DT{b}(\gamma)=a_3$.)}
&= \DT{a_1} \DT{b} \DT{a_1} \DT{a_3} \DT{b} \DT{a_1} \DT{b} \DT{a_1} \underline{\DT{b} \DT{a_3} \DT{b}} \DT{a_1} \cdot \LDT{d_1} \LDT{d_2} \\
&= \DT{a_1} \DT{b} \DT{a_1} \DT{a_3} \DT{b} \DT{a_1} \DT{b} \DT{a_1} \DT{a_3} \DT{b} \DT{a_3} \DT{a_1} \cdot \LDT{d_1} \LDT{d_2} \\
&= \DT{a_1} \DT{b} \DT{a_1} \DT{a_3} \DT{b} \DT{a_1} \DT{b} \DT{a_3} \underline{\DT{a_1} \DT{b} \DT{a_1}} \DT{a_3} \cdot \LDT{d_1} \LDT{d_2} \\
&= \DT{a_1} \DT{b} \DT{a_1} \DT{a_3} \DT{b} \DT{a_1} \underline{\DT{b} \DT{a_3} \DT{b}} \DT{a_1} \DT{b} \DT{a_3} \cdot \LDT{d_1} \LDT{d_2} \\
&= \DT{a_1} \DT{b} \DT{a_1} \DT{a_3} \DT{b} \DT{a_1} \DT{a_3} \DT{b} \DT{a_3} \DT{a_1} \DT{b} \DT{a_3} \cdot \LDT{d_1} \LDT{d_2} \\
&= \DT{a_1} \DT{b} \cdot \underline{\DT{a_1} \DT{a_3} \LDT{d_1}} \cdot \DT{b} \DT{a_1} \DT{a_3} \DT{b} \cdot \underline{\DT{a_3} \DT{a_1} \LDT{d_2}} \cdot \DT{b} \DT{a_3}. 
\end{align*}
Now we substitute the two lantern relations
\begin{align*}
\DT{d_1} \DT{\gamma_1} \DT{a_2} = \DT{a_1} \DT{a_3} \DT{\delta_1} \DT{d_3}, \\
\DT{d_2} \DT{a_4} \DT{\gamma_2} = \DT{a_1} \DT{a_3} \DT{\delta_2} \DT{d_4},
\end{align*}
to get 
\begin{align*}
1 &= \DT{a_1} \DT{b} \cdot \DT{a_1} \DT{a_3} \LDT{d_1} \cdot \LDT{a_1} \LDT{a_3} \DT{d_1} \DT{\gamma_1} \DT{a_2} \LDT{\delta_1} \LDT{d_3} \cdot \DT{b} \DT{a_1} \DT{a_3} \DT{b} \\
& \qquad \cdot  \LDT{\delta_2} \LDT{d_4} \DT{a_4} \DT{\gamma_2} \DT{d_2} \LDT{a_1} \LDT{a_3} \cdot \DT{a_3} \DT{a_1} \LDT{d_2} \cdot \DT{b} \DT{a_3}. 
\end{align*}
Cancellation and commutativity yield 
\begin{align*}
\DT{\delta_1} \DT{\delta_2} &= \underline{\DT{a_1} \DT{b} \DT{\gamma_1}} \DT{a_2} \DT{b} \DT{a_1} \cdot \LDT{d_3} \LDT{d_4} \cdot \DT{a_3} \DT{b} \DT{a_4} \underline{\DT{\gamma_2} \DT{b} \DT{a_3}} \\
&= \DT{B_0} \DT{a_1} \DT{b} \DT{a_2} \DT{b} \DT{a_1} \cdot \LDT{d_3} \LDT{d_4} \cdot \DT{a_3} \DT{b} \DT{a_4} \DT{b} \DT{a_3} \DT{B_0^\prime}
\intertext{where $B_0 := \DT{a_1}\DT{b}(\gamma_1)$ and $B_0^\prime := \LDT{a_3} \LDT{b}(\gamma_2)$,}
&= \DT{B_0} \DT{a_1} \underline{\DT{b} \DT{a_2} \DT{b}} \DT{a_1} \cdot \LDT{d_3} \LDT{d_4} \cdot \DT{a_3} \underline{\DT{b} \DT{a_4} \DT{b}} \DT{a_3} \DT{B_0^\prime} \\
&= \DT{B_0} \underline{\DT{a_1} \DT{a_2} \DT{b}} \DT{a_2} \DT{a_1} \cdot \LDT{d_3} \LDT{d_4} \cdot \DT{a_3} \DT{a_4} \underline{\DT{b} \DT{a_4} \DT{a_3}} \DT{B_0^\prime} \\
&= \DT{B_0} \DT{B_1} \DT{a_1} \DT{a_2} \DT{a_2} \DT{a_1} \cdot \LDT{d_3} \LDT{d_4} \cdot \DT{a_3} \DT{a_4} \DT{a_4} \DT{a_3} \DT{B_1^\prime} \DT{B_0^\prime} 
\intertext{where $B_1 := \DT{a_1}\DT{a_2}(b)$ and $B_1^\prime := \LDT{a_3} \LDT{a_4}(b)$,}
&= \DT{B_0} \DT{B_1} \DT{a_1} \DT{a_2} \cdot \underline{\DT{a_1} \DT{a_4} \LDT{d_4}} \cdot \underline{\LDT{d_3} \DT{a_2} \DT{a_3}} \cdot \DT{a_4} \DT{a_3} \DT{B_1^\prime} \DT{B_0^\prime}. 
\end{align*}
We further substitute two more lantern relations, specifically
\begin{align*}
\DT{d_4} \DT{x_1} \DT{y_1} = \DT{a_1} \DT{a_4} \DT{\delta_4} \DT{d_6}, \\
\DT{d_3} \DT{x_2} \DT{y_2} = \DT{a_2} \DT{a_3} \DT{\delta_3} \DT{d_5},
\end{align*}
so that
\begin{align*}
\DT{\delta_1} \DT{\delta_2} &=
\DT{B_0} \DT{B_1} \DT{a_1} \DT{a_2} \cdot \DT{a_1} \DT{a_4} \LDT{d_4} \cdot \LDT{a_1} \LDT{a_4} \DT{d_4} \DT{x_1} \DT{y_1} \LDT{\delta_4} \LDT{d_6} \\
& \qquad \cdot \LDT{d_5} \LDT{\delta_3} \DT{x_2} \DT{y_2} \DT{d_3} \LDT{a_2} \LDT{a_3} \cdot \LDT{d_3} \DT{a_2} \DT{a_3} \cdot \DT{a_4} \DT{a_3} \DT{B_1^\prime} \DT{B_0^\prime}. 
\end{align*}
This now gives
\begin{align}
\notag
\DT{\delta_1} \DT{\delta_2} \DT{\delta_3} \DT{\delta_4} \DT{d_5} \DT{d_6} &=
\DT{B_0} \DT{B_1} \DT{a_1} \DT{a_2} \DT{x_1} \DT{y_1} \DT{x_2} \DT{y_2} \DT{a_4} \DT{a_3} \DT{B_1^\prime } \DT{B_0^\prime} 
\intertext{or}
\DT{\delta_1} \DT{\delta_2} \DT{\delta_3} \DT{\delta_4} \DT{d_5} \DT{d_6} &=
\DT{B_0} \DT{B_1} \cdot \DT{a_1} \DT{a_2} \DT{x_1} \DT{x_2} \cdot \DT{y_1} \DT{y_2} \DT{a_3} \DT{a_4} \cdot \DT{B_1^\prime} \DT{B_0^\prime}
\label{eq:6holed_torus_relation} 
\end{align}
in $\M(\Sigma_1^6)$.
Finally, we push the boundary components $d_5$ and $d_6$ as indicated in Figure~\ref{fig:6HoledTorus_Rearrangement} so we get the Dehn twist curves as depicted in Figure~\ref{fig:6HoledTorus}.
With this configuration of the curves in mind, the relation~\eqref{eq:6holed_torus_relation} is the $6$--holed torus relation we wanted.
Note that we have used one $2$--chain relation and five lantern relations to construct the relation~\eqref{eq:6holed_torus_relation}.

\begin{figure}[htbp]
	\centering
	\includegraphics[height=175pt]{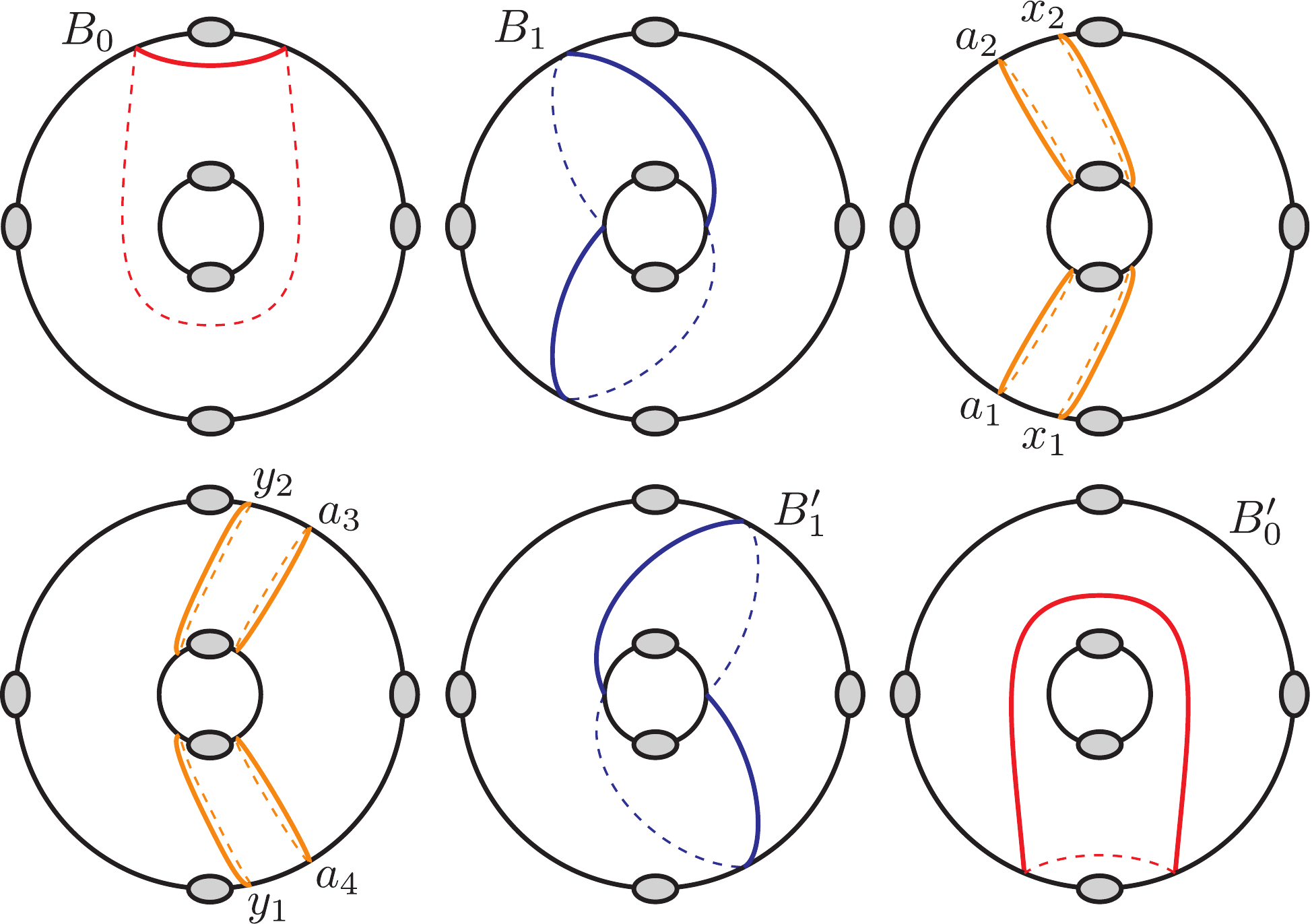}
	\caption{The curve configuration for our $6$--holed torus relation.} \label{fig:6HoledTorus}
\end{figure}

\smallskip
\noindent \textit{\underline{A lift of the $6$--holed torus relation}}:
Here we pause our construction to capture a lift of the relation~\eqref{eq:6holed_torus_relation}  in $\M(\Sigma_1^6)$ to $\M(\Sigma_{1,1}^6)$. 
This lift will involve a non-trivial point-pushing map, and thus it will not yield an honest section, but we will plug this information in later to describe a \emph{pseudosection} for the penultimate signature zero genus--$9$ Lefschetz fibration.

Consider the lower half of the $6$--holed torus $\Sigma_1^6$ and glue a $2$--holed disk to the boundary component $d_6$ so that the entire surface becomes $\Sigma_1^7$, as illustrated on the left in Figure~\ref{fig:6HoledTorus_MarkedPoint_Construction1}. Notice that the curves $a_1,x_1,d_6^\prime,d_6^{\prime\prime}$ cobound a $\Sigma_0^4$, so this is set up for a lantern substitution. We further add a disk with a marked point along the new boundary component $d_6^{\prime\prime}$, which results in the surface $\Sigma_{1,1}^6$.
Now the curves $d_6$, $x_1^\prime$, $a_1^\prime$ and $d_6^{\prime\prime}$, as shown on the right  in Figure~\ref{fig:6HoledTorus_MarkedPoint_Construction1}, make up the configuration of a lantern relation, though this relation is reduced to
\begin{align*}
\DT{d_6} \DT{x_1^\prime} \DT{a_1^\prime} = \DT{d_6^\prime} \DT{x_1} \DT{a_1}
\end{align*} 
for the Dehn twist along $d_6^{\prime\prime}$ is trivial in $\M(\Sigma_{1,1}^6)$.
We substitute this relation to the $6$--holed torus relation~\eqref{eq:6holed_torus_relation} as
\begin{align*}
\DT{\delta_1} \DT{\delta_2} \DT{\delta_3} \DT{\delta_4} \DT{d_5} \DT{d_6} &=
\DT{B_0} \DT{B_1} \cdot \DT{a_1} \DT{a_2} \DT{x_1} \DT{x_2} \cdot \LDT{a_1} \LDT{x_1} \LDT{d_6^\prime} \DT{d_6} \DT{x_1^\prime} \DT{a_1^\prime} \cdot  \DT{y_1} \DT{y_2} \DT{a_3} \DT{a_4} \cdot \DT{B_1^\prime} \DT{B_0^\prime}.
\end{align*}
By commutativity and cancellation, we obtain
\begin{align*}
\DT{\delta_1} \DT{\delta_2} \DT{\delta_3} \DT{\delta_4} \DT{d_5} \DT{d_6^\prime} &=
\DT{B_0} \DT{B_1} \cdot \DT{a_2} \DT{x_2} \cdot \DT{x_1^\prime} \DT{a_1^\prime} \cdot \DT{y_1} \DT{y_2} \DT{a_3} \DT{a_4} \cdot \DT{B_1^\prime} \DT{B_0^\prime}.
\end{align*}
We then  introduce canceling pairs $\DT{a_1} \LDT{a_1}$ and $\LDT{y_1^\prime} \DT{y_1^\prime}$ where the curve $y_1^\prime$ is as in Figure~\ref{fig:6HoledTorus_MarkedPoint_Construction2}.
\begin{align*}
\DT{\delta_1} \DT{\delta_2} \DT{\delta_3} \DT{\delta_4} \DT{d_5} \DT{d_6^\prime} &=
\DT{B_0} \DT{B_1} \cdot \DT{a_2} \DT{x_2} \cdot \DT{a_1} \LDT{a_1} \cdot \DT{x_1^\prime} \DT{a_1^\prime} \cdot \LDT{y_1^\prime} \DT{y_1^\prime} \cdot \DT{y_1} \DT{y_2} \DT{a_3} \DT{a_4} \cdot \DT{B_1^\prime} \DT{B_0^\prime} \\
&= \DT{B_0} \DT{B_1} \cdot \DT{a_2} \DT{x_2} \DT{a_1} \DT{x_1^\prime} \cdot \LDT{a_1} \DT{a_1^\prime} \cdot \LDT{y_1^\prime} \DT{y_1} \cdot \DT{y_1^\prime} \DT{y_2} \DT{a_3} \DT{a_4} \cdot \DT{B_1^\prime} \DT{B_0^\prime}.
\end{align*}
Observe that $a_1, a_1^\prime$ are parallel curves between which the marked point lies, so are $y_1,y_1^\prime$. Therefore, the factors $\LDT{a_1} \DT{a_1^\prime}$ and $\LDT{y_1^\prime} \DT{y_1}$ can be viewed as, respectively, the point-pushing maps $\mathcal{P}_{\vec{\beta}_1}$ and $\mathcal{P}_{\vec{\beta}_2}$ along the oriented loops $\vec{\beta}_1$ and $\vec{\beta}_2$ in Figure~\ref{fig:6HoledTorus_MarkedPoint_Construction3}.
The product $\mathcal{P}_{\vec{\beta}_1} \mathcal{P}_{\vec{\beta}_2}$, in turn, can be seen as the single point-pushing map $\mathcal{P}_{\vec{\alpha}_0}$ along $\vec{\alpha}_0$, which is homotopic to the concatenation of $\vec{\beta}_2$ and $\vec{\beta}_1$. The oriented loop $\vec{\alpha}_0$ is shown in Figure~\ref{fig:6HoledTorus_MarkedPoint_Construction3}. 
\begin{figure}[htbp]
	\centering
	\subfigure[Left: Adding a $2$--holed disk to the $6$--holed torus. Right: Capping off with a disk with a marked point, and getting a reduced lantern configuration. \label{fig:6HoledTorus_MarkedPoint_Construction1}]
	{\includegraphics[height=58pt]{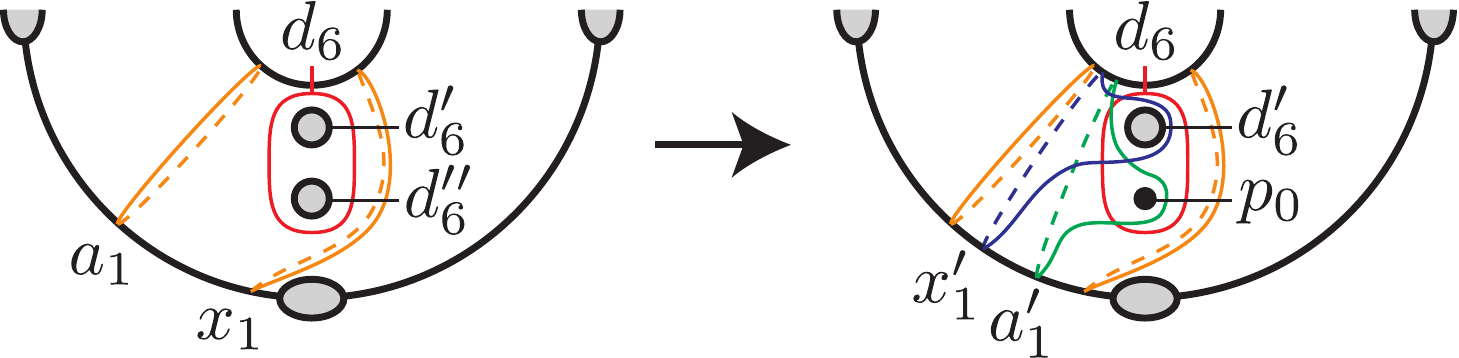}} 
	\hspace{10pt}
	\subfigure[Introducing canceling pairs about $a_1$ and $y_1^\prime$. \label{fig:6HoledTorus_MarkedPoint_Construction2}]
	{\includegraphics[height=52pt]{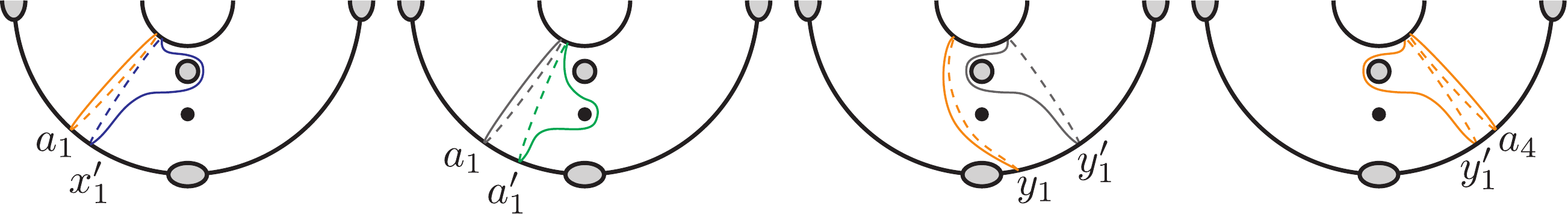}} 
	\hspace{10pt}
	\subfigure[The oriented loops $\vec{\beta}_1$, $\vec{\beta}_2$, and $\vec{\alpha}_0$ of the point-pushing maps. \label{fig:6HoledTorus_MarkedPoint_Construction3}]
	{\includegraphics[height=46pt]{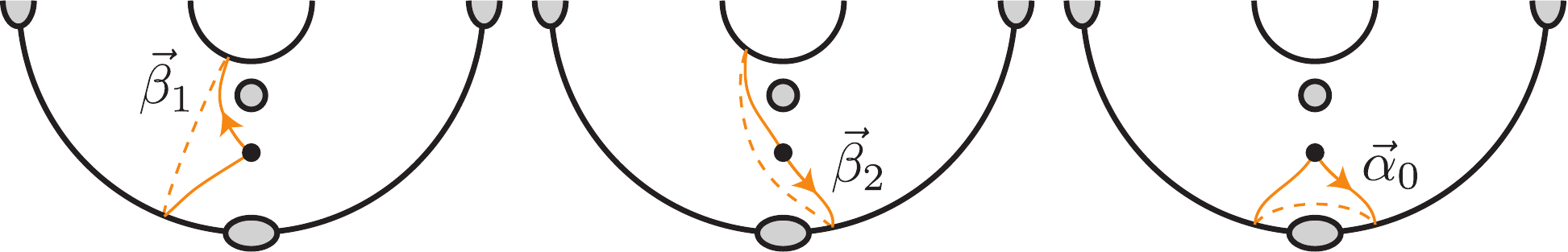}} 
	\caption{Finding a pseudosection of the $6$--holed torus relation.} 
	\label{fig:findingpseudosection}
\end{figure}

In summary, we have
\begin{align*}
\DT{\delta_1} \DT{\delta_2} \DT{\delta_3} \DT{\delta_4} \DT{d_5} \DT{d_6^\prime} 
&= \DT{B_0} \DT{B_1} \cdot \DT{a_2} \DT{x_2} \DT{a_1} \DT{x_1^\prime} \cdot \LDT{a_1} \DT{a_1^\prime} \cdot \LDT{y_1^\prime} \DT{y_1} \cdot \DT{y_1^\prime} \DT{y_2} \DT{a_3} \DT{a_4} \cdot \DT{B_1^\prime} \DT{B_0^\prime} \\
&= \DT{B_0} \DT{B_1} \cdot \DT{a_2} \DT{x_2} \DT{a_1} \DT{x_1^\prime} \cdot \mathcal{P}_{\vec{\beta}_1} \cdot \mathcal{P}_{\vec{\beta}_2} \cdot \DT{y_1^\prime} \DT{y_2} \DT{a_3} \DT{a_4} \cdot \DT{B_1^\prime} \DT{B_0^\prime} \\
&= \DT{B_0} \DT{B_1} \cdot \DT{a_2} \DT{x_2} \DT{a_1} \DT{x_1^\prime} \cdot \mathcal{P}_{\vec{\alpha}_0} \cdot \DT{y_1^\prime} \DT{y_2} \DT{a_3} \DT{a_4} \cdot \DT{B_1^\prime} \DT{B_0^\prime}.
\end{align*}
Finally, by dropping the prime symbol from $d_6^\prime$, $x_1^\prime$, $y_1^\prime$ (which should cause no confusion with the previous notation) we rewrite this relation in $\M(\Sigma_{1,1}^6)$ as:
\begin{align}
\DT{\delta_1} \DT{\delta_2} \DT{\delta_3} \DT{\delta_4} \DT{d_5} \DT{d_6} &=
\DT{B_0} \DT{B_1} \cdot \DT{a_1} \DT{a_2} \DT{x_1} \DT{x_2} \cdot \mathcal{P}_{\vec{\alpha}_0} \cdot \DT{y_1} \DT{y_2} \DT{a_3} \DT{a_4} \cdot \DT{B_1^\prime} \DT{B_0^\prime},
\label{eq:6holed_torus_relation_with_a_marked_point} 
\end{align}
where the curves are as illustrated in Figure~\ref{fig:6HoledTorus_MarkedPoint}.
This is our lift of the $6$--holed torus relation~\eqref{eq:6holed_torus_relation}.
\begin{figure}[htbp]
	\centering
	\includegraphics[height=175pt]{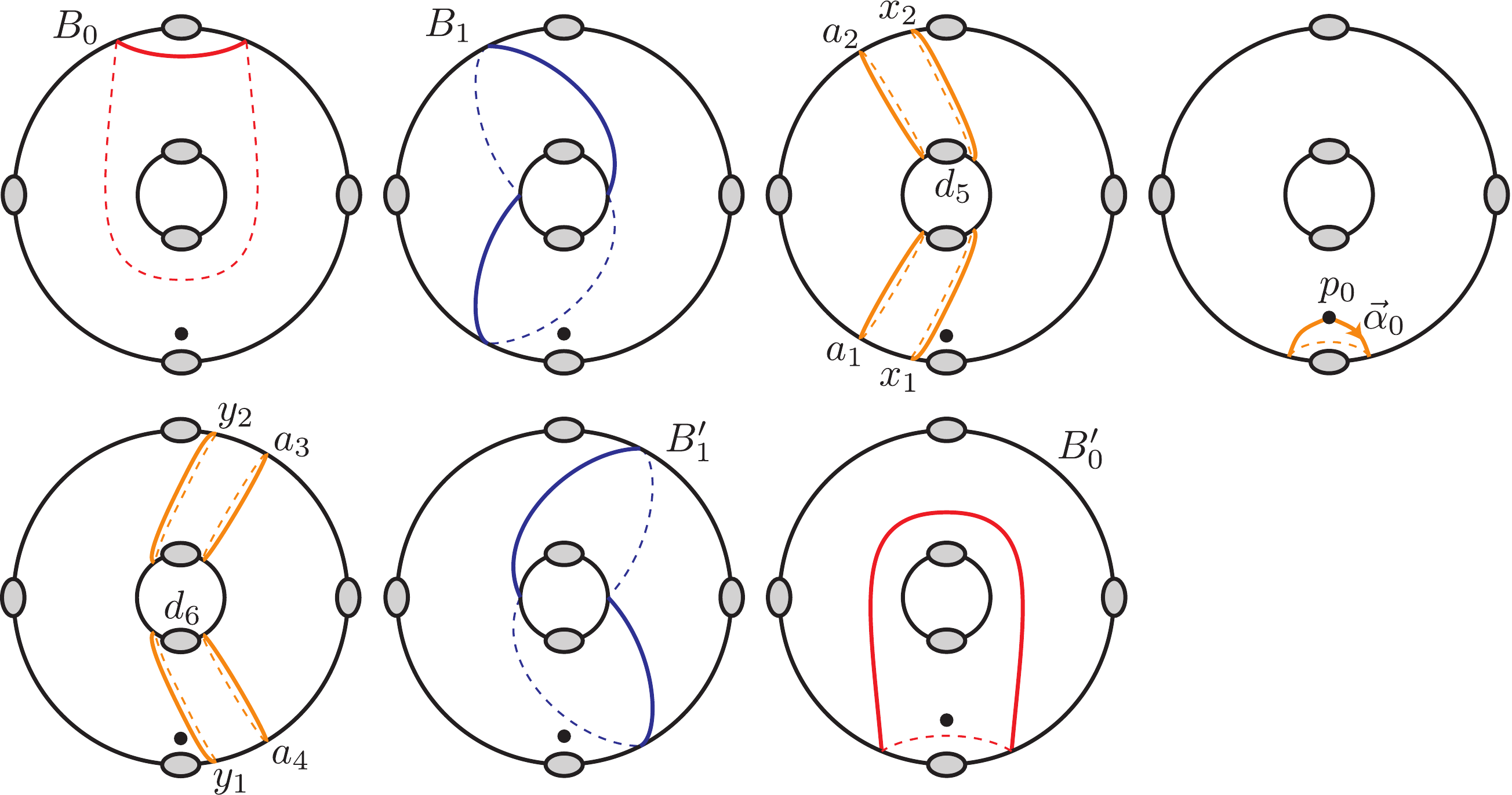}
	\caption{The $6$--holed torus relation with a marked point.} \label{fig:6HoledTorus_MarkedPoint}
\end{figure}
It happens that the oriented loop $\vec{\alpha}_0$ of the point-pushing map in the lift encircles one of the boundary components without intersecting any of the Dehn twist curves. This will make it easier to describe lifts in later steps of our construction to follow.

\smallskip
\noindent \textit{\underline{A genus--$2$ pencil}}:
We now construct the genus--$2$ Lefschetz pencil that will become one of our main building blocks. Take the boundary components $d_5$, $d_6$ of the $6$--holed torus $\Sigma_{1,1}^6$ in the previous step and connect them by a tube as shown in Figure~\ref{fig:Genus2_6HoledTorusEmbedding}. This gives a  $\Sigma_{2,1}^4$, a genus--$2$ surface with four boundary components and a marked point. The curves $d_5$ and $d_6$ become identical, so we denote both by $d$.
\begin{figure}[htbp]
	\centering
	\subfigure[The embedding of $\Sigma_{1,1}^6$ into $\Sigma_{2,1}^4$. \label{fig:Genus2_6HoledTorusEmbedding}]
	{\includegraphics[height=115pt]{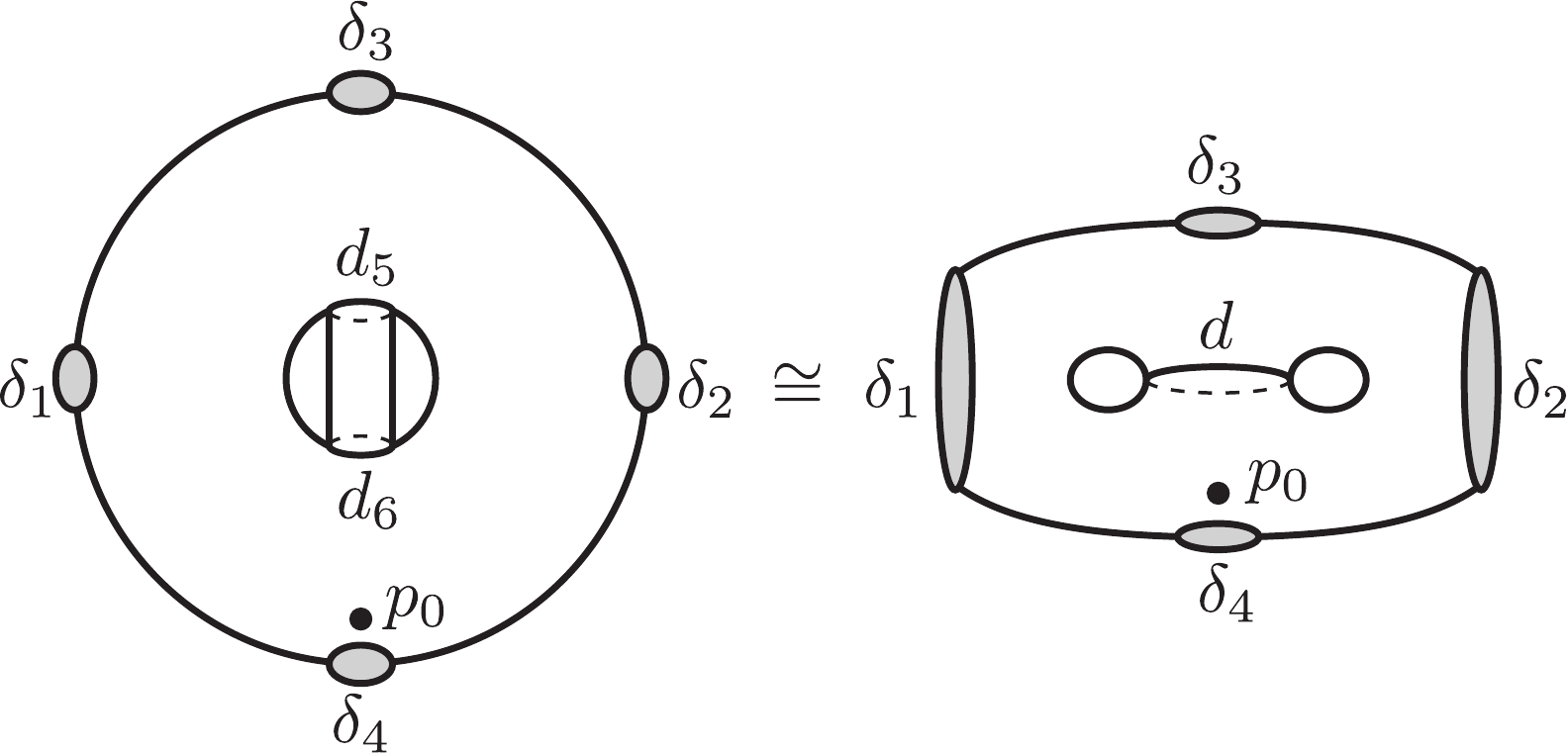}} 
	\hspace{10pt}
	\subfigure[Two configurations of lantern curves in $\Sigma_{2,1}^4$. \label{fig:Genus2_6HoledTorusEmbedding_Lantern}]
	{\includegraphics[height=70pt]{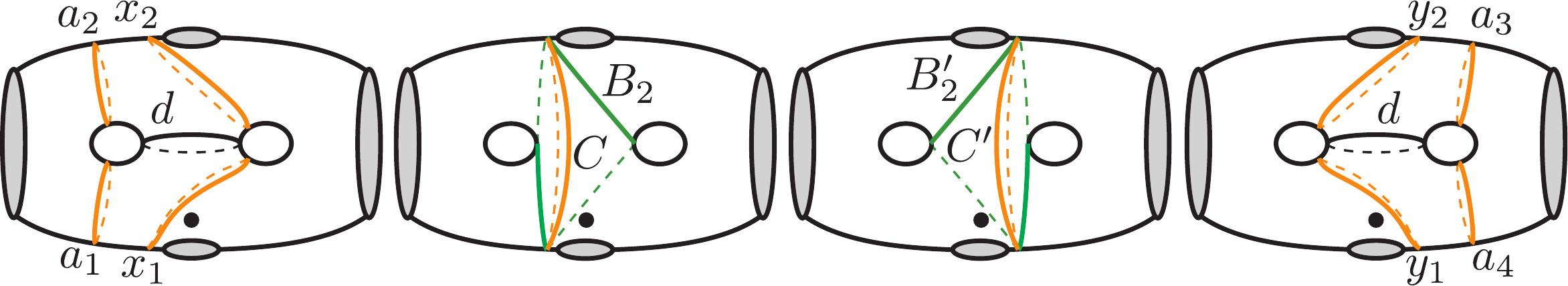}} 
	\hspace{10pt}
	\subfigure[Rearranging the boundary components $\delta_3$ and $\delta_4$.
	On the right, $\delta_4$ is on the front and $\delta_3$ is on the back of the surface. \label{fig:Genus2Push_MarkedPoint}]
	{\includegraphics[height=70pt]{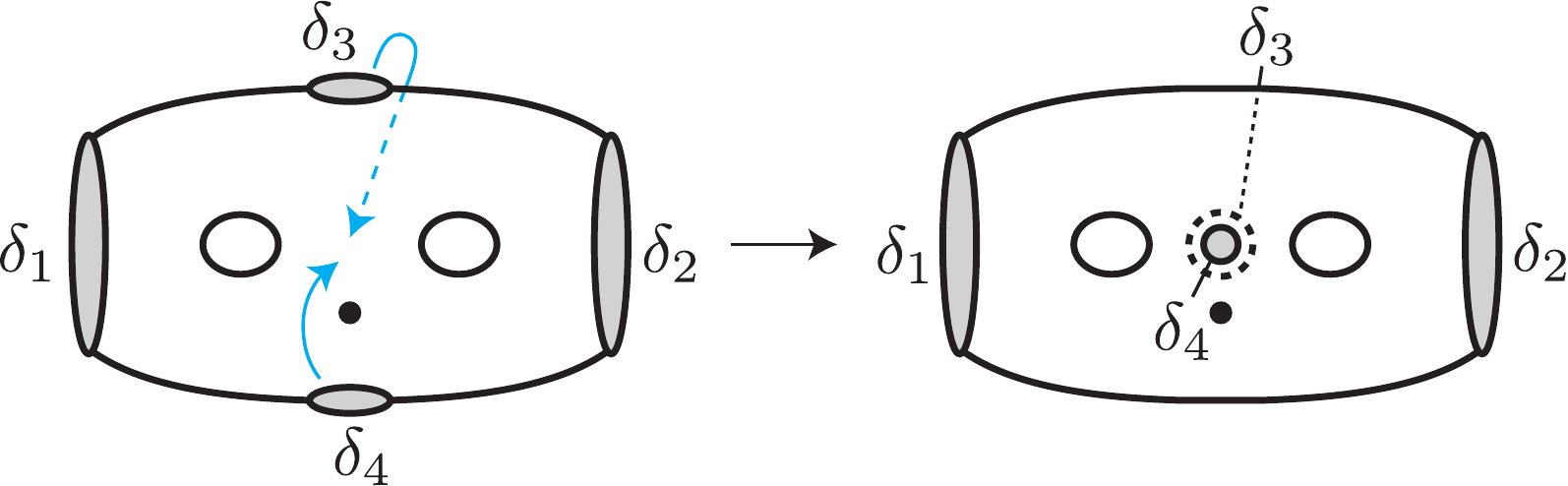}} 
	\caption{Construction of a genus--$2$ pencil.} 
	\label{fig:ConstructionOfGenus2LP}
\end{figure}

\begin{figure}[htbp]
	\centering
	\includegraphics[height=120pt]{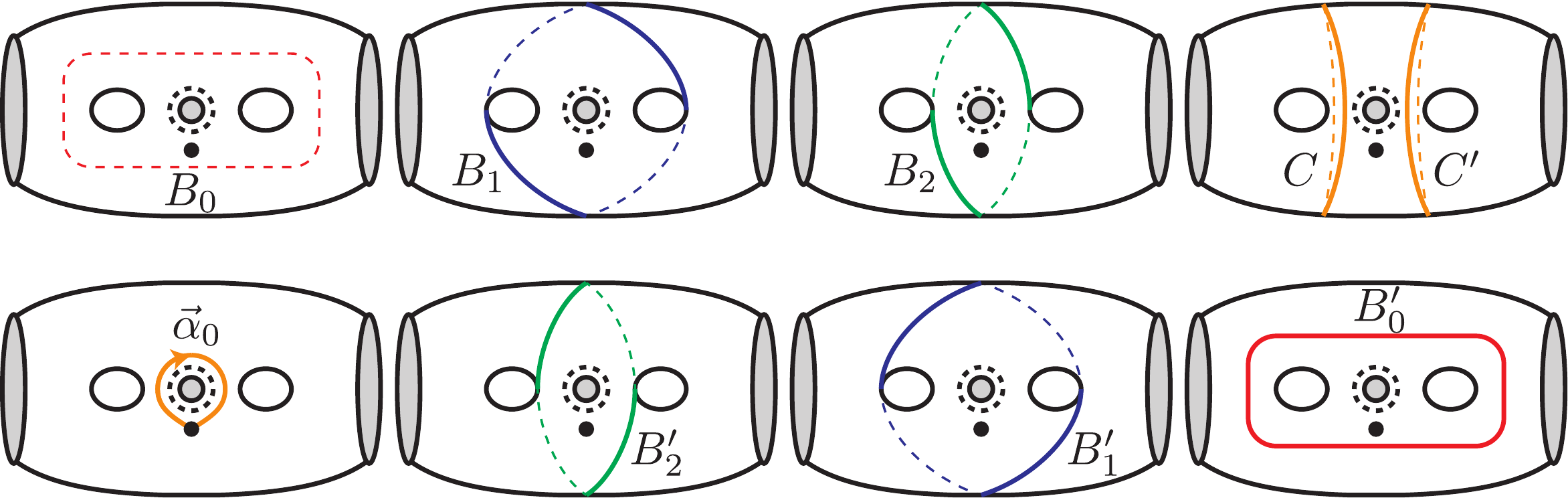}
	\caption{Vanishing cycles of a genus--$2$ pencil.
		Notice that $B_0,B_0^\prime$ is a bounding pair and the two separating curves $C$, $C^\prime$ are disjoint.}  \label{fig:MatsumotoSectionsIIA_MirrorSymmetrics_MarkedPoint}
\end{figure}

Now for the quadruples of curves $a_1,a_2,x_1,x_2,d$ and $y_1,y_2,a_3,a_4,d$ each cobounding a $\Sigma_0^4$ as shown in Figure~\ref{fig:Genus2_6HoledTorusEmbedding_Lantern}, we have the lantern relations
\begin{align*}
\DT{d} \DT{B_2} \DT{C} &= \DT{a_1} \DT{a_2} \DT{x_1} \DT{x_2}, \\
\DT{C^\prime} \DT{B_2^\prime} \DT{d} &= \DT{y_1} \DT{y_2} \DT{a_3} \DT{a_4},
\end{align*}
where the curves $B_2, C, C^\prime, B_2^\prime$ are as  in Figure~\ref{fig:Genus2_6HoledTorusEmbedding_Lantern}.
Combining the two with our $6$--holed torus relation~\eqref{eq:6holed_torus_relation_with_a_marked_point} yields
\begin{align*}
\DT{\delta_1} \DT{\delta_2} \DT{\delta_3} \DT{\delta_4} \DT{d} \DT{d} &=
\DT{B_0} \DT{B_1} \cdot \underline{\DT{a_1} \DT{a_2} \DT{x_1} \DT{x_2}} \cdot \mathcal{P}_{\vec{\alpha}_0} \cdot \underline{\DT{y_1} \DT{y_2} \DT{a_3} \DT{a_4}} \cdot \DT{B_1^\prime} \DT{B_0^\prime} \\
&= \DT{B_0} \DT{B_1} \cdot \DT{d} \DT{B_2} \DT{C} \cdot \mathcal{P}_{\vec{\alpha}_0} \cdot \DT{C^\prime} \DT{B_2^\prime} \DT{d} \cdot \DT{B_1^\prime} \DT{B_0^\prime}.
\end{align*}
By canceling the $\DT{d}$ factors we obtain the relation
\begin{align} 
\DT{\delta_1} \DT{\delta_2} \DT{\delta_3} \DT{\delta_4} &=
\DT{B_0} \DT{B_1} \DT{B_2} \DT{C} \cdot \mathcal{P}_{\vec{\alpha}_0} \cdot \DT{C^\prime} \DT{B_2^\prime} \DT{B_1^\prime} \DT{B_0^\prime}
\label{eq:MatsumotoLP_IIA_MarkedPoint} 
\end{align}
in $\M(\Sigma_{2,1}^4)$.
Moreover, after pushing the boundary components $\delta_3$ and $\delta_4$ as indicated in Figure~\ref{fig:Genus2Push_MarkedPoint}, we get a neater presentation of the Dehn twist curves as illustrated in Figure~\ref{fig:MatsumotoSectionsIIA_MirrorSymmetrics_MarkedPoint}.
Pairs of curves labeled with the same letters, but one decorated with a prime and one without it, are all symmetric under the obvious involution. 

The relation~\eqref{eq:MatsumotoLP_IIA_MarkedPoint} (after forgetting the marked point) is the monodromy factorization for our genus--$2$ Lefschetz pencil. Observe that in the monodromy factorization~\eqref{eq:MatsumotoLP_IIA_MarkedPoint}  the  pair $B_0, B_0^\prime$ cobounds two copies of $\Sigma_1^4$ with the boundary components, and we also have two disjoint separating curves $C$ and $C^\prime$, each of which cobounds a copy of $\Sigma_1^2$ with a boundary component. 

Let us compute the signature $\sigma$ of the total space of this genus--$2$ pencil. As we noted earlier, the $6$--holed torus relation~\eqref{eq:6holed_torus_relation} is derived by using a single $2$--chain relation and five lantern relations. We have then used two more lantern relations to obtain the relation~\eqref{eq:MatsumotoLP_IIA_MarkedPoint}. (Here we forget the marked point and the relations we employed for finding the point-pushing map.) Since the $2$--chain relation and the lantern relation contribute $-7$ and $+1$ to the signature, respectively, we compute the signature of the pencil as \, $\sigma = 1(-7) + 7(+1) =0$. 

We can moreover describe an explicit spin structure on this pencil (which we will make use of in Section~5.1.) Let $\{ \alpha_1, \beta_1, \alpha_2, \beta_2, \delta_1, \delta_2, \delta_3 \}$ be a  basis for $H_1(\Sigma_2^4;\mathbb{Z}_2)$, given by the same labeled curves in Figure~\ref{fig:MatsumotoSectionsIIA_MirrorSymmetric_H1BasisMod2}.
\begin{figure}[htbp]
	\centering
	\includegraphics[height=50pt]{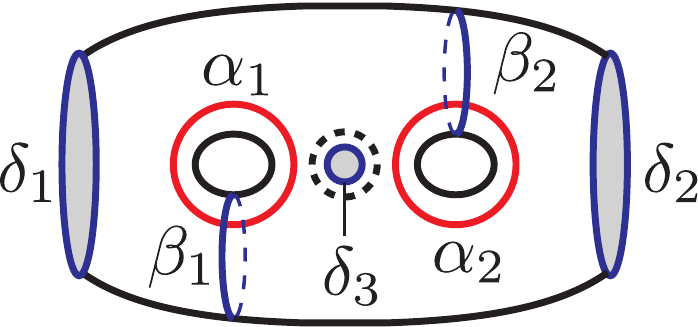}
	\caption{A basis for $H_1(\Sigma_2^4;\mathbb{Z}_2)$.}  \label{fig:MatsumotoSectionsIIA_MirrorSymmetric_H1BasisMod2}
\end{figure}

\noindent We can then compute the $\Z_2$--homology classes of the vanishing cycles as:
\begin{align*} 
B_0 &= \alpha_1 + \alpha_2 + \delta_1 + \delta_2 + \delta_3; \\
B_1 &= \alpha_1 + \beta_1 + \alpha_2 + \beta_2 + \delta_1 + \delta_2 + \delta_3; \\
B_2 &= \beta_1 + \beta_2 + \delta_1 + \delta_2 + \delta_3; \\
C &= \delta_1; \\
C^\prime &= \delta_2; \\
B_2^\prime &= \beta_1 + \beta_2 + \delta_3; \\
B_1^\prime &= \alpha_1 + \beta_1 + \alpha_2 + \beta_2 + \delta_3; \\
B_0^\prime &= \alpha_1 + \alpha_2 + \delta_3.
\end{align*}
If we now define a quadratic form $q$ on $H_1(\Sigma_2^4; \Z_2)$ such that
\begin{equation} \label{Genus2SpinStructure}
q(\alpha_i)=q(\beta_i)=q(\delta_j)=1 \text{ for all } i=1,2 \text{ and } j=1,\ldots, 4 \, ,
\end{equation}
then it maps each vanishing cycle to $1$, and by Theorem~\ref{SpinLP}, we get a spin structure on the total space of the pencil.

Equipping the pencil with a Gompf-Thurston symplectic form, we get a symplectic $4$--manifold. Then, observing that the fiber of the pencil violates the adjunction inequality, we conclude that the total space has to be a rational or a ruled surface.
As it quickly follows from the above calculation, the rank of the $\Z_2$--homology of this $4$--manifold is two, thus it should be the ruled surface $T^2 \times S^2$.

\begin{remark}
There are other explicit monodromy factorizations for genus--$2$ pencils with four base points on the ruled surface $T^2 \times S^2$, which were obtained by the second author in \cite{Hamada} as lifts of Matsumoto's well-known relation \cite{Matsumoto}. It is natural to ask whether the one we discovered here is Hurwitz equivalent to any of those, and we will show that  it is indeed the case in Appendix~A.
\end{remark}

\smallskip
\subsection{A signature zero genus--$3$ pencil}  \label{SecGenus3} \

We will next describe a genus--$3$ pencil, the total space of which has signature zero and is spin, whereas its monodromy has four (in fact five) pairs of Dehn twists along certain bounding pairs.

We breed two copies of the genus--$2$ pencil constructed in the previous subsection to obtain the desired genus--$3$ pencil. The curves $\delta_1, \delta_3, C_3, C_4$ on $\Sigma_{3,2}^4$ in Figure~\ref{fig:Genus3Breeding} cobound a subsurface diffeomorphic to $\Sigma_{2,1}^4$. The same goes for the curves $\delta_2, \delta_4, C_1, C_2$. Thus, we can embed two copies of the relation~\eqref{eq:MatsumotoLP_IIA_MarkedPoint} in $\M(\Sigma_{2,1}^4)$ into $\M(\Sigma_{3,1}^4)$ as
\begin{align*}
& \DT{a^\prime} \DT{x} \DT{b} \DT{C_1} \cdot \mathcal{P}_{\vec{\alpha}_1} \cdot \DT{C_2} \DT{d^\prime} \DT{w} \DT{a} = \DT{\delta_1} \DT{\delta_3} \DT{C_3} \DT{C_4}, \\
& \DT{c^\prime} \DT{z} \DT{d} \DT{C_3} \cdot \mathcal{P}_{\vec{\alpha}_2} \cdot \DT{C_4} \DT{b^\prime} \DT{y} \DT{c} = \DT{\delta_2} \DT{\delta_4} \DT{C_1} \DT{C_2},
\end{align*}
where the curves are as shown in Figure~\ref{fig:Genus3LP_4BoundingPairs}. 
Note that the second embedding is obtained by the first embedding followed by a rotation by $\pi$ about the horizontal line, so while the first marked point is near $\delta_1$ on the front of the surface, the second is placed near $\delta_4$ on the back. Dehn twist curves in the two relations above are the images, under the respective embeddings, of the curves of the relation \eqref{eq:MatsumotoLP_IIA_MarkedPoint} in the same order.

\begin{figure}[htbp]
	\centering
	\includegraphics[height=110pt]{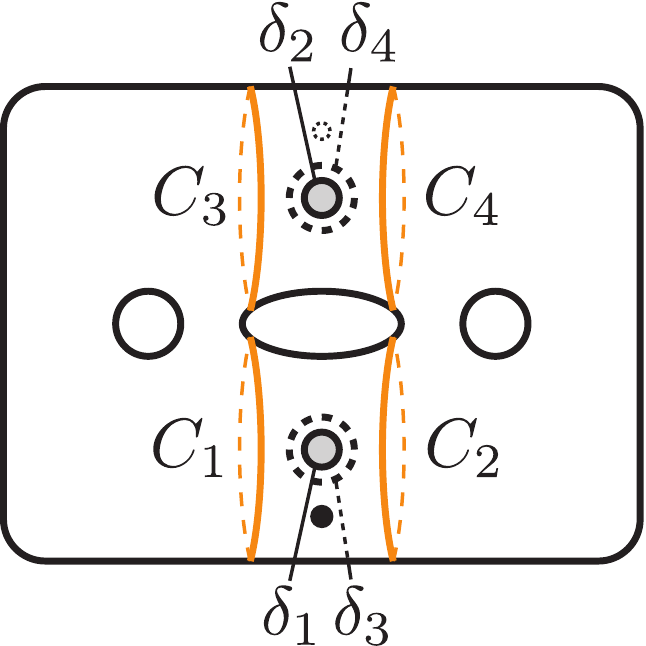}
	\caption{The curves $C_i, \delta_i$ on $\Sigma_{3,2}^4$.} 
	\label{fig:Genus3Breeding}
\end{figure}
\begin{figure}[htbp]
	\centering
	\includegraphics[height=160pt]{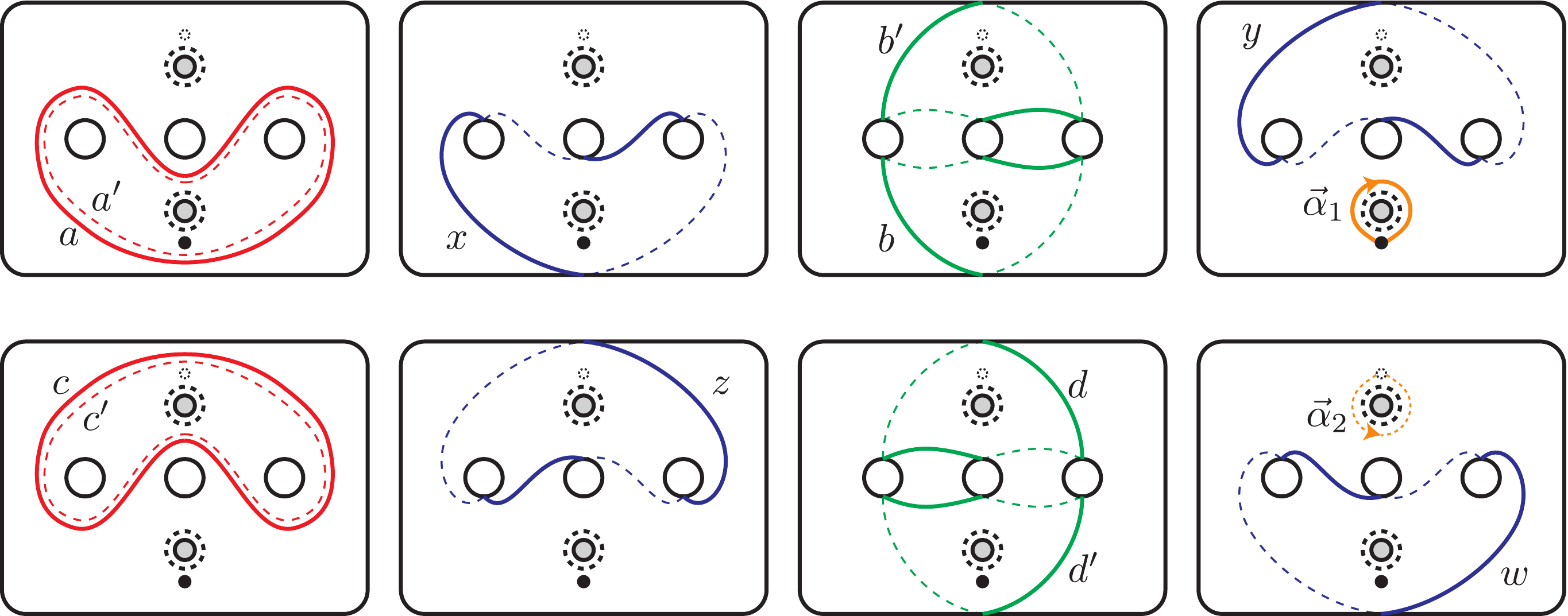}
	\caption{Vanishing cycles of the genus--$3$ pencil.} 
	\label{fig:Genus3LP_4BoundingPairs}
\end{figure}

\enlargethispage{0.1in}
By cyclic permutations, we can rewrite the relations as
\begin{align*}
& \DT{d^\prime} \DT{w} \DT{a} \DT{a^\prime} \DT{x} \DT{b} \cdot \mathcal{P}_{\vec{\alpha}_1} = \DT{\delta_1} \DT{\delta_2} \DT{C_3} \DT{C_4} \LDT{C_2} \LDT{C_1}, \\
& \DT{b^\prime} \DT{y} \DT{c} \DT{c^\prime} \DT{z} \DT{d} \cdot \mathcal{P}_{\vec{\alpha}_2} = \DT{\delta_3} \DT{\delta_4} \DT{C_1} \DT{C_2} \LDT{C_4} \LDT{C_3}.
\end{align*}
Combining them we get
\begin{align*}
\DT{d^\prime} \DT{w} \DT{a} \DT{a^\prime} \DT{x} \DT{b} \cdot \mathcal{P}_{\vec{\alpha}_1} \cdot \DT{b^\prime} \DT{y} \DT{c} \DT{c^\prime} \DT{z} \DT{d} \cdot \mathcal{P}_{\vec{\alpha}_2} = \DT{\delta_1} \DT{\delta_2} \DT{\delta_3} \DT{\delta_4},
\end{align*}
which, using cyclic permutation and commutativity, can be expressed as
\begin{align} \label{eq:Genus3LP_4BoundingPairs}
\DT{a} \DT{a^\prime} \DT{x} \DT{b} \DT{b^\prime} \DT{y} \DT{c} \DT{c^\prime} \DT{z} \DT{d} \DT{d^\prime} \DT{w} \cdot \mathcal{P}_{\vec{\alpha}_1} \mathcal{P}_{\vec{\alpha}_2} = \DT{\delta_1} \DT{\delta_2} \DT{\delta_3} \DT{\delta_4}.
\end{align}
This is a relation in $\M(\Sigma_{3,2}^4)$, which, by forgetting the marked points, reduces to the positive factorization of a genus--$3$ pencil.

In the relation~\eqref{eq:Genus3LP_4BoundingPairs}, we have the two bounding pairs $a,a^\prime$ and $c,c^\prime$, which are inherited from the genus--$2$ pencil. There are two more bounding pairs $b,b^\prime$ and $d,d^\prime$; see Figure~\ref{fig:Genus3LP_4BoundingPairs}. (The existence of these four bounding pairs is the most essential feature for our constructing of the signature zero Lefschetz fibration in the next subsection.) There is in fact a fifth  bounding pair we can get after Hurwitz moves: We can move the factor $\DT{x}$ over the subword $\DT{b}\DT{b^\prime} \DT{y}\DT{c}\DT{c^\prime}$ (i.e. we conjugate this subword with $\DT{x}$) to get the subword $\DT{x}\DT{z}$ in the monodromy, where the pair $x,z$ too splits the surface into two genus--$1$ subsurfaces. (Note that the pair $y,w$ is yet another bounding pair, but this property is destroyed when we bring $\DT{x}$ and $\DT{z}$ together above.)

The total space $Y$ of our genus--$3$ pencil  has signature zero since the relation~\eqref{eq:Genus3LP_4BoundingPairs} is the combination of two copies of the relation~\eqref{eq:MatsumotoLP_IIA_MarkedPoint} with signature zero. The Euler characteristic is easily calculated as $\eu(Y)=4 -4 g +l -b=4-4\cdot 3+12 -4=0$ (where $g$ is the fiber genus, $l$ is the number of Lefschetz critical points and $b$ is the number of base points). 

To pin down the topology of $Y$ we will calculate its first homology. Let us use the  symplectic basis $\{ \alpha_1, \beta_1, \alpha_2, \beta_2, \alpha_3, \beta_3 \}$ for $H_1(\Sigma_4;\mathbb{Z})$ as in Figure~\ref{fig:Genus3_H1Basis}.
\begin{figure}[htbp]
	\centering
	\includegraphics[height=70pt]{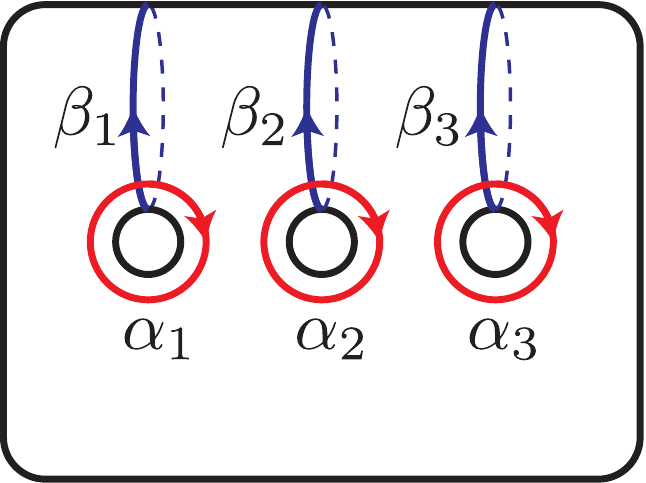}
	\caption{The basis for $H_1(\Sigma_4;\mathbb{Z})$.} 
	\label{fig:Genus3_H1Basis}
\end{figure}
The homology classes of the vanishing cycles in this basis are:
$a = a^\prime = c = c^\prime = \alpha_1 + \alpha_3$, \,$x = z = \alpha_1 + \alpha_3 -\beta_1 +\beta_2 -\beta_3$, \,$b = b^\prime = d = d^\prime = \beta_1 -\beta_2 +\beta_3$\, and $y = w = \alpha_1 + \alpha_3 + \beta_1 -\beta_2 + \beta_3$.  In turn, the relations in $H_1(Y; \Z)$ we obtain by setting the vanishing cycles equal to zero are: $\alpha_3 = -\alpha_1$ \,and $\beta_3 = -\beta_1 + \beta_2$. Therefore $H_1(Y;\mathbb{Z})$ is freely generated by $\alpha_1, \beta_1, \alpha_2, \beta_2$, and $H_1(Y;\mathbb{Z})=\mathbb{Z}^4$.

Since $\sigma(Y)=\eu(Y)$ and $b_1(Y)=4$, $Y$ cannot be a rational or ruled surface. Since this genus--$3$ pencil has $4$ base points, by  \cite[Theorem 1.2]{BaykurHayano}, $Y$ is a symplectic Calabi-Yau $4$--manifold, rational homology equivalent to $T^4$. (In fact one can show that $\pi_1(Y) \cong \Z^4$ and thus $Y$ is homeomorphic to $T^4$.) In particular we see that $Y$ is spin.
Alternatively, one can directly construct a quadratic form on $\Sigma_3^4$ that satisfies the conditions in Theorem~\ref{SpinLP} to show $Y$ is spin.

\begin{remark}
There are other explicit monodromy factorizations for genus--$3$ pencils with four base points on symplectic Calabi-Yau $4$--manifolds homeomorphic to $T^4$, obtained by the first author in \cite{BaykurGenus3} and the second author and Hayano in  \cite{HamadaHayano}. It is once again natural to ask whether the one we discovered here is Hurwitz equivalent to any of those, and we will also confirm that this is the case in Appendix~A, which in particular will imply that this $4$--manifold $Y$ is \emph{diffeomorphic} to the standard  $T^4$.
\end{remark}

\smallskip
\subsection{A signature zero genus--$9$ Lefschetz fibration over the $2$--sphere}  \

We will now describe our genus--$9$ fibration $(X,f)$, the total space of which has signature zero. Further properties of $X$, such as it being spin, will be explored in Section~\ref{topofkeyLF}. Expectedly, its monodromy factorization will not contain any Dehn twists along separating curves (which would destroy the spin property), but  it contains Dehn twists along a quadruple of bounding curves which split the $\Sigma_9$ into two copies of $\Sigma_3^4$. This extra property will  be essential for our arguments in Section~5.

We will first explain how we obtain our signature zero Lefschetz fibration, without specifying all the embeddings involved in this construction, and thus without an explicit description of all the curves in the final monodromy factorization. We will also omit the marked point in this exposition. This first round of information suffices to justify the existence of a signature zero genus--$9$ Lefschetz fibration. Afterwards, we will describe the embeddings explicitly, and also include the marked point for the pseudosection calculation. The latter will take the chunk of this subsection.

\smallskip
\noindent \textit{\underline{Schematic construction}}:
For simplicity, we omit all the marked points in any of the figures we will refer to here. We will breed six copies of our genus--$3$ pencil with monodromy factorization~\eqref{eq:Genus3LP_4BoundingPairs}.
As seen on the right side of Figure~\ref{fig:Genus9_Genus3Embedding1and2}, the curves $\Delta_1, \Delta_2, \Delta_3, \Delta_4$ on $\Sigma_9$ bound two copies of $\Sigma_3^4$, which constitute the top and the bottom halves of $\Sigma_9$.
We embed two copies of the relation~\eqref{eq:Genus3LP_4BoundingPairs} in $\M(\Sigma_3^4)$ into $\M(\Sigma_9)$ via the embeddings $\Phi_1$ and $\Phi_2$ as explained  in Figure~\ref{fig:Genus9_Genus3Embedding1and2} (and its caption), such that the first one is supported on the top and the second one on the bottom half:
\begin{align*} 
	&\DT{a_1} \DT{a_1^\prime} \DT{x_1} \DT{b_1} \DT{b_1^\prime} \DT{y_1} \DT{c_1} \DT{c_1^\prime} \DT{z_1} \DT{d_1} \DT{d_1^\prime} \DT{w_1} = \DT{\Delta_1} \DT{\Delta_2} \DT{\Delta_3} \DT{\Delta_4}, \\
	&\DT{a_2} \DT{a_2^\prime} \DT{x_2} \DT{b_2} \DT{b_2^\prime} \DT{y_2} \DT{c_2} \DT{c_2^\prime} \DT{z_2} \DT{d_2} \DT{d_2^\prime} \DT{w_2} = \DT{\Delta_1} \DT{\Delta_2} \DT{\Delta_3} \DT{\Delta_4}.
\end{align*}
These curves are explicitly given in Figures~\ref{fig:Genus9_BoundingQuadruples} and~\ref{fig:SignatureZeroLF_VanishingCycles}.
As usual, the Dehn twist curves in the two relations above are the images, under the respective embeddings $\Phi_1$ and $\Phi_2$, of the curves of the relation \eqref{eq:Genus3LP_4BoundingPairs} in the same order -- and the same goes for our four other embeddings to follow. 
\begin{figure}[htbp]
	\begin{tabular}{c}
		\begin{minipage}[t]{1\hsize}
			\centering
			\includegraphics[height=130pt]{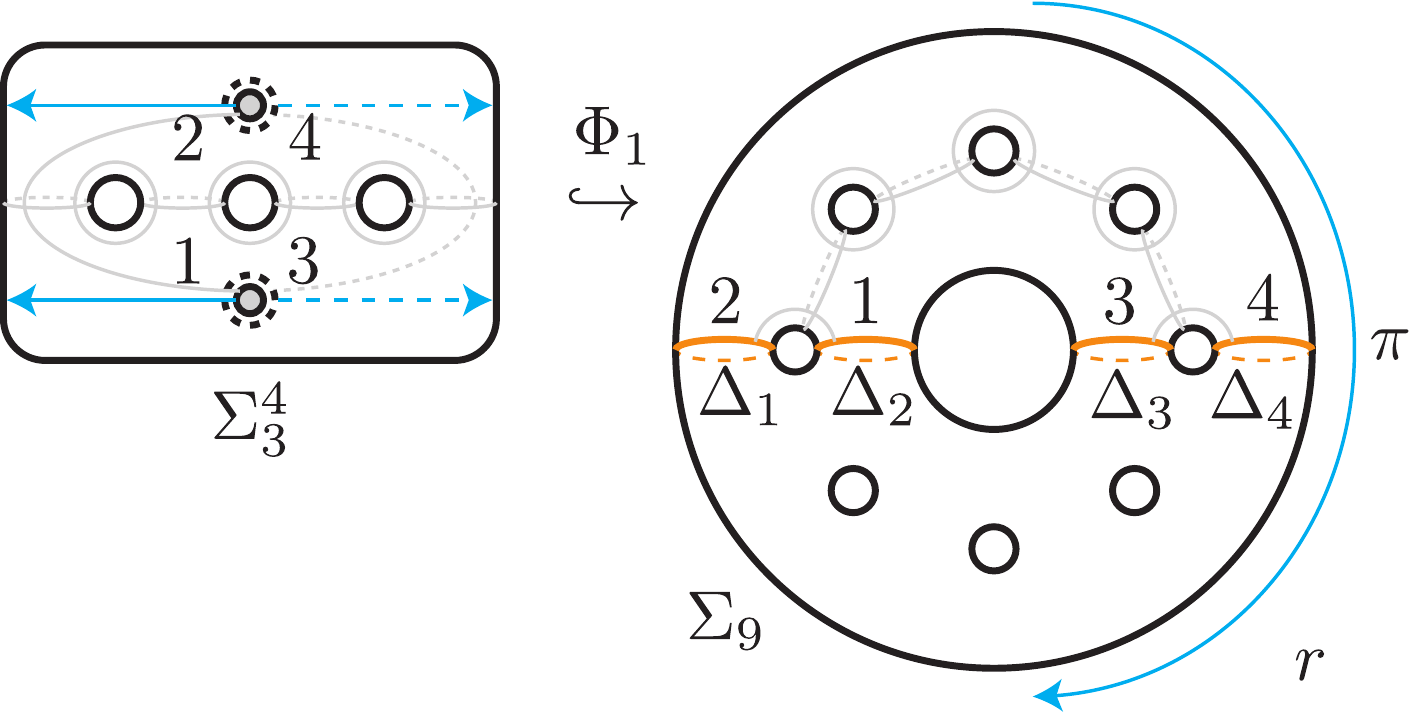}
			\caption{The embedding $\Phi_1$ of $\Sigma_3^4$ into $\Sigma_{9}$. First, we slide the four boundary components as indicated by the blue arrows. Then embed the surface into the top half of $\Sigma_9$. The labels $1,2,3,4$ indicate how the boundary components are matched with the  quadruple $\Delta_1, \Delta_2, \Delta_3, \Delta_4$. Also, the gray curves, which fill $\Sigma_3^4$, and their embedded images are drawn so that the embedding $\Phi_1$ is uniquely defined. The second embedding $\Phi_2$ is given by $\Phi_1$ followed by the $\pi$--rotation $r$, that is $\Phi_2 := r \circ \Phi_1$.} 
			\label{fig:Genus9_Genus3Embedding1and2}
		\end{minipage} \\ 
		\vspace{1\baselineskip} \\
		\begin{minipage}[t]{1\hsize}
			\centering
			\includegraphics[height=100pt]{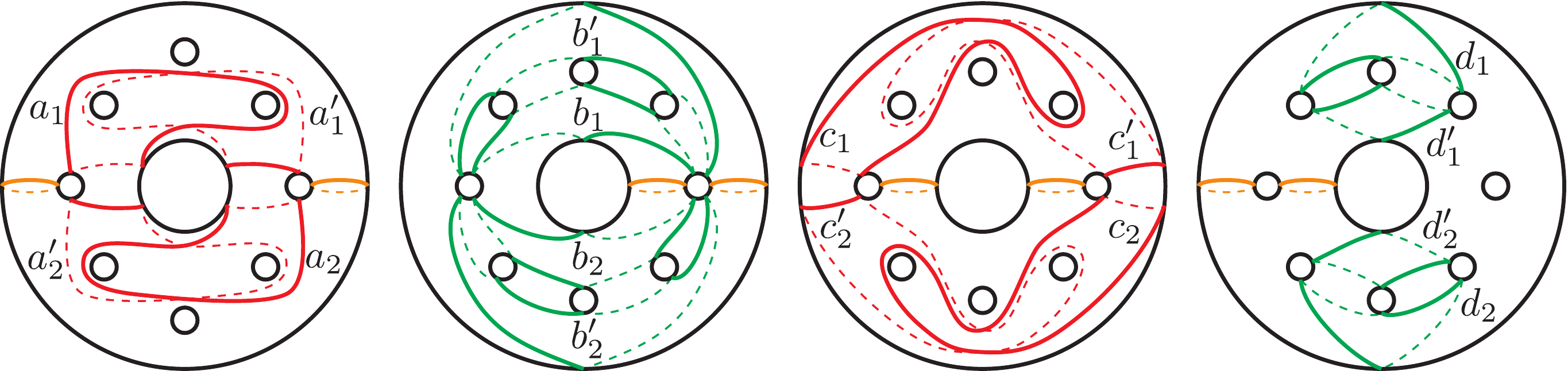}
			\caption{The four bounding quadruples.} 
			\label{fig:Genus9_BoundingQuadruples}
		\end{minipage}
	\end{tabular}
\end{figure}

Since these two relations have disjoint supports, we can combine them to get 
\begin{align*} 
	\DT{a_1} \DT{a_1^\prime} \DT{a_2} \DT{a_2^\prime} \DT{x_1} \DT{x_2} \DT{b_1} \DT{b_1^\prime} \DT{b_2} \DT{b_2^\prime} \DT{y_1} \DT{y_2} \DT{c_1} \DT{c_1^\prime} \DT{c_2} \DT{c_2^\prime} \DT{z_1} \DT{z_2} \DT{d_1} \DT{d_1^\prime} \DT{d_2} \DT{d_2^\prime} \DT{w_1} \DT{w_2}
	= \DT{\Delta_1}^2 \DT{\Delta_2}^2 \DT{\Delta_3}^2 \DT{\Delta_4}^2, 
\end{align*}
in $\M(\Sigma_{9})$, which we can rewrite as
\begin{align} \label{eq:Signature0LF_BeforeSubstitutions}
	& \underline{\DT{a_1} \DT{a_1^\prime} \DT{a_2} \DT{a_2^\prime} \LDT{\Delta_1} \LDT{\Delta_4}} \cdot \DT{x_1} \DT{x_2} \cdot 
	\underline{\DT{b_1} \DT{b_1^\prime} \DT{b_2} \DT{b_2^\prime} \LDT{\Delta_3} \LDT{\Delta_4}} \cdot \DT{y_1} \DT{y_2} \\
	& \cdot 
	\underline{\DT{c_1} \DT{c_1^\prime} \DT{c_2} \DT{c_2^\prime} \LDT{\Delta_2} \LDT{\Delta_3}} \cdot \DT{z_1} \DT{z_2} \cdot 
	\underline{\DT{d_1} \DT{d_1^\prime} \DT{d_2} \DT{d_2^\prime} \LDT{\Delta_1} \LDT{\Delta_2}}
	\cdot \DT{w_1} \DT{w_2}
= 1. \notag
\end{align}
The most essential feature of this relation is that it contains Dehn twists along four bounding quadruples, as singled out in Figure~\ref{fig:Genus9_BoundingQuadruples}, such that the subsurfaces they cobound are diffeomorphic to $\Sigma_3^4$ with pairs of $\Delta_i$ as genus--$1$ bounding pairs in each $\Sigma_3^4$. 

Now, let us examine the configuration of the curves 
$a_1,a_1^\prime,a_2,a_2^\prime, \Delta_1, \Delta_4$.
Let $S_3$ be the subsurface bounded by the quadruple $a_1,a_1^\prime,a_2,a_2^\prime$ that contains $\Delta_1$ and $\Delta_4$. Then it is easy to observe that $S_3$ is diffeomorphic to $\Sigma_3^4$ and the pair $\Delta_1, \Delta_4$ splits $S_3$ into two genus--$1$ subsurfaces with the boundary components $a_1, a_1^\prime, \Delta_1, \Delta_4$ and $a_2, a_2^\prime, \Delta_1, \Delta_4$, respectively.
Turning to Figure~\ref{fig:Genus3LP_4BoundingPairs} we note that the pair $d,d^\prime$ also splits $\Sigma_3^4$ into two genus--$1$ subsurfaces with the boundary components $\delta_1, \delta_2, d, d^\prime$ and $\delta_3, \delta_4, d, d^\prime$, respectively. This means that the tuple of curves $(a_1,a_1^\prime,a_2,a_2^\prime, \Delta_1, \Delta_4$) in $S_3$ is topologically equivalent to $(\delta_1, \delta_2, \delta_3, \delta_4, d,d^\prime)$ in $\Sigma_3^4$.
Therefore there exists an embedding $\Phi_3$ of $\Sigma_3^4$ into $\Sigma_9$ such that 
\begin{itemize}
	\item $\Phi_3$ maps $(\delta_1,\delta_2,\delta_3,\delta_4,d,d^\prime)$ to $(a_1,a_1^\prime,a_2,a_2^\prime, \Delta_1, \Delta_4)$.
\end{itemize}
Identical arguments guarantee that there exist embeddings $\Phi_4, \Phi_5, \Phi_6$ of $\Sigma_3^4$ into $\Sigma_9$ such that 
\begin{itemize}
	\item $\Phi_4$ maps $(\delta_1,\delta_2,\delta_3,\delta_4,d,d^\prime)$ to $(b_1,b_1^\prime,b_2,b_2^\prime, \Delta_3, \Delta_4)$,
	\item $\Phi_5$ maps $(\delta_1,\delta_2,\delta_3,\delta_4,d,d^\prime)$ to $(c_1,c_1^\prime,c_2,c_2^\prime, \Delta_2, \Delta_3)$,
	\item $\Phi_6$ maps $(\delta_1,\delta_2,\delta_3,\delta_4,d,d^\prime)$ to $(d_1,d_1^\prime,d_2,d_2^\prime, \Delta_1, \Delta_2)$.
\end{itemize}
Now the genus--$3$ relation~\eqref{eq:Genus3LP_4BoundingPairs} (without marked points) can be rearranged as
\begin{align*}
\DT{w} \DT{a} \DT{a^\prime} \DT{x} \DT{b} \DT{b^\prime} \DT{y} \DT{c} \DT{c^\prime} \DT{z} = \DT{\delta_1} \DT{\delta_2} \DT{\delta_3} \DT{\delta_4} \LDT{d} \LDT{d^\prime}.
\end{align*}
So if we set $a_i := \Phi_i(a), a_i^\prime := \Phi_i(a^\prime), x_i := \Phi_i(x)$, and so on, then copies of this relation are embedded into $\Gamma_9$ as
\begin{align*}
& \DT{w_3} \DT{a_3} \DT{a_3^\prime} \DT{x_3} \DT{b_3} \DT{b_3^\prime} \DT{y_3} \DT{c_3} \DT{c_3^\prime} \DT{z_3} = \DT{a_1} \DT{a_1^\prime} \DT{a_2} \DT{a_2^\prime} \LDT{\Delta_1} \LDT{\Delta_4}, \\
& \DT{w_4} \DT{a_4} \DT{a_4^\prime} \DT{x_4} \DT{b_4} \DT{b_4^\prime} \DT{y_4} \DT{c_4} \DT{c_4^\prime} \DT{z_4} = \DT{b_1} \DT{b_1^\prime} \DT{b_2} \DT{b_2^\prime} \LDT{\Delta_3} \LDT{\Delta_4}, \\
& \DT{w_5} \DT{a_5} \DT{a_5^\prime} \DT{x_5} \DT{b_5} \DT{b_5^\prime} \DT{y_5} \DT{c_5} \DT{c_5^\prime} \DT{z_5} = \DT{c_1} \DT{c_1^\prime} \DT{c_2} \DT{c_2^\prime} \LDT{\Delta_2} \LDT{\Delta_3}, \\
& \DT{w_6} \DT{a_6} \DT{a_6^\prime} \DT{x_6} \DT{b_6} \DT{b_6^\prime} \DT{y_6} \DT{c_6} \DT{c_6^\prime} \DT{z_6} = \DT{d_1} \DT{d_1^\prime} \DT{d_2} \DT{d_2^\prime} \LDT{\Delta_1} \LDT{\Delta_2}.
\end{align*}
Then we can substitute those four relations into the underlined parts of the relation~\eqref{eq:Signature0LF_BeforeSubstitutions}  to obtain:
\begin{align} \label{eq:Signature0LF_schematic}
&\DT{w_3} \DT{a_3} \DT{a_3^\prime} \DT{x_3} \DT{b_3} \DT{b_3^\prime} \DT{y_3} \DT{c_3} \DT{c_3^\prime} \DT{z_3} \DT{x_1} \DT{x_2} \cdot
\DT{w_4} \DT{a_4} \DT{a_4^\prime} \DT{x_4} \DT{b_4} \DT{b_4^\prime} \DT{y_4} \DT{c_4} \DT{c_4^\prime} \DT{z_4} \DT{y_1} \DT{y_2} \\ \notag
&\cdot
\DT{w_5} \DT{a_5} \DT{a_5^\prime} \DT{x_5} \DT{b_5} \DT{b_5^\prime} \DT{y_5} \DT{c_5} \DT{c_5^\prime} \DT{z_5} \DT{z_1} \DT{z_2} \cdot
\DT{w_6} \DT{a_6} \DT{a_6^\prime} \DT{x_6} \DT{b_6} \DT{b_6^\prime} \DT{y_6} \DT{c_6} \DT{c_6^\prime} \DT{z_6} \DT{w_1} \DT{w_2}
=1.
\end{align}
This is a positive factorization of the identity in $\M(\Sigma_9)$, so it provides a genus--$9$ Lefschetz fibration $f\colon X \to S^2$. Since we only used copies of the relation~\eqref{eq:Genus3LP_4BoundingPairs}, which has signature zero,  the total space $X$ has signature zero.

\smallskip
\noindent \textit{\underline{Explicit construction}}:
We will now describe explicit embeddings $\Phi_3, \Phi_4, \Phi_5, \Phi_6$ of $\Sigma_3^4$ (with or without a marked point) into $\Sigma_{9,1}$ to obtain an explicit monodromy factorization of the genus--$9$ fibration, as well as to pin-point a pseudosection. 

\begin{figure}[htbp]
	\centering
	\includegraphics[height=105pt]{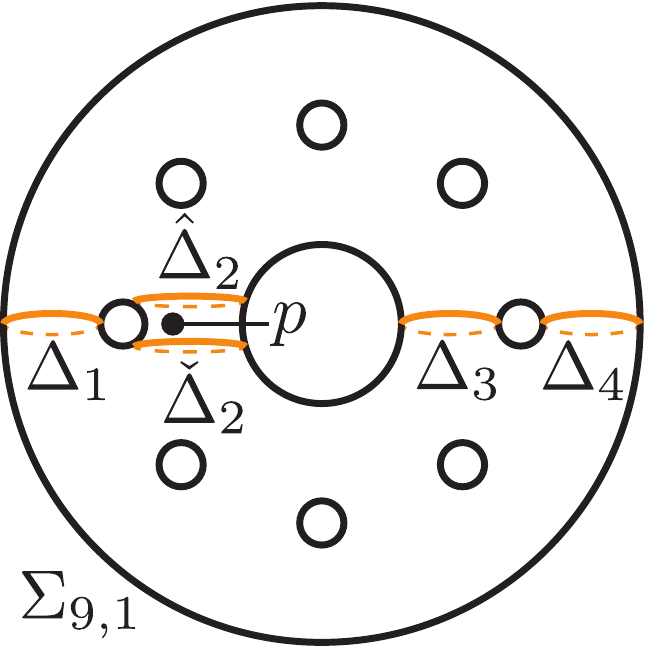}
	\caption{The curves $\Delta_1, \hat{\Delta}_2, \check{\Delta}_2, \Delta_3, \Delta_4$ on $\Sigma_{9,1}$.} 
	\label{fig:Genus9Deltas_MarkedPoint}
\end{figure}
On $\Sigma_9$ we take two parallel copies of $\Delta_2$, $\hat{\Delta}_2$ and $\check{\Delta}_2$, and add a marked point between them as in Figure~\ref{fig:Genus9Deltas_MarkedPoint}.
The first two embeddings $\Phi_1$ and $\Phi_2$ are now regarded as embeddings of $\Sigma_3^4$ into $\Sigma_{9,1}$ in a straightforward fashion, which provide the two relations
\begin{align*} 
\DT{a_1} \DT{a_1^\prime} \DT{x_1} \DT{b_1} \DT{b_1^\prime} \DT{y_1} \DT{c_1} \DT{c_1^\prime} \DT{z_1} \DT{d_1} \DT{d_1^\prime} \DT{w_1} = \DT{\Delta_1} \DT{\hat{\Delta}_2} \DT{\Delta_3} \DT{\Delta_4}, \\
\DT{a_2} \DT{a_2^\prime} \DT{x_2} \DT{b_2} \DT{b_2^\prime} \DT{y_2} \DT{c_2} \DT{c_2^\prime} \DT{z_2} \DT{d_2} \DT{d_2^\prime} \DT{w_2} = \DT{\Delta_1} \DT{\check{\Delta}_2} \DT{\Delta_3} \DT{\Delta_4}
\end{align*}
in $\M(\Sigma_{9,1})$.
Combining them as before, we get a factorization (with no point-pushing map yet):
\begin{align} \label{eq:Signature0LF_MarkedPoint_BeforeSubstitutions}
& \underline{\DT{a_1} \DT{a_1^\prime} \DT{a_2} \DT{a_2^\prime} \LDT{\Delta_1} \LDT{\Delta_4}} \cdot \DT{x_1} \DT{x_2} \cdot 
\underline{\DT{b_1} \DT{b_1^\prime} \DT{b_2} \DT{b_2^\prime} \LDT{\Delta_3} \LDT{\Delta_4}} \cdot \DT{y_1} \DT{y_2} \\
& \cdot 
\underline{\DT{c_1} \DT{c_1^\prime} \DT{c_2} \DT{c_2^\prime} \LDT{\check{\Delta}_2} \LDT{\Delta_3}} \cdot \DT{z_1} \DT{z_2} \cdot 
\underline{\DT{d_1} \DT{d_1^\prime} \DT{d_2} \DT{d_2^\prime} \LDT{\Delta_1} \LDT{\hat{\Delta}_2}} \cdot \DT{w_1} \DT{w_2}
= 1. \notag
\end{align}

We can now describe our embeddings $\Phi_3, \Phi_4, \Phi_5, \Phi_6$.
We first modify the presentation of the genus--$3$ relation~\eqref{eq:Genus3LP_4BoundingPairs} so that one of the bounding pair sits in a ``standard position'' as illustrated in Figure~\ref{fig:Genus3LP_StandardBoundingPair}. 
Take the surface $\Sigma_{3,2}^4$ in Figure~\ref{fig:Genus3LP_4BoundingPairs} and slightly slide the boundary components as indicated in Figure~\ref{fig:Genus3LP_Modification}.
Then we conjugate the relation~\eqref{eq:Genus3LP_4BoundingPairs} by the series of Dehn twists 
$$\DT{\beta} \DT{\gamma} \DT{\alpha} \DT{\beta } \DT{\gamma} \DT{\alpha} \DT{\eta}^{-1}$$
with the curves given on the right of Figure~\ref{fig:Genus3LP_Modification}.
This puts the bounding pair $d,d^\prime$ in the standard position as shown in Figure~\ref{fig:Genus3LP_StandardBoundingPair}, where the images of the other curves are also given. Here we keep using the same symbols for the curves as those before the conjugation.
\begin{figure}[htbp]
	\begin{tabular}{c}
	\begin{minipage}[t]{1\hsize}
		\centering
		\includegraphics[height=85pt]{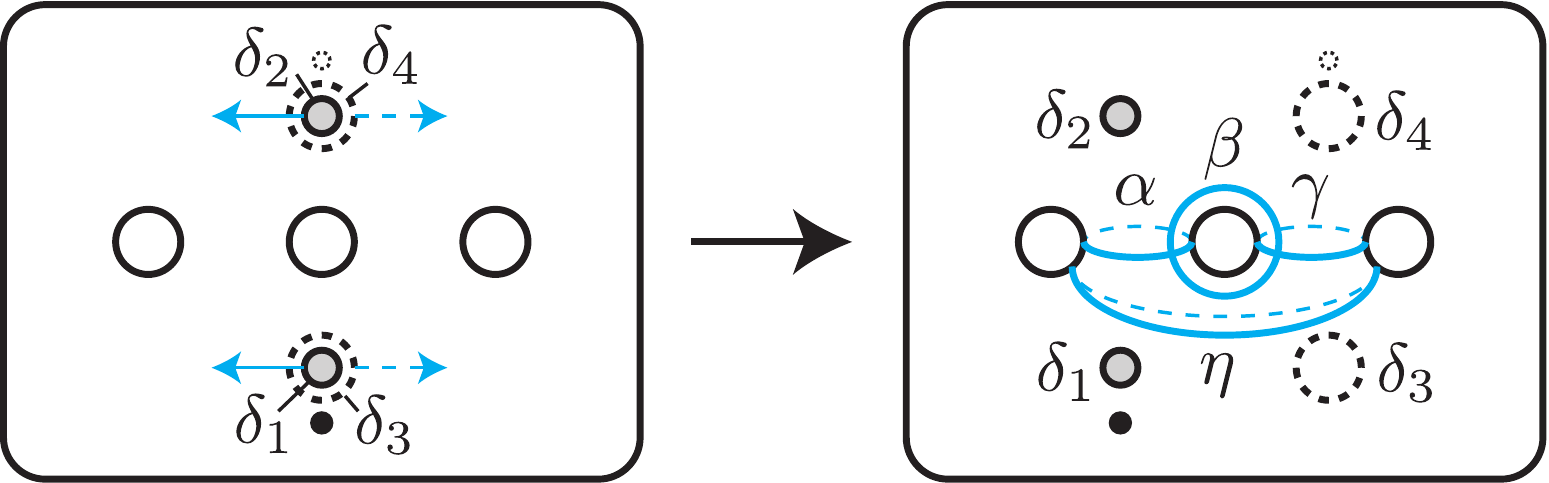}
		\caption{Sliding the boundary components and the curves for the conjugation.} 
		\label{fig:Genus3LP_Modification}
	\end{minipage} \\ 
	\vspace{1\baselineskip} \\
	\begin{minipage}[t]{1\hsize}
		\centering
		\includegraphics[height=255pt]{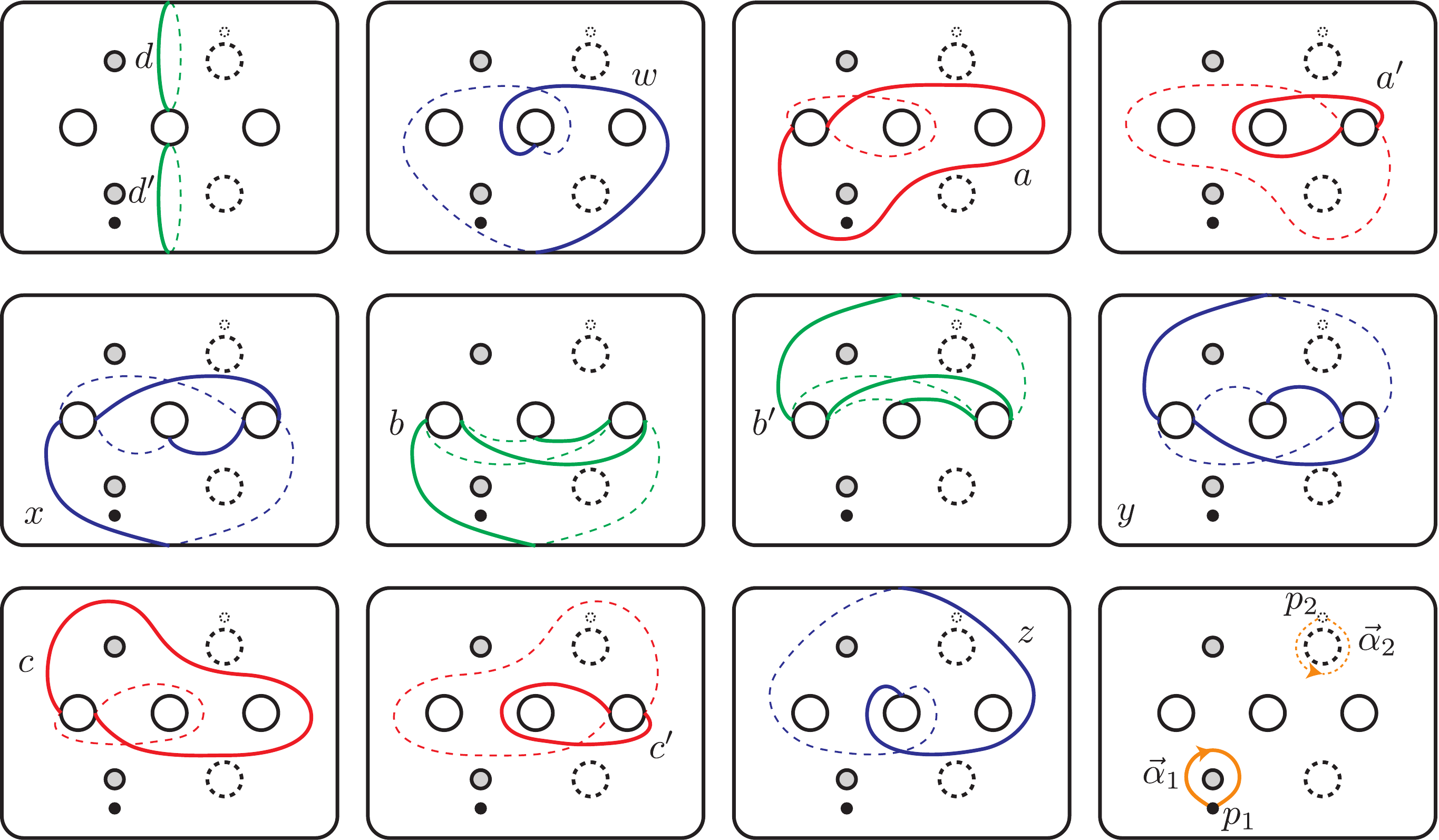}
		\caption{Monodromy curves of the genus--$3$ pencil with simpler $d$ and $d^\prime$.} 
		\label{fig:Genus3LP_StandardBoundingPair}
	\end{minipage}
	\end{tabular}
\end{figure}

By cyclic permutation and commutativity, the relation~\eqref{eq:Genus3LP_4BoundingPairs} becomes
\begin{align} \label{eq:Genus3LP_StandardBoundingPair}
\mathcal{P}_{\vec{\alpha}_2} \cdot \DT{w} \DT{a} \DT{a^\prime} \DT{x} \DT{b} \DT{b^\prime} \DT{y} \DT{c} \DT{c^\prime} \DT{z} \cdot \mathcal{P}_{\vec{\alpha}_1}  = \DT{\delta_1} \DT{\delta_2} \DT{\delta_3} \DT{\delta_4} \LDT{d} \LDT{d^\prime}.
\end{align}
We will embed the relation~\eqref{eq:Genus3LP_StandardBoundingPair} with the curves in Figure~\ref{fig:Genus3LP_StandardBoundingPair} into $\M(\Sigma_{9,1})$.
For $\Phi_3$ and $\Phi_4$ we do not need the marked points.
For $\Phi_5$ we include the marked point $p_1$ and forget $p_2$, whereas for $\Phi_6$ we keep $p_2$ and forget $p_1$. In Figures~\ref{fig:Genus9_Genus3Embedding3}--(d), 
we describe the embeddings
\begin{align*}
\Phi_3: \Sigma_3^4 \to \Sigma_{9,1}, \quad
\Phi_4: \Sigma_3^4 \to \Sigma_{9,1}, \quad
\Phi_5: \Sigma_{3,1}^4 \to \Sigma_{9,1}, \quad
\Phi_6: \Sigma_{3,1}^4 \to \Sigma_{9,1}, 
\end{align*}
as the compositions of embeddings $\tilde{\Phi}_i$ (which are easier to visualize) and diffeomorphisms $\psi_i$ of $\Sigma_{9,1}$.
The embeddings $\tilde{\Phi}_i$ are uniquely specified by describing how the gray curves that fill the genus--$3$ surfaces are embedded. The diffeomorphisms $\psi_3, \psi_4, \psi_5$ and $\psi_6$, are given by the products of Dehn twists below. Here $t_i$ means the right-handed Dehn twist along the curve labeled $i$ in the respective Figures~\ref{fig:Genus9_Genus3EmbeddingPsi3}--(d):
\begin{align*}
\psi_3 &:= t_{\check{4}}^{-1}t_{\hat{4}}^{-1} t_{\check{3}}^{-1}t_{\hat{3}}^{-1} t_{\check{2}}^{-1}t_{\hat{2}}^{-1} t_{\check{1}}^{-1}t_{\hat{1}}^{-1}, \\
\psi_4 &:= t_{11}^2 t_{\check{10}}t_{\hat{10}} t_{\check{9}}^{-1}t_{\hat{9}}^{-1} t_{\check{8}}^{-1}t_{\hat{8}}^{-1} t_{\check{5}}^{-1}t_{\hat{5}}^{-1} t_{\check{4}}^{-1}t_{\hat{4}}^{-1} t_{\check{7}}^{-1}t_{\hat{7}}^{-1} t_{\check{6}}^{-1}t_{\hat{6}}^{-1} \\
& \qquad \cdot t_{\check{3}}^{-1}t_{\hat{3}}^{-1} t_{\check{2}}^{-1}t_{\hat{2}}^{-1} t_{\check{1}}^{-1}t_{\hat{1}}^{-1} t_{\check{5}}^{-1}t_{\hat{5}}^{-1} t_{\check{4}}^{-1}t_{\hat{4}}^{-1} t_{\check{3}}^{-1}t_{\hat{3}}^{-1} t_{\check{2}}^{-1}t_{\hat{2}}^{-1} t_{\check{1}}^{-1}t_{\hat{1}}^{-1}, \\
\psi_5 &:= t_{\check{4}}t_{\hat{4}} t_{\check{3}}t_{\hat{3}} t_{\check{2}}t_{\hat{2}} t_{\check{1}}t_{\hat{1}}, \\
\psi_6 &:= t_{\check{6}}^{-1}t_{\hat{6}}^{-1} t_{\check{5}}t_{\hat{5}} t_{\check{4}}t_{\hat{4}} t_{\check{3}}^{-1}t_{\hat{3}}^{-1} t_{\check{2}}^{-1}t_{\hat{2}}^{-1} t_{\check{1}}^{-1}t_{\hat{1}}^{-1} t_{\check{3}}^{-1}t_{\hat{3}}^{-1} t_{\check{2}}^{-1}t_{\hat{2}}^{-1} t_{\check{1}}^{-1}t_{\hat{1}}^{-1}.
\end{align*} 
\begin{figure}[htbp]
	\centering
	\subfigure[The embedding $\Phi_3 = \psi_3 \circ \tilde{\Phi}_3$. \label{fig:Genus9_Genus3Embedding3}]
	{\includegraphics[height=115pt]{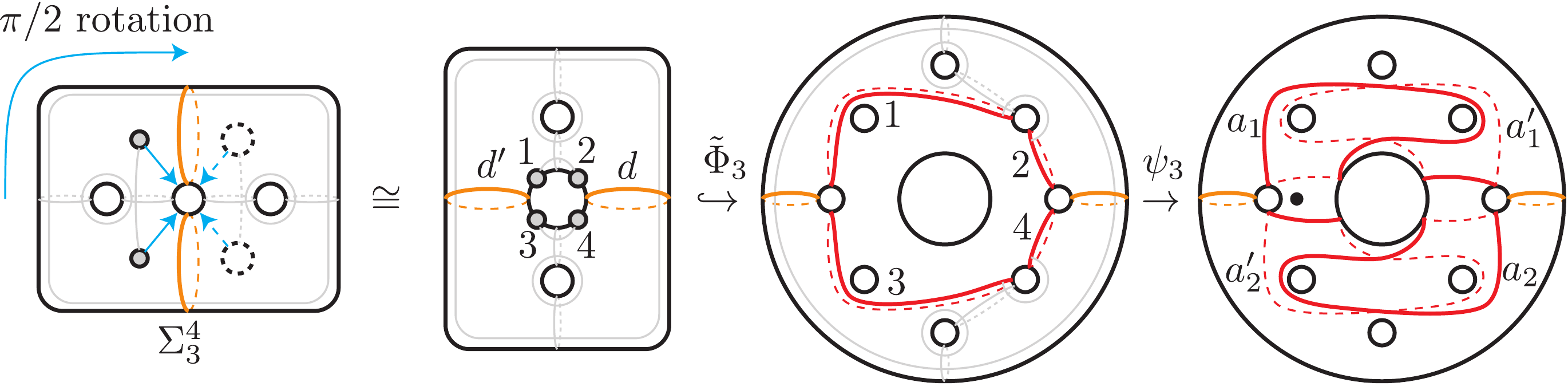}} 
	\hspace{0pt}	
	\subfigure[The embedding $\Phi_4 = \psi_4 \circ \tilde{\Phi}_4$. \label{fig:Genus9_Genus3Embedding4}]
	{\includegraphics[height=115pt]{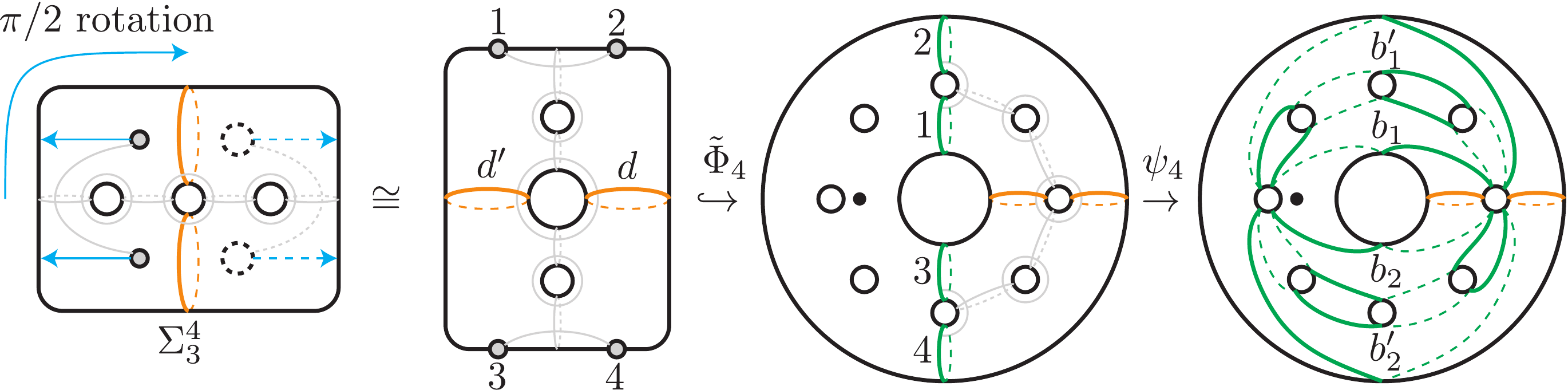}} 
	\hspace{0pt}	
	\subfigure[The embedding $\Phi_5 = \psi_5 \circ \tilde{\Phi}_5$. \label{fig:Genus9_Genus3Embedding5}]
	{\includegraphics[height=115pt]{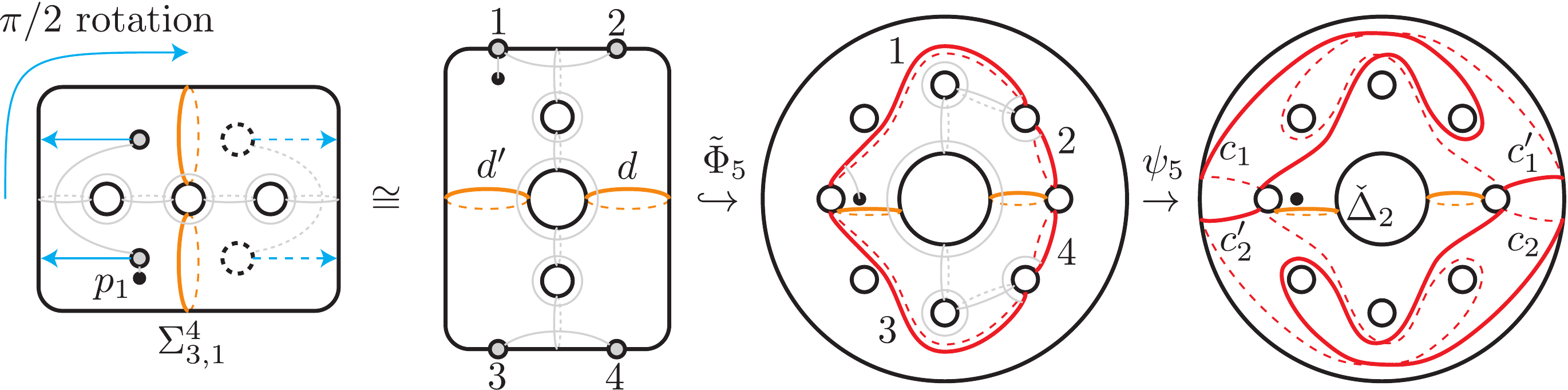}} 
	\hspace{0pt}	
	\subfigure[The embedding $\Phi_6 = \psi_6 \circ \tilde{\Phi}_6$. \label{fig:Genus9_Genus3Embedding6}]
	{\includegraphics[height=115pt]{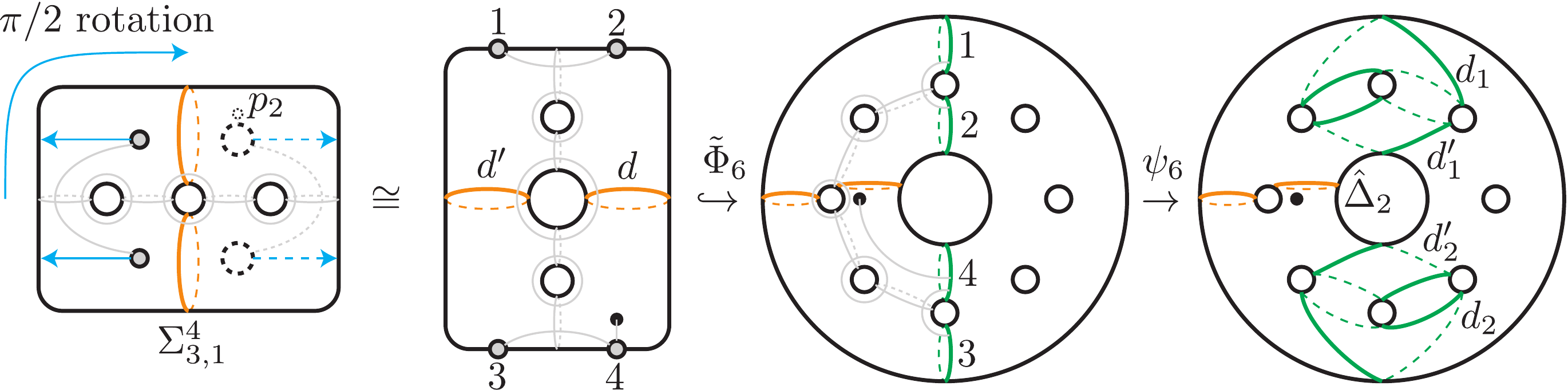}} 
	\caption{The embeddings.} 	
	\label{fig:Genus9_Genus3Embeddings}
\end{figure}
\begin{figure}[htbp]
	\centering
	\subfigure[Curves for $\psi_3$ \label{fig:Genus9_Genus3EmbeddingPsi3}]
	{\includegraphics[height=100pt]{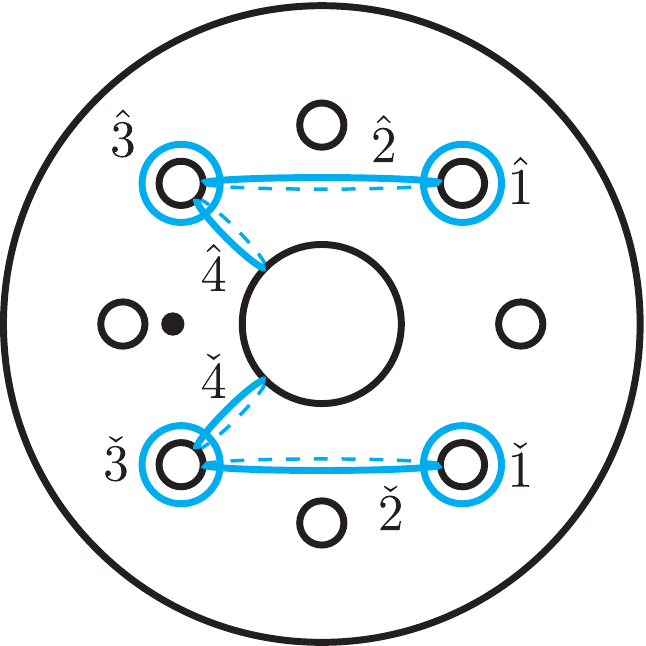}} 
	\hspace{0pt}	
	\subfigure[for $\psi_4$ \label{fig:Genus9_Genus3EmbeddingPsi4}]
	{\includegraphics[height=100pt]{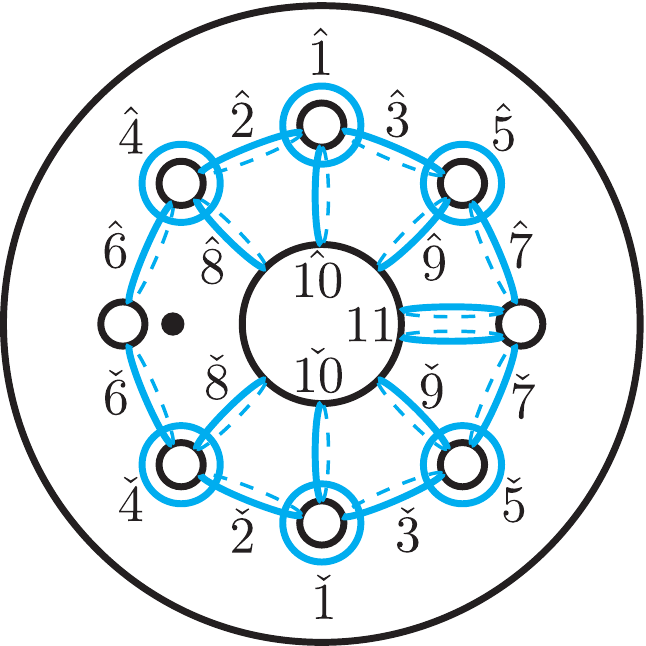}} 
	\hspace{0pt}	
	\subfigure[for $\psi_5$ \label{fig:Genus9_Genus3EmbeddingPsi5}]
	{\includegraphics[height=100pt]{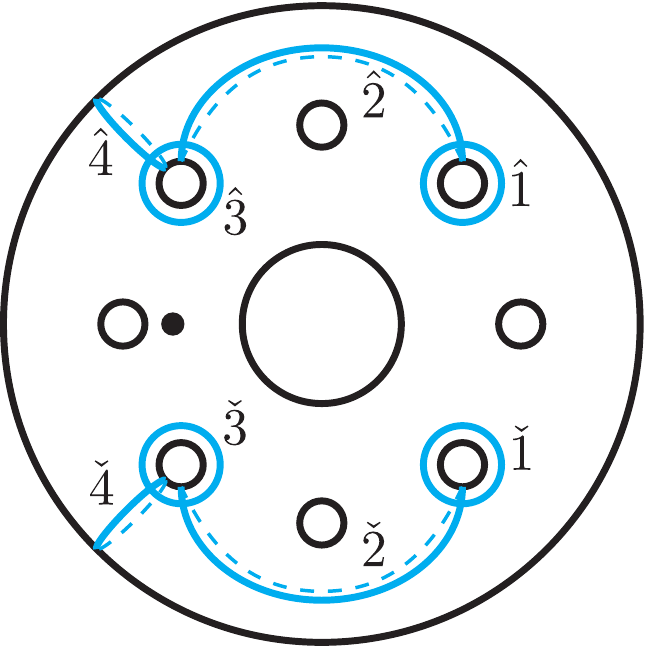}} 
	\hspace{0pt}	
	\subfigure[for $\psi_6$ \label{fig:Genus9_Genus3EmbeddingPsi6}]
	{\includegraphics[height=100pt]{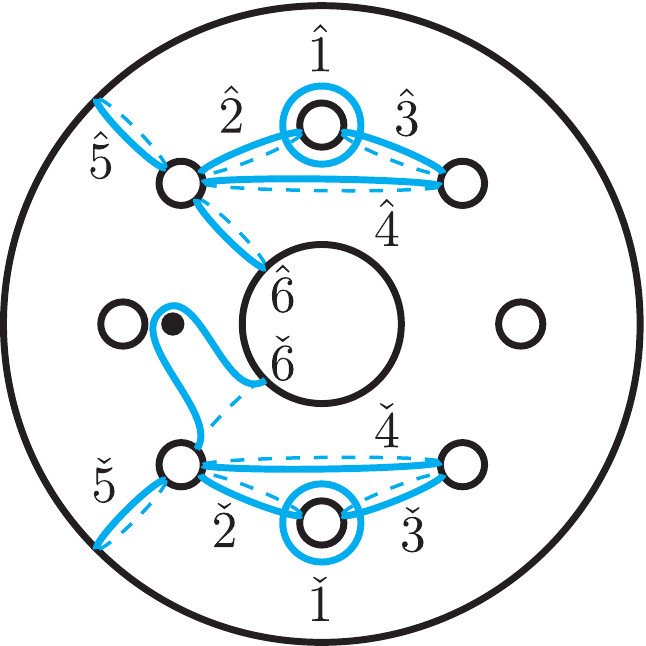}} 
	\caption{Dehn twist curves for the diffeomorphisms $\psi_3, \psi_4, \psi_5, \psi_6$.} 	
	\label{fig:Genus9_Genus3EmbeddingPsi's}
\end{figure}
So we have
\begin{itemize}
	\item $\Phi_3$ maps $(\delta_1,\delta_2,\delta_3,\delta_4,d,d^\prime)$ to $(a_1,a_1^\prime,a_2^\prime, a_2, \Delta_4, \Delta_1)$,
	\item $\Phi_4$ maps $(\delta_1,\delta_2,\delta_3,\delta_4,d,d^\prime)$ to $(b_1,b_1^\prime,b_2,b_2^\prime,  \Delta_4, \Delta_3)$,
	\item $\Phi_5$ maps $(\delta_1,\delta_2,\delta_3,\delta_4,d,d^\prime)$ to $(c_1,c_1^\prime,c_2^\prime, c_2, \Delta_3, \check{\Delta}_2)$,
	\item $\Phi_6$ maps $(\delta_1,\delta_2,\delta_3,\delta_4,d,d^\prime)$ to $(d_1,d_1^\prime,d_2,d_2^\prime, \hat{\Delta}_2, \Delta_1)$.
\end{itemize}
Under these embeddings, the modified genus--$3$ relation~\eqref{eq:Genus3LP_StandardBoundingPair} yields
\begin{align*}
\DT{w_3} \DT{a_3} \DT{a_3^\prime} \DT{x_3} \DT{b_3} \DT{b_3^\prime} \DT{y_3} \DT{c_3} \DT{c_3^\prime} \DT{z_3} &= \DT{a_1} \DT{a_1^\prime} \DT{a_2} \DT{a_2^\prime} \LDT{\Delta_1} \LDT{\Delta_4}, \\
\DT{w_4} \DT{a_4} \DT{a_4^\prime} \DT{x_4} \DT{b_4} \DT{b_4^\prime} \DT{y_4} \DT{c_4} \DT{c_4^\prime} \DT{z_4} &= \DT{b_1} \DT{b_1^\prime} \DT{b_2} \DT{b_2^\prime} \LDT{\Delta_3} \LDT{\Delta_4}, \\
\DT{w_5} \DT{a_5} \DT{a_5^\prime} \DT{x_5} \DT{b_5} \DT{b_5^\prime} \DT{y_5} \DT{c_5} \DT{c_5^\prime} \DT{z_5} \cdot \mathcal{P}_{\vec{\alpha}_1} &= \DT{c_1} \DT{c_1^\prime} \DT{c_2} \DT{c_2^\prime} \LDT{\check{\Delta}_2} \LDT{\Delta_3}, \\
\mathcal{P}_{\vec{\alpha}_2} \cdot \DT{w_6} \DT{a_6} \DT{a_6^\prime} \DT{x_6} \DT{b_6} \DT{b_6^\prime} \DT{y_6} \DT{c_6} \DT{c_6^\prime} \DT{z_6} &= \DT{d_1} \DT{d_1^\prime} \DT{d_2} \DT{d_2^\prime} \LDT{\Delta_1} \LDT{\hat{\Delta}_2}.
\end{align*}
The images of the curves under these embeddings are given in Figures~\ref{fig:SignatureZeroLF_VanishingCycles}.
In addition, the oriented loops $\vec{\alpha}_1, \vec{\alpha}_2$ for the point-pushing maps are given in Figure~\ref{fig:Genus9_PushMaps}.
By substituting all into the relation~\eqref{eq:Signature0LF_MarkedPoint_BeforeSubstitutions},
we get
\begin{align*} 
&\DT{w_3} \DT{a_3} \DT{a_3^\prime} \DT{x_3} \DT{b_3} \DT{b_3^\prime} \DT{y_3} \DT{c_3} \DT{c_3^\prime} \DT{z_3} \DT{x_1} \DT{x_2} \cdot
\DT{w_4} \DT{a_4} \DT{a_4^\prime} \DT{x_4} \DT{b_4} \DT{b_4^\prime} \DT{y_4} \DT{c_4} \DT{c_4^\prime} \DT{z_4} \DT{y_1} \DT{y_2} \\ \notag
&\cdot
\DT{w_5} \DT{a_5} \DT{a_5^\prime} \DT{x_5} \DT{b_5} \DT{b_5^\prime} \DT{y_5} \DT{c_5} \DT{c_5^\prime} \DT{z_5} \cdot \underline{\mathcal{P}_{\vec{\alpha}_1} \cdot \DT{z_1} \DT{z_2} \cdot \mathcal{P}_{\vec{\alpha}_2}} \cdot
\DT{w_6} \DT{a_6} \DT{a_6^\prime} \DT{x_6} \DT{b_6} \DT{b_6^\prime} \DT{y_6} \DT{c_6} \DT{c_6^\prime} \DT{z_6} \DT{w_1} \DT{w_2}
=1. 
\end{align*}
Here $z_1$ and $z_2$ are disjoint, $\vec{\alpha}_1$ and $z_2$ are disjoint, and $z_1$ and $\vec{\alpha}_2$ are disjoint, so we get
\begin{align*}
\mathcal{P}_{\vec{\alpha}_1} \cdot \DT{z_1} \DT{z_2} \cdot \mathcal{P}_{\vec{\alpha}_2} 
= \mathcal{P}_{\vec{\alpha}_1} \cdot \DT{z_2} \DT{z_1} \cdot \mathcal{P}_{\vec{\alpha}_2} 
= \DT{z_2} \cdot \mathcal{P}_{\vec{\alpha}_1} \mathcal{P}_{\vec{\alpha}_2} \cdot \DT{z_1}.
\end{align*}
Then the  product $\mathcal{P}_{\vec{\alpha}_1} \mathcal{P}_{\vec{\alpha}_2}$ is equal to the single point-pushing map along $\vec{\alpha}$, which is homotopic to the concatenation $\vec{\alpha}_2 \cdot \vec{\alpha}_1$.
\begin{figure}[htbp]
	\centering
	\includegraphics[height=100pt]{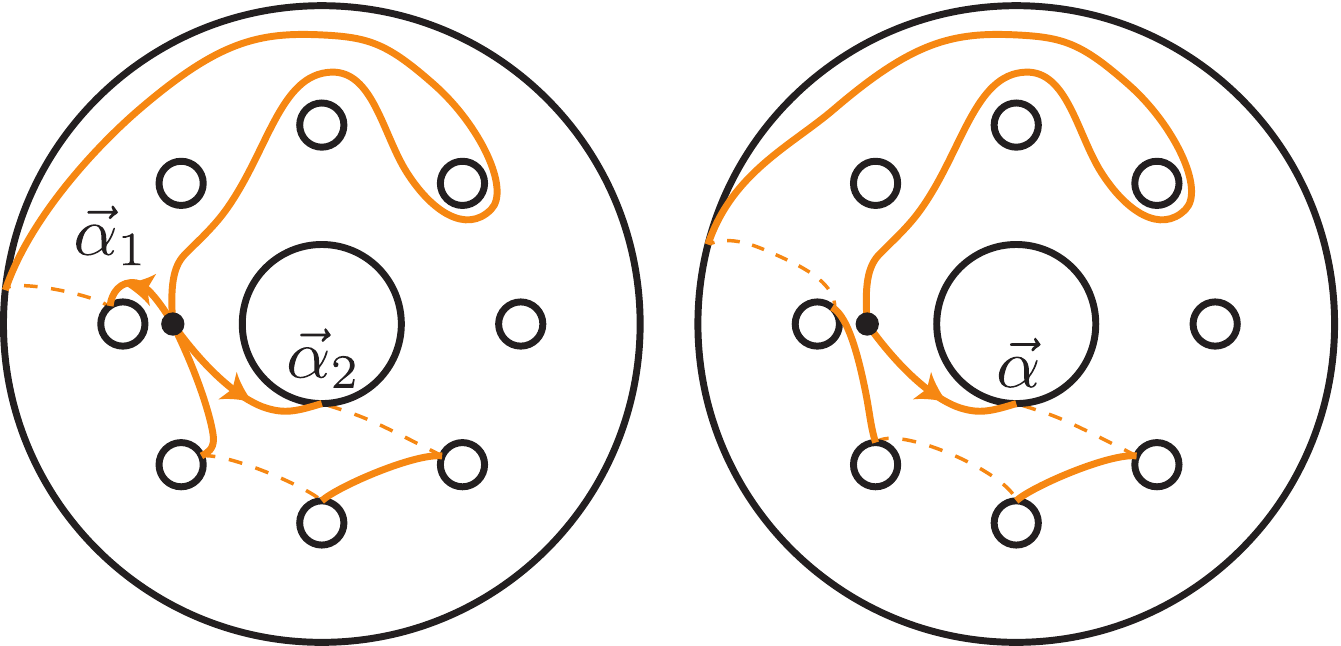}
	\caption{Curves of the point-pushing maps.} 
	\label{fig:Genus9_PushMaps}
\end{figure}%

Hence, we obtain the factorization in $\M(\Sigma_{9,1})$: 
\begin{align} \label{eq:Signature0LF_Explicit}
 \DT{w_3} \DT{a_3} \DT{a_3^\prime} \DT{x_3} \DT{b_3} \DT{b_3^\prime} \DT{y_3} \DT{c_3} \DT{c_3^\prime} \DT{z_3} \DT{x_1} \DT{x_2} \cdot
\DT{w_4} \DT{a_4} \DT{a_4^\prime} \DT{x_4} \DT{b_4} \DT{b_4^\prime} \DT{y_4} \DT{c_4} \DT{c_4^\prime} \DT{z_4} \DT{y_1} \DT{y_2}  & \\ 
\notag
\cdot 
\DT{w_5} \DT{a_5} \DT{a_5^\prime} \DT{x_5} \DT{b_5} \DT{b_5^\prime} \DT{y_5} \DT{c_5} \DT{c_5^\prime} \DT{z_5} \cdot \DT{z_2} \cdot \mathcal{P}_{\vec{\alpha}} \cdot \DT{z_1} \cdot
\DT{w_6} \DT{a_6} \DT{a_6^\prime} \DT{x_6} \DT{b_6} \DT{b_6^\prime} \DT{y_6} \DT{c_6} \DT{c_6^\prime} \DT{z_6} \DT{w_1} \DT{w_2}
&=& 1,
\end{align}
where all the Dehn twist curves are as in Figure~\ref{fig:SignatureZeroLF_VanishingCycles}.
Forgetting the marked point (and thus the point-pushing map in the factorization), we get an explicit monodromy factorization of our signature zero Lefschetz fibration. 

\begin{figure}[htbp]
	\centering
	\includegraphics[height=600pt]{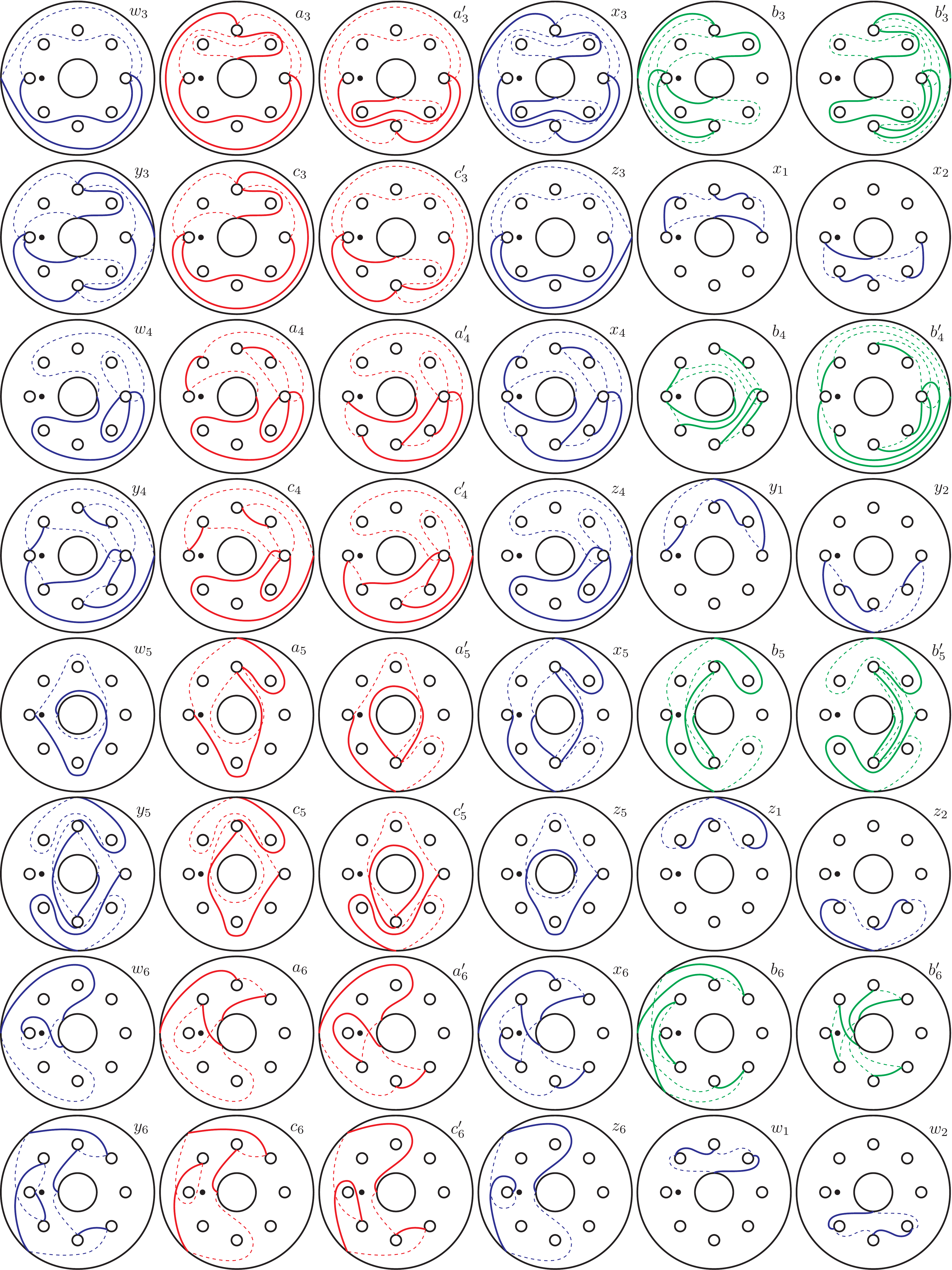}
	\caption{Vanishing cycles of the genus--$9$ Lefschetz fibration with signature zero.} 
	\label{fig:SignatureZeroLF_VanishingCycles}
\end{figure}

\bigskip
\section{Proof of Theorem~A}

Here is an outline for our proof of Theorem~A: Since the novel genus--$9$ Lefschetz fibration $(X,f)$ we built in the previous section will serve as one of the main building blocks of our constructions, we will first analyze the algebraic and differential topology of $X$ in some detail. We will then first prove all the statements for genus--$9$ fibrations, deferring the construction of higher genera examples till the end. We will first construct Lefschetz fibrations with prescribed signatures, and then the spin ones. All will be done by taking products of conjugated positive factorizations of the identity in $\M(\Sigma_g)$ (corresponding to twisted fiber sums of the fibrations) and breedings (one of which corresponds to lantern substitution, thus the rational blowdown), which will require extra care in the spin case.

\subsection{The topology of the signature zero genus-$9$ Lefschetz fibration} \label{topofkeyLF}  \

The essential information we need for later arguments is summed up as follows:

\begin{theorem} \label{KeyLF}
There is a symplectic genus--$9$ Lefschetz fibration $(X,f)$ with \mbox{monodromy factorization}
\begin{align} \label{eq:Signature0LF}
&\DT{w_3} \DT{a_3} \DT{a_3^\prime} \DT{x_3} \DT{b_3} \DT{b_3^\prime} \DT{y_3} \DT{c_3} \DT{c_3^\prime} \DT{z_3} \DT{x_1} \DT{x_2} 
\DT{w_4} \DT{a_4} \DT{a_4^\prime} \DT{x_4} \DT{b_4} \DT{b_4^\prime} \DT{y_4} \DT{c_4} \DT{c_4^\prime} \DT{z_4} \DT{y_1} \DT{y_2} \\ \notag
&\cdot
\DT{w_5} \DT{a_5} \DT{a_5^\prime} \DT{x_5} \DT{b_5} \DT{b_5^\prime} \DT{y_5} \DT{c_5} \DT{c_5^\prime} \DT{z_5} \DT{z_2}  \DT{z_1} 
\DT{w_6} \DT{a_6} \DT{a_6^\prime} \DT{x_6} \DT{b_6} \DT{b_6^\prime} \DT{y_6} \DT{c_6} \DT{c_6^\prime} \DT{z_6} \DT{w_1} \DT{w_2}
=1,
\end{align}
in $\mathrm{Mod}(\Sigma_9)$, where the Dehn twist curves are as in Figure~\ref{fig:SignatureZeroLF_VanishingCycles}. The total space $X$ is a spin symplectic $4$--manifold of general type, which is not deformation equivalent to any compact complex surface, where $\eu(X)=16$, $\sigma(X)=0$, and $H_1(X)=\Z^7 \oplus \Z_4 \oplus \Z_2$. 
\end{theorem}

\begin{proof}
The Lefschetz fibration $(X,f)$ prescribed by the positive factorization~\eqref{eq:Signature0LF} we have constructed in the previous section, is a symplectic fibration with respect to a Gompf-Thurston symplectic form $\omega$ we can equip it with. While it will take some effort to prove that $X$ is spin, all other claims on the differential topology of $X$ follow from our claims on its algebraic topology: For $X$ is spin, it is obviously minimal, and \mbox{$c_1^2(X)=2 \eu(X) + 3 \sigma(X)= 32 >0$,} so the Kodaira dimension of the symplectic $4$--manifold $(X, \omega)$ is $\kappa(X)=2$, in other words, it is of \emph{general type}. Because $b_1(X)$ is odd, $X$ is not even homotopy equivalent to a compact complex surface of general type, which are all known to be  K\"{a}hler. 

\smallskip
We calculate the algebraic invariants of $X$ next.

\noindent{\textit{\underline{Euler characteristic and signature}}:}
The Euler characteristic of $X$ is the easiest to calculate by the formula $\eu(X)=4-4g +n$, where the genus of the Lefschetz fibration $g=9$, and the number of nodes, corresponding to the Dehn twists in the monodromy factorization, is $n=48$. 

Although the calculation of the signature of a Lefschetz fibration usually requires computer assistance to run an algorithm, we  leveraged the fact that we have built the monodromy factorization~\eqref{eq:Signature0LF} for $(X,f)$ from scratch, using only basic relations in the mapping class group. Thanks to the work of Endo and Nagami \cite{EndoNagami}, we easily calculate the signature of $X$ by an algebraic count of the relations we have employed to derive the final positive factorization of the identity in $\M(\Sigma_9)$. Embeddings of relations into higher genera surfaces, cancellations of positive and negative Dehn twists, and Hurwitz moves do not affect the signature calculation. Since we built our genus--$9$ factorization through embeddings of the signature zero genus--$3$ relation, cancellations and Hurwitz moves, we conclude that $\sigma(X)=0$. 

Note that we calculated the signature of our genus--$2$ and genus--$3$ pencils in the same way, recalling that every use of the $2$--chain and the lantern relation contributes $-7$ and $+1$ to the signature count \cite{EndoNagami}. Algebraically, we used one $2$--chain and seven lantern relations to derive the monodromy factorization of the genus--$2$ pencil, whereas these numbers are doubled for the \mbox{genus--$3$} pencil. In turn, a total of $12$ $2$--chain and $84$ lantern relations yield the positive factorization for $(X,f)$. (These large numbers might demonstrate why it is more feasible to build such a relation in multiple steps via breedings.)

\smallskip
\noindent{\textit{\underline{First homology}}:}
The total space of a genus--$g$ Lefschetz fibration over the $2$--sphere has a handle decomposition  $(D^2 \times \Sigma_g) \cup \, \sum_i  h_i \, \cup (D^2 \times \Sigma_g)$, where $h_i$ are $2$--handles attached along the loops on $\Sigma_g$, which are the Dehn twist curves $c_i$ in the monodromy factorization of the fibration  \cite{GompfStipsicz}. Therefore, the first homology of the total space can be calculated by taking a quotient $H_1(\Sigma_g)$, as the first $D^2 \x \Sigma_g$ contains all the $1$--handles, by the abelianized relations induced by the attaching circles of all the $2$--handles, which are the vanishing cycles $c_i$, and the attaching circle of the last $2$--handle coming from the second $D^2 \times \Sigma_g$ in the above decomposition. 

Accordingly, we will calculate $H_1(X)$ using the monodromy factorization of our genus--$9$ fibration. Let $\{\alpha_i, \beta_i\}$ be the homology generators of $H_1(\Sigma_9)$, represented by the  loops in Figure~\ref{fig:Genus9_H1Basis}. 
After picking an auxiliary orientation on each Dehn twist curve in Figure~\ref{fig:SignatureZeroLF_VanishingCycles}, and calculating the number of its algebraic intersections with $\{\alpha_i, \beta_i\}$, we can easily read off the relation induced by the corresponding vanishing cycle. We tabulate this data below in a way the coefficients are easily visible, as it will be vital to our arguments for determining the spin type of $X$ as well. 
\begin{figure}[htbp]
	\centering
	\includegraphics[height=140pt]{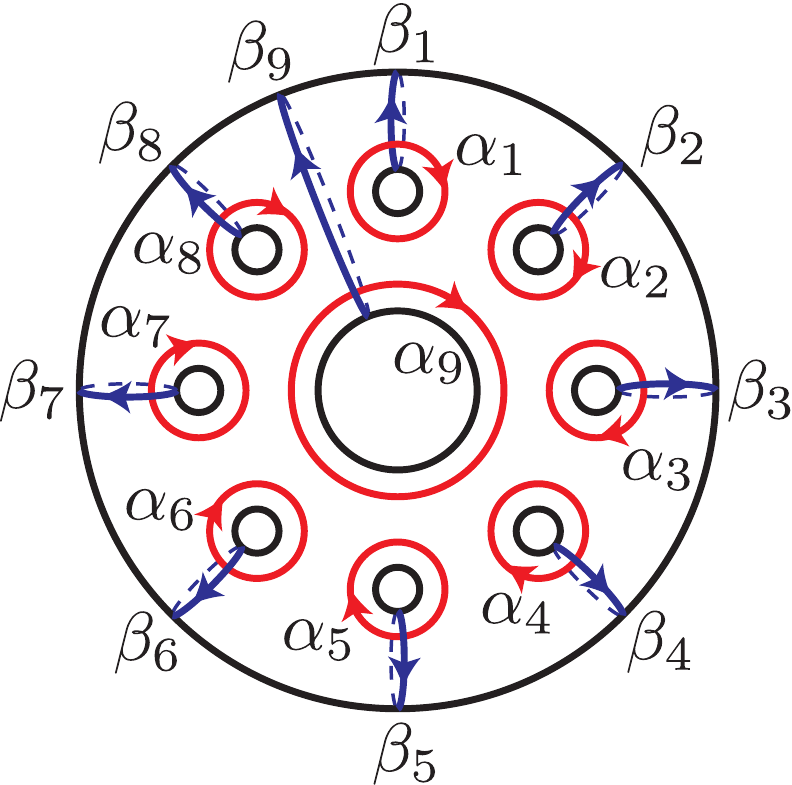}
	\caption{The symplectic basis for $H_1(\Sigma_9;\mathbb{Z})$.} 
	\label{fig:Genus9_H1Basis}
\end{figure}

\medskip
\noindent \textit{Homology classes of the vanishing cycles:} 
\begin{alignat*}{20}
	w_3 &= \alpha_1 & +2\alpha_2 & +2\alpha_3 & +2\alpha_4 & +\alpha_5 & +2\alpha_6 & +2\alpha_7 & +2\alpha_8 & +2\alpha_9 
	&  &  & +2\beta_3 &  &  &  & -\beta_7 &  &  ,\\
	a_3 &= \alpha_1 & +3\alpha_2 & +2\alpha_3 & +2\alpha_4 & +\alpha_5 & +2\alpha_6 & +2\alpha_7 & +3\alpha_8 & +2\alpha_9 
	&  &  & +2\beta_3 &  &  &  & -\beta_7 &  & -\beta_9 ,\\
	a_3^\prime &= \alpha_1 & +2\alpha_2 & +2\alpha_3 & +3\alpha_4 & +\alpha_5 & +3\alpha_6 & +2\alpha_7 & +2\alpha_8 & +2\alpha_9 
	& &  & +2\beta_3 &  &  &  & -\beta_7 &  & -\beta_9 ,\\
	x_3 &= \alpha_1 & +3\alpha_2 & +2\alpha_3 & +3\alpha_4 & +\alpha_5 & +3\alpha_6 & +2\alpha_7 & +3\alpha_8 & +2\alpha_9 
	& & & +2\beta_3 &  &  &  & -\beta_7 &  & -2\beta_9 ,\\
	b_3 &=  & \alpha_2 &  & +\alpha_4 &  & +\alpha_6 &  & +\alpha_8 &  
	& &  &  &  &  &  & +\beta_7 &  & -2\beta_9 ,\\
	b_3^\prime &=  & \alpha_2 &  & +\alpha_4 &  & +\alpha_6 &  & +\alpha_8 & 
	& &  & +\beta_3 &  &  &  &  &  & -2\beta_9 ,\\
	y_3 &= \alpha_1 & +\alpha_2 & +2\alpha_3 & +\alpha_4 & +\alpha_5 & +\alpha_6 & +2\alpha_7 & +\alpha_8 & +2\alpha_9 
	&  &  & +\beta_3 &  &  &  & -2\beta_7 &  & +2\beta_9 ,\\
	c_3 &= \alpha_1 & +\alpha_2 & +2\alpha_3 & +2\alpha_4 & +\alpha_5 & +2\alpha_6 & +2\alpha_7 & +\alpha_8 & +2\alpha_9 
	&  &  & +\beta_3 &  &  &  & -2\beta_7 &  & +\beta_9 ,\\
	c_3^\prime &= \alpha_1 & +2\alpha_2 & +2\alpha_3 & +\alpha_4 & +\alpha_5 & +\alpha_6 & +2\alpha_7 & +2\alpha_8 & +2\alpha_9 
	&  &  & +\beta_3 &  &  &  & -2\beta_7 &  & +\beta_9 ,\\
	z_3 &= \alpha_1 & +2\alpha_2 & +2\alpha_3 & +2\alpha_4 & +\alpha_5 & +2\alpha_6 & +2\alpha_7 & +2\alpha_8 & +2\alpha_9 
	&  &  & +\beta_3 &  &  &  & -2\beta_7 &  &  ,\\
	x_1 &=  & \alpha_2 &  &  &  &  &  & +\alpha_8 &  
	& +\beta_1 & -\beta_2 & +\beta_3 &  &  &  & +\beta_7 & -\beta_8 & -\beta_9 ,\\
	x_2 &=  &  &  & \alpha_4 &  & +\alpha_6 &  &  &  
	&  &  & +\beta_3 & -\beta_4 & +\beta_5 & -\beta_6 & +\beta_7 &  & -\beta_9 ,\\
	w_4 &= \alpha_1 & +2\alpha_2 & +2\alpha_3 & +2\alpha_4 & +\alpha_5 & +\alpha_6 &  & +\alpha_8 &  
	&  &  & +\beta_3 &  &  &  &  &  & -\beta_9 ,\\
	a_4 &= \alpha_1 & +2\alpha_2 & +2\alpha_3 & +2\alpha_4 & +\alpha_5 & +\alpha_6 &  & +\alpha_8 &  
	& +\beta_1 & -\beta_2 & +\beta_3 &  &  &  & +\beta_7 & -\beta_8 & -\beta_9 ,\\
	a_4^\prime &= \alpha_1 & +2\alpha_2 & +2\alpha_3 & +2\alpha_4 & +\alpha_5 & +\alpha_6 &  & +\alpha_8 &  
	&  &  & +\beta_3 & -\beta_4 & +\beta_5 & -\beta_6 & +\beta_7 &  & -\beta_9 ,\\
	x_4 &= \alpha_1 & +2\alpha_2 & +2\alpha_3 & +2\alpha_4 & +\alpha_5 & +\alpha_6 &  & +\alpha_8 &  
	& +\beta_1 & -\beta_2 & +\beta_3 & -\beta_4 & +\beta_5 & -\beta_6 & +2\beta_7 & -\beta_8 & -\beta_9 ,\\
	b_4 &=  &  &  &  &  &  &  &  &  
	& \beta_1 & -\beta_2 & +\beta_3 & -\beta_4 & +\beta_5 & -\beta_6 & +2\beta_7 & -\beta_8 & -\beta_9 ,\\
	b_4^\prime &=  &  &  &  &  &  &  &  &  
	& \beta_1 & -\beta_2 & +\beta_3 & -\beta_4 & +\beta_5 & -\beta_6 & +2\beta_7 & -\beta_8 &  ,\\
	y_4 &= \alpha_1 & +2\alpha_2 & +2\alpha_3 & +2\alpha_4 & +\alpha_5 & +\alpha_6 &  & +\alpha_8 &  
	& -\beta_1 & +\beta_2 & -\beta_3 & +\beta_4 & -\beta_5 & +\beta_6 & -2\beta_7 & +\beta_8 &  ,\\
	c_4 &= \alpha_1 & +2\alpha_2 & +2\alpha_3 & +2\alpha_4 & +\alpha_5 & +\alpha_6 &  & +\alpha_8 &  
	& -\beta_1 & +\beta_2 & -\beta_3 &  &  &  & -\beta_7 & +\beta_8 &  ,\\
	c_4^\prime &= \alpha_1 & +2\alpha_2 & +2\alpha_3 & +2\alpha_4 & +\alpha_5 & +\alpha_6 &  & +\alpha_8 &  
	&  &  & -\beta_3 & +\beta_4 & -\beta_5 & +\beta_6 & -\beta_7 &  &  ,\\
	z_4 &= \alpha_1 & +2\alpha_2 & +2\alpha_3 & +2\alpha_4 & +\alpha_5 & +\alpha_6 &  & +\alpha_8 &  
	&  &  & -\beta_3 &  &  &  &  &  &  ,\\
	y_1 &=  & \alpha_2 &  &  &  &  &  & +\alpha_8 &  
	& -\beta_1 & +\beta_2 & -\beta_3 &  &  &  & -\beta_7 & +\beta_8 &  ,\\
	y_2 &=  &  &  & \alpha_4 &  & +\alpha_6 &  &  &  
	&  &  & -\beta_3 & +\beta_4 & -\beta_5 & +\beta_6 & -\beta_7 &  &  ,\\
	w_5 &= \alpha_1 &  &  &  & +\alpha_5 &  &  &  & +2\alpha_9 
	&  &  &  &  &  &  & -\beta_7 &  & +\beta_9 ,\\
	a_5 &= \alpha_1 & +\alpha_2 &  &  & +\alpha_5 &  &  & +\alpha_8 & +2\alpha_9 
	&  &  &  &  &  &  & -\beta_7 &  &  ,\\
	a_5^\prime &= \alpha_1 &  &  & +\alpha_4 & +\alpha_5 & +\alpha_6 &  &  & +2\alpha_9 
	&  &  &  &  &  &  & -\beta_7 &  &  ,\\
	x_5 &= \alpha_1 & +\alpha_2 &  & +\alpha_4 & +\alpha_5 & +\alpha_6 &  & +\alpha_8 & +2\alpha_9 
	&  &  &  &  &  &  & -\beta_7 &  & -\beta_9 ,\\
	b_5 &=  & \alpha_2 &  & +\alpha_4 &  & +\alpha_6 &  & +\alpha_8 &  
	&  &  &  &  &  &  & -\beta_7 &  & -\beta_9 ,\\
	b_5^\prime &=  & \alpha_2 &  & +\alpha_4 &  & +\alpha_6 &  & +\alpha_8 &  
	&  &  & -\beta_3 &  &  &  &  &  & -\beta_9 ,\\
	y_5 &= \alpha_1 & -\alpha_2 &  & -\alpha_4 & +\alpha_5 & -\alpha_6 &  & -\alpha_8 & +2\alpha_9 
	&  &  & +\beta_3 &  &  &  &  &  & +\beta_9 ,\\
	c_5 &= \alpha_1 & -\alpha_2 &  &  & +\alpha_5 &  &  & -\alpha_8 & +2\alpha_9 
	&  &  & +\beta_3 &  &  &  &  &  &  ,\\
	c_5^\prime &= \alpha_1 &  &  & -\alpha_4 & +\alpha_5 & -\alpha_6 &  &  & +2\alpha_9 
	&  &  & +\beta_3 &  &  &  &  &  &  ,\\
	z_5 &= \alpha_1 &  &  &  & +\alpha_5 &  &  &  & +2\alpha_9 
	&  &  & +\beta_3 &  &  &  &  &  & -\beta_9 ,\\
	z_1 &=  & \alpha_2 &  &  &  &  &  & +\alpha_8 &  
	& +\beta_1 & -\beta_2 &  &  &  &  &  & -\beta_8 &  ,\\
	z_2 &=  &  &  & \alpha_4 &  & +\alpha_6 &  &  &  
	&  &  &  & -\beta_4 & +\beta_5 & -\beta_6 &  &  &  ,\\
	w_6 &= \alpha_1 &  &  &  & +\alpha_5 & +\alpha_6 & +2\alpha_7 & +\alpha_8 &  
	&  &  &  &  &  &  & +\beta_7 &  & -2\beta_9 ,\\
	a_6 &= \alpha_1 &  &  &  & +\alpha_5 & +\alpha_6 & +2\alpha_7 & +\alpha_8 &  
	& -\beta_1 & +\beta_2 &  &  &  &  &  & +\beta_8 & -2\beta_9 ,\\
	a_6^\prime &= \alpha_1 &  &  &  & +\alpha_5 & +\alpha_6 & +2\alpha_7 & +\alpha_8 &  
	&  &  &  & +\beta_4 & -\beta_5 & +\beta_6 &  &  & -2\beta_9 ,\\
	x_6 &= \alpha_1 &  &  &  & +\alpha_5 & +\alpha_6 & +2\alpha_7 & +\alpha_8 &  
	& -\beta_1 & +\beta_2 &  & +\beta_4 & -\beta_5 & +\beta_6 & -\beta_7 & +\beta_8 & -2\beta_9 ,\\
	b_6 &=  &  &  &  &  &  &  &  &  
	& \beta_1 & -\beta_2 &  & -\beta_4 & +\beta_5 & -\beta_6 & +\beta_7 & -\beta_8 &  ,\\
	b_6^\prime &=  &  &  &  &  &  &  &  &  
	& \beta_1 & -\beta_2 &  & -\beta_4 & +\beta_5 & -\beta_6 & +\beta_7 & -\beta_8 & +\beta_9 ,\\
	y_6 &= \alpha_1 &  &  &  & +\alpha_5 & +\alpha_6 & +2\alpha_7 & +\alpha_8 &  
	& +\beta_1 & -\beta_2 &  & -\beta_4 & +\beta_5 & -\beta_6 & +\beta_7 & -\beta_8 & -\beta_9 ,\\
	c_6 &= \alpha_1 &  &  &  & +\alpha_5 & +\alpha_6 & +2\alpha_7 & +\alpha_8 &  
	& +\beta_1 & -\beta_2 &  &  &  &  &  & -\beta_8 & -\beta_9 ,\\
	c_6^\prime &= \alpha_1 &  &  &  & +\alpha_5 & +\alpha_6 & +2\alpha_7 & +\alpha_8 & &  &  &  & -\beta_4 & +\beta_5 & -\beta_6 &  &  & -\beta_9 ,\\
	z_6 &= \alpha_1 &  &  &  & +\alpha_5 & +\alpha_6 & +2\alpha_7 & +\alpha_8 &  
	&  &  &  &  &  &  & -\beta_7 &  & -\beta_9 ,\\
	w_1 &=  & \alpha_2 &  &  &  &  &  & +\alpha_8 &  
	& -\beta_1 & +\beta_2 &  &  &  &  &  & +\beta_8 & -\beta_9 ,\\
	w_2 &=  &  &  & \alpha_4 &  & +\alpha_6 &  &  &  
	&  &  &  & +\beta_4 & -\beta_5 & +\beta_6 &  &  & -\beta_9.
\end{alignat*}

\bigskip
Recall that there is one more $2$--handle we have to consider, and as we explained in the proof of Proposition~\ref{divisibility} in Section~\ref{sec:spin}, its attaching circle is homologous to the sum of the oriented curves of the point-pushing maps for \emph{any} chosen marked point inducing a lift of the monodromy factorization from $\M(\Sigma_9)$ to $\M(\Sigma_{9,1})$. (If we could find a section of $(X,f)$, the lift for the corresponding marked point would have no point-pushing maps, so this homology class would be trivial.)  Looking at the point-pushing map in the relation~\eqref{eq:Signature0LF_Explicit} in $\M(\Sigma_{9,1})$ we get the following additional relation in homology: 
\[
s = \alpha_2  +\alpha_8   +\beta_4 -\beta_5 +\beta_6  -\beta_7 -\beta_9 =0 \, . \\
\]

With this explicit presentation in hand, determining the finitely generated abelian group $H_1(X)$ is now a straightforward calculation. We will show that $H_1(X)=\Z^7 \oplus \Z_4 \oplus \Z_2.$

From $b_4^\prime-b_4=\beta_9$ we see $\beta_9=0$.
Then, $a_3-w_3 = \alpha_2 + \alpha_8 -\beta_9$, $a_3^\prime-w_3 = \alpha_4 + \alpha_6 -\beta_9$ give $\alpha_2 + \alpha_8 = 0$, $\alpha_4 + \alpha_6=0$.
Since $b_3 = \beta_7$ modulo $\beta_9 = \alpha_2 + \alpha_8 = \alpha_4 + \alpha_6=0$, we have $\beta_7=0$. Similarly, $b_3^\prime = 0$ gives $\beta_3=0$.
Now $x_1 = \beta_1 -\beta_2 -\beta_8$ and $x_2 = -\beta_4 +\beta_5 -\beta_6$ modulo $\beta_9 = \alpha_2 + \alpha_8 = \alpha_4 + \alpha_6= \beta_3 = \beta_7 =0$, thus $\beta_1 -\beta_2 -\beta_8 = -\beta_4 +\beta_5 -\beta_6 =0$, or $\beta_1 = \beta_2 + \beta_8$ and $\beta_5 = \beta_4 + \beta_6$.
From $w_5=0$ with $\beta_7 = \beta_9 = 0$ we see $\alpha_1 + \alpha_5 + 2\alpha_9 = 0$.
Looking at $w_3=0$ together with $\alpha_1 + \alpha_5 + 2\alpha_9 = \alpha_2 + \alpha_8 = \alpha_4 + \alpha_6 = \beta_3 = \beta_7 = 0$ we get $2\alpha_3 + 2\alpha_7 =0$.
Combining
$0 = w_4+w_6 = 2\alpha_1 + 2\alpha_2 +2\alpha_3 +2\alpha_4 +2\alpha_5 +2\alpha_6 +2\alpha_7 +2\alpha_8 +\beta_3 +\beta_7 -3\beta_9$ and
$\alpha_2 + \alpha_8 = \alpha_4 + \alpha_6 = 2\alpha_3 + 2\alpha_7 = \beta_3 = \beta_7 = \beta_9 =0$ we obtain $2\alpha_1 + 2\alpha_5 =0$. Since we also know $\alpha_1 + \alpha_5 + 2\alpha_9=0$ we deduce $4\alpha_9=0$.
We turn to $w_6=0$ and substitute $\alpha_1 + \alpha_5 = 2\alpha_9, \beta_7=\beta_9=0$ to get $-2\alpha_9 +\alpha_6 +2\alpha_7 +\alpha_8 =0$, or $\alpha_8 = -\alpha_6 -2\alpha_7 +2\alpha_9$.
In summary, we have
\begin{align*}
&2(\alpha_3 + \alpha_7) =0; \\
&4\alpha_9 =0; \\
&\alpha_2 = -\alpha_8 = \alpha_6 +2\alpha_7 -2\alpha_9; \\
&\alpha_4 = -\alpha_6; \\
&\alpha_5 = -\alpha_1 - 2\alpha_9; \\
&\alpha_8 = -\alpha_6 -2\alpha_7 +2\alpha_9; \\
&\beta_1 = \beta_2 + \beta_8; \\
&\beta_3 = 0; \\
&\beta_5 = \beta_4 + \beta_6; \\
&\beta_7 = 0; \\
&\beta_9 = 0.
\end{align*}
It is easy to check that any other relations that come from the vanishing cycles can be deduced from the above relations.
The extra relation coming from the point-pushing map is
\[
s = \alpha_2  +\alpha_8   +\beta_4 -\beta_5 +\beta_6  -\beta_7 -\beta_9 =0. \\
\]
However, this relation can be deduced already from the relations coming from vanishing cycles since $\alpha_2 + \alpha_8 = \beta_4 -\beta_5 +\beta_6 = \beta_7 = \beta_9 =0$.
This implies that we have a pseudosection.
So we can take $g_1 =\alpha_1, g_2 = \alpha_3, g_3 = \alpha_6, 
g_4 = \beta_2, g_5 = \beta_4, g_6 = \beta_6, g_7 = \beta_8,
g_8 = \alpha_9, g_9 = \alpha_3 + \alpha_7$ as generators and the only relations among them are
$4g_8=0$ and $2g_9=0$. We conclude that
\[H_1(X)=\Z^7 \oplus \Z_4 \oplus \Z_2. \]

\smallskip
\noindent{\textit{\underline{Other algebraic invariants}}:}
While we can derive an explicit  presentation of $\pi_1(X)$ in the same fashion we did for $H_1(X)$ earlier, we will not actually need all of this more massive presentation for any of our arguments to follow, even the ones that involve killing the fundamental group. Instead, it will suffice to observe that in the monodromy factorization of $(X,f)$, we have three disjoint nonseparating curves, coming from the bounding quadruple of curves $x_1, z_1, x_2, z_2$; see Figure~\ref{fig:Genus9_5thBoundingQuadruple}. These three curves simultaneously kill three generators of $\pi_1(\Sigma_9)$.

\begin{figure}[htbp]
	\centering
	\includegraphics[height=100pt]{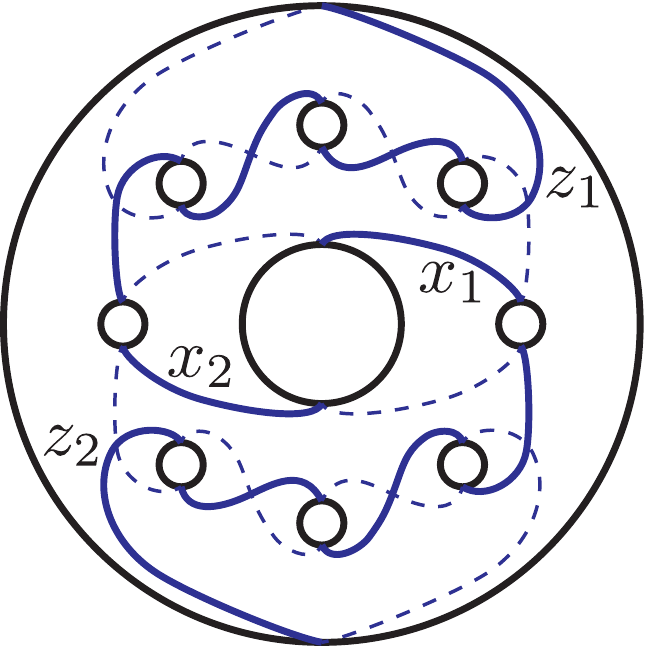}
	\caption{The bounding quadruple.} 
	\label{fig:Genus9_5thBoundingQuadruple}
\end{figure}

As it is the case for any closed oriented $4$--manifold, the remaining homology groups of $X$, as well as $b_2^+(X)$ and $b_2^-(X)$, are already determined by $\eu(X)$, $\sigma(X)$ and $H_1(X)$ by Poincar\'{e} duality and the universal coefficient theorem. In particular, $H_2(X)=\Z^{28} \oplus \Z_4 \oplus \Z_2$ and $b_2^+(X)=b_2^-(X)=14$. Moreover $X$ can be seen to have an even intersection form, which will follow from its spin type.

\smallskip
\noindent{\textit{\underline{Spin type}}:} To build a spin Lefschetz fibration, our heuristic has been to breed exclusively Lefschetz pencils on spin $4$--manifolds, as the monodromy factorizations of such pencils consist of Dehn twists that preserve a spin structure on a compact surface with boundary \cite{BaykurHayanoMonden}. The groundwork for describing spin structures on $X$ was laid out in Section~2, and particularly in Theorem~\ref{SpinLF2}.

From our calculation of the $\Z$--homology classes of the vanishing cycles on $\Sigma_9$ in the symplectic basis $\{\alpha_i, \beta_i\}$ represented by the oriented curves with the same labels in Figure~\ref{fig:Genus9_H1Basis}, we easily deduce the \mbox{$\Z_2$--homology} classes of them. We will single out two spin structures on $X$, restrictions of which to the regular fiber will have different Arf invariants. Let the quadratic form $q_0$ on $H_1(\Sigma_9; \Z_2)$ be defined by
\[q_0(\alpha_i)=q_0(\beta_j)=1 \text{ for all $i, j$ except for } q_0(\beta_9)=0 \,,  \]
and the quadratic form $q_1$ by
\[ q_1(\alpha_i)=q_1(\beta_j)=1 \text{ for all $i, j$ except for }  q_1(\alpha_7)=q_1(\beta_9)=0 \, . \]
Both are easily seen to evaluate as $1$ on all the vanishing cycles. Let $s_0, s_1 \in \mathrm{Spin}(\Sigma_g)$ be the spin structures corresponding to these two forms $q_0, q_1$, which have Arf invariants $\mathrm{Arf}(q_0)=0$ and $\mathrm{Arf}(q_1)=1$, respectively. 

Next, recall that the homology relation $\alpha_2 +\alpha_8 +\beta_4 -\beta_5 +\beta_6 -\beta_7 -\beta_9 =0$ induced by the pseudosection $s$ is already generated by the relations induced by the vanishing cycles we listed during our calculation of $H_1(X)$. So by Proposition~\ref{divisibility}, the homology class of the regular fiber $[F]$ is primitive in $H_2(X)$. As we have $(X,f)$ with a primitive fiber class and signature $\sigma(X) \equiv 0$ (mod $16$), by Theorem~\ref{SpinLF2}, the existence of the quadratic form $q_1$ with $\textrm{Arf}(q_1)=1$ now implies that we have spin structures $\mathfrak{s}_i$ on $X$ coming from the spin structures $s_i$, for each $i=0, 1$. Note that $\mathfrak{s}_0, \mathfrak{s}_1$ are only two of the 512 spin structures on $X$, for $H^1(X; \Z_2)= \Z_2^9$ acts freely and transitively on $\textrm{Spin}(X)$. 
\end{proof}

\medskip
\begin{remark} \label{KodIncreasing}
It might be interesting for symplectic construction enthusiasts to observe how the topology of the pieces we built evolved: Our most elementary building blocks are a genus--$0$ pencil (corresponding to the lantern relation) and genus--$1$ pencils (corresponding to $6$--holed torus and $2$--chain relations) on rational surfaces, breeding of which gives us a genus--$2$ pencil on the ruled surface $T^2 \times S^2$. All have symplectic Kodaira dimension $\kappa=-\infty$. Using copies of the genus--$2$ pencil on the ruled surface, we obtained a symplectic Calabi-Yau $4$--torus, which has $\kappa=0$. We can also re-present our construction of the genus--$9$ fibration so that it is made out of two copies of a genus--$5$ pencil (each one of which is obtained by breeding a pair of our genus--$3$ pencils), along with two more copies of the genus--$3$ pencil. These genus--$5$ pencils have $\kappa=1$.  Lastly, our resulting genus--$9$ fibration $(X,f)$ yields a symplectic $4$--manifold with $\kappa=2$.
\end{remark}

\begin{remark} \label{breeding}
Although we have used the breeding technique to effectively produce symplectic  $4$--manifolds out of smaller symplectic $4$--manifolds, it is not inherently a symplectic operation. One first produces achiral pencils/fibrations on $4$--manifolds that are often not symplectic, and only after pairing the vanishing cycles of all the negative nodes with matching vanishing cycles for positive nodes we can eliminate them to get a symplectic pencil/fibration. In general, the matching pairs would be visible only after a sequence of Hurwitz moves, or equivalently, a sequence of handle slides between the Lefschetz $2$--handles. To approximate a surgery description, one can interpret the breeding construction as a two step process, where one first takes a fiber sum of achiral fibrations, whose monodromies are now supported on a larger surface, and then attaches a $5$--dimensional $3$--handle. (The latter can be seen by looking at the handle diagram of the neighborhood of the two matching nodes; see \cite{BaykurGenus3}.) As observed in  \cite{GompfStipsicz}, when this matching vanishing cycle is trivial in the fundamental group of the complement, eliminating the two Lefschetz $2$--handles is equivalent to removing an $S^2 \times S^2$ or $S^2 \twprod S^2$  summand from a connected sum (and capping off with $D^4$).
\end{remark}

\medskip
\subsection{Lefschetz fibrations with prescribed signatures}  \

We are now ready to prove our main theorem.

\noindent \textit{Proof of Theorem~A.} \,To keep our presentation simple, here we will not attempt to keep the topology of the arbitrary signature examples small, but instead, we will refine our arguments in specific cases, as we will do so in Section~5. In fact, the following bit of information will be enough to run all our arguments:  \textit{There is a signature zero genus--$9$ Lefschetz fibration $(X,f)$, where $X$  is spin and $\pi_1(X)$ has subgroups of  arbitrarily large index.} 
The existence of such a fibration is guaranteed by Theorem~\ref{KeyLF}. We can assume, after a global conjugation, that the Lefschetz fibration $(X,f)$ has a monodromy factorization in  $\M(\Sigma_9)$: 
\begin{equation} \label{simplified}
t_{B_1} \, P_1 = 1 \,  ,  
\end{equation}
where $P_1$ is a product of positive Dehn twists, and $B_1$ is the non-separating curve in Figure~\ref{fig:lanternembedding}. 

\medskip
\noindent \underline{\textit{Any signature}}:\, There are $\phi_1, \phi_2 \in \M(\Sigma_9)$ such that $\phi_1(B_1)=B_2$ and $\phi_2(B_1)=B'_2$, where $B_1, B_2$ and $B'_2$ are the curves shown in Figure~\ref{fig:lanternembedding}. Conjugating the monodromy factorization~\eqref{simplified} with $\phi_i$ for each $i=1,2$, we get \,$t_{B_2} P_2= 1$ and $t_{B'_2} P_3=1$  in $\M(\Sigma_9)$, where $P_i = P_1^{\phi_i}$ is the product of positive Dehn twists along the images of the Dehn twist curves in $P_1$ under $\phi_i$, in the same word order. After cyclic permutations, we get product factorizations \,$t_{B_1} P_1 \cdot P_1 t_{B_1} =1$ and \,$  t_{B_2} P_2 \cdot P_3 t_{B'_2} =1$, and in turn, we get
\begin{equation} \label{lanternsetup}
t_{B_1}^2  t_{B_2}  t_{B'_2} P_1^2 P_2 P_3 = 1 
\end{equation}
in $\M(\Sigma_9)$, where $t_{B_1}, t_{B_2}$ and  $t_{B'_2}$ all commute with each other. The signature of the corresponding Lefschetz fibration, which is nothing but a twisted fiber sum of four copies of the signature zero fibration $(X,f)$, is zero. Since $B_1$, $B'_1$, $B_2$ and $B'_2$ (where $B'_1$ is isotopic to $B_1$) cobound a subsurface $\Sigma_0^4$, we can make a lantern substitution (or equivalently breed with the genus--$0$ pencil on $\CP$) in the  monodromy factorization~\eqref{lanternsetup}, and get
\begin{equation} \label{lanternresult}
t_X t_Y t_Z P_1^2 P_2 P_3  = 1  \,  
\end{equation}
in $\M(\Sigma_9)$,  where $X, Y, Z$ are the images of the standard lantern curves under the embedding of the above $\Sigma_2^2$ into $\Sigma_9$. Since the signature of the lantern relation is $1$, the  factorization~\eqref{lanternresult} prescribes a new Lefschetz fibration $(X_1, f_1)$ with $\sigma(X_1)=1$.

\begin{figure}[htbp]
	\centering
	\subfigure[  \label{fig:lanternembedding}]
	{\includegraphics[height=60pt]{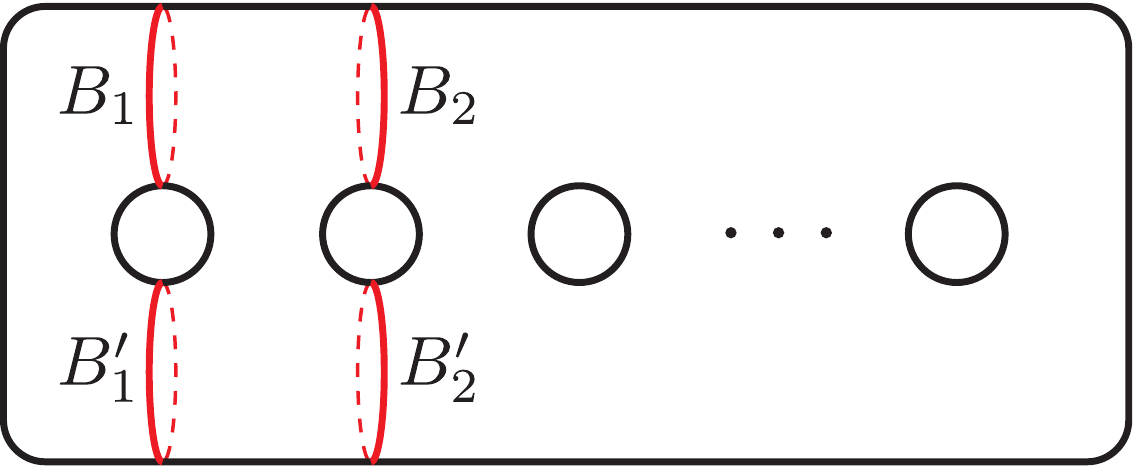}} 
	\hspace{10pt}
	\subfigure[  \label{fig:K3embedding}]
	{\includegraphics[height=60pt]{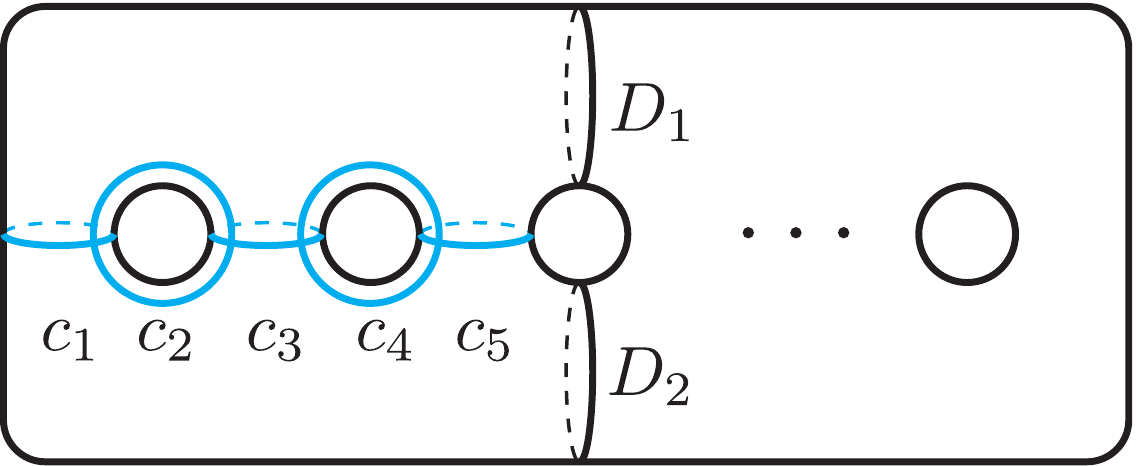}} 
	\caption{Embeddings of lantern and $5$--chain relation curves.} 	
	\label{fig:lanternandK3}
\end{figure}

By taking fiber sum of copies of $(X_1, f_1)$ with copies of any genus--$9$ Lefschetz fibration with a negative signature, we get Lefschetz fibrations $(X_k, f_k)$ with $\sigma(X_k)=k$, for any prescribed $k \in \Z$.  For instance, by taking fiber sum of $k$ copies of $(X_1, f_1)$, we obtain a Lefschetz fibration $(X_k, f_k)$ with $\sigma(X_k)=k$, for any $k \in \Z^+$. If  we employ in this procedure any of the Lefschetz fibrations discussed in Example~\ref{hyperelliptic}, since their vanishing cycles already kill $\pi_1(\Sigma_g)$, we get simply-connected examples, with  prescribed signature.  For example, we can take the fiber sum of  $(X_{41}, f_{41})$ with a genus--$9$ Lefschetz fibration on $\CP \# 41 \CPb$ which has the monodromy factorization~\eqref{hyperellipticrelation}, to get a simply-connected Lefschetz fibration with signature $1$. For later reference, let $(X_k, f_k)$ denote a fixed family of \emph{simply-connected} genus--$9$ Lefschetz fibrations with signature $k \in \Z$ we obtain through these arguments. 

\medskip
\noindent \underline{\textit{Spin with any signature divisible by $16$}}:\, 
The spin structure $\mathfrak{s}$ on $X$ is induced by some spin structure on $\Sigma_9$, which corresponds to a quadratic form on $H_1(\Sigma_9; \Z_2)$ mapping all the Dehn twist curves in the monodromy factorization, and in particular $B_1$, to $1$. 

\begin{figure}[htbp]
	\centering
	\subfigure[ \label{fig:basischanges_a}]
	{\includegraphics[height=60pt]{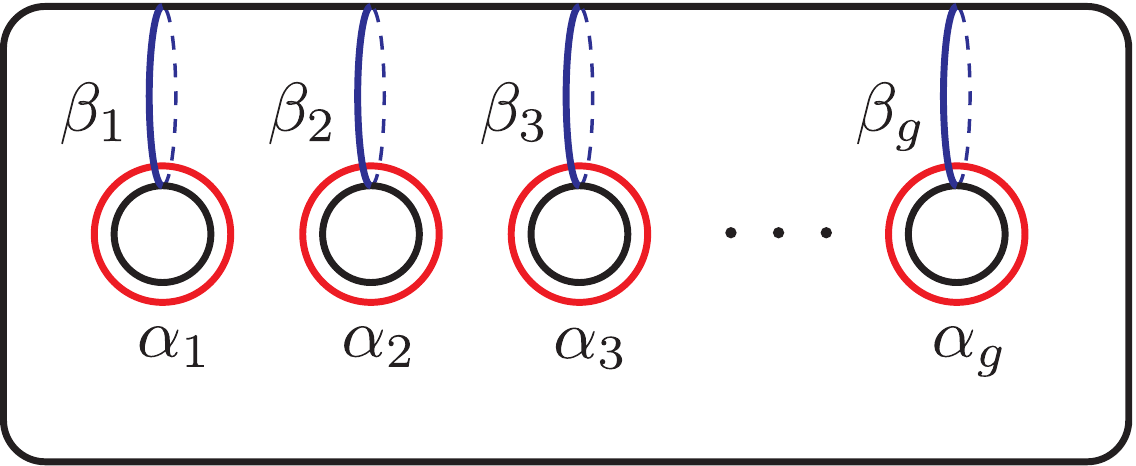}} 
	\hspace{5pt}
	\subfigure[ \label{fig:basischanges_b}]
	{\includegraphics[height=60pt]{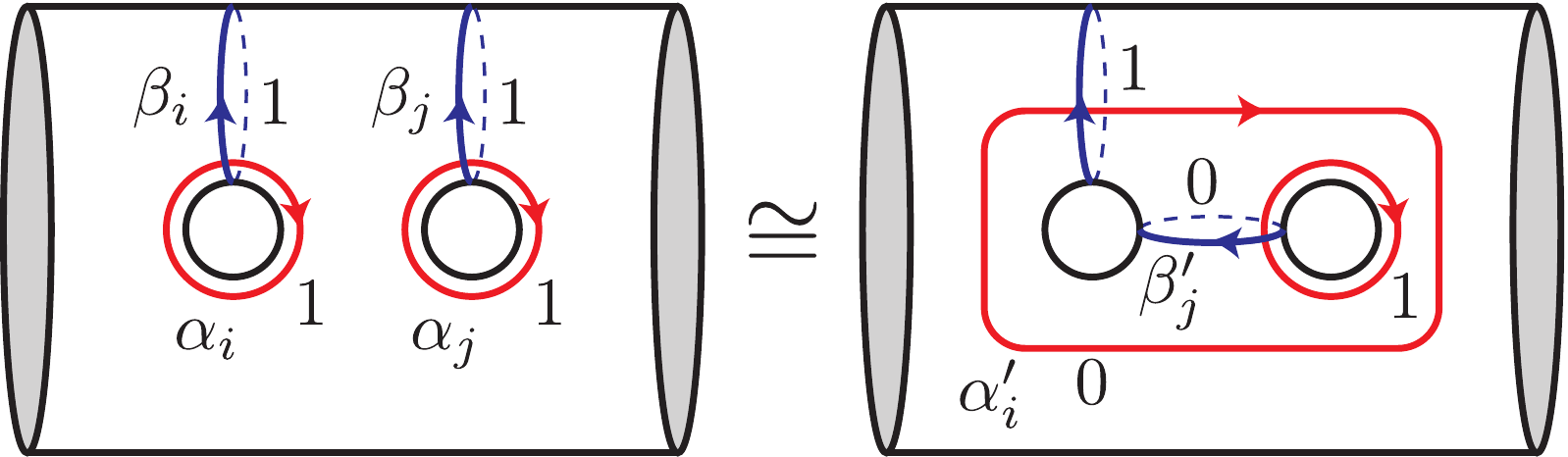}} 
	\subfigure[ \label{fig:basischanges_c}]
	{\includegraphics[height=60pt]{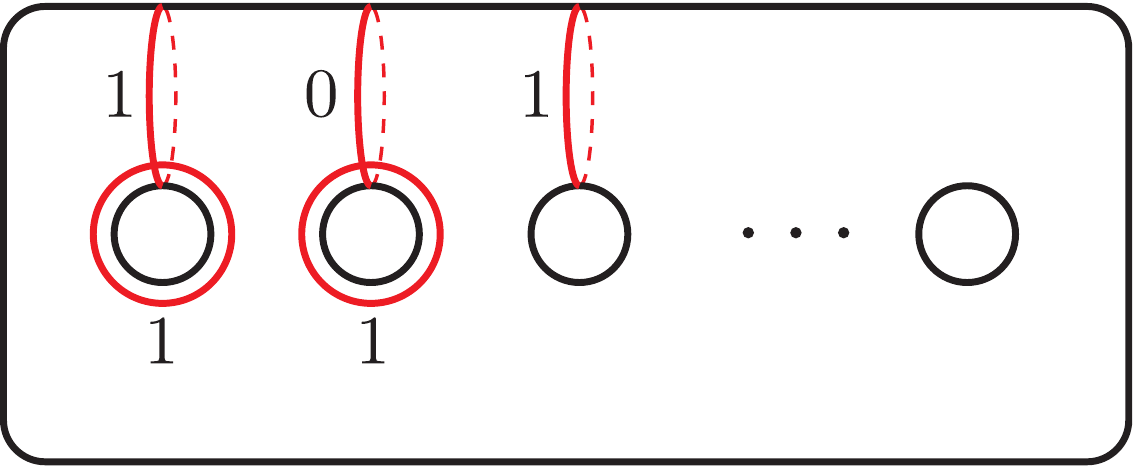}} 
	\caption{Bases changes.} 	
	\label{fig:basischanges}
\end{figure}

Let $\{\alpha_i, \beta_i\}$ be the curves in Figure~\ref{fig:basischanges_a}, which constitute a symplectic basis for $H_1(\Sigma_9; \Z_2)$. 
We claim that, there is a spin structure $s \in \mathrm{Spin}(\Sigma_9)$, where the corresponding quadratic form $q$ is such that
\[ q(\beta_1)=q(\beta_3)=q(\alpha_1)=q(\alpha_2)=1 \text{ and } q(\beta_2)=0 \, , \]
and $q$ sends every vanishing cycle (we will get after a conjugation) to $1$. 

While the values of the remaining basis elements under $q$ will not matter for our arguments, we will rely on the fiber genus  $g$ being at least $3$. We will obtain $s$ as the image, under a diffeomorphism $\varphi$ of $\Sigma_9$,  of the initial spin structure on $\Sigma_9$ yielding the spin structure $\mathfrak{s}$ on $(X,f)$.
(The promised quadratic form $q$ will be then obtained by pulling-back under $\varphi^{-1}$ the quadratic form for the initial spin structure.) We'll require $\varphi$ to fix $B_1$. We will obtain this diffeomorphism $\varphi$ as a composition of several diffeomorphisms of $\Sigma_9$, each one of which realizes a well-known symplectic basis change for $H_1(\Sigma_9; \Z_2)$; cf. \cite{Browder, Johnson}. At each step, we take the new spin structure we get under the diffeomorphism of $\Sigma_9$. Note that every Dehn twist curve  in the monodromy factorization will be conjugated by the diffeomorphisms, and will be mapped to $1$ again under the new quadratic form corresponding to the new spin structure.

Below, we will focus on the values of the ordered pair of elements  $\alpha_i, \beta_i$ in $\Z_2 \times \Z_2$ under the quadratic form $q$, and will simply call them a \emph{$(q(\alpha_i), q(\beta_i))$ pair}.

First of all, note that if $q(\alpha_i)= q(\beta_i)=0$ for some  $i$, we can replace $\alpha_i$ with $\alpha'_i:=t_{\beta_i}(\alpha_i)$ so  that $q(\alpha'_i)=1$, or $\beta_i$ with $\beta'_i:=t_{\alpha_i}(\beta_i)$ so that $q(\beta'_i)=1$. Therefore, any time $q(\alpha_i)\, q(\beta_i)=0$, we can find a pair of geometrically dual curves in a small neighborhood of $\alpha_i \cup \beta_i$, where $q$ maps our pick of one of the  two curves to $1$ and the other one to $0$. So there are diffeomorphisms (which obviously change the spin structure) that trade between a $(0,0)$, a $(1,0)$ or a $(0,1)$ pair locally.

Secondly, suppose $q(\alpha_i)  \, q(\beta_i)=1=q(\alpha_j)\, q(\beta_j)$ for some $i \neq j$. It is easy to see that there is a diffeomorphism supported in a $\Sigma_2^2$ neighborhood of $(\alpha_i \cup \beta_i) \sqcup (\alpha_j \cup \beta_j)$, which replaces $\alpha_i$ and $\beta_j$ with $\alpha'_i$ and $\beta'_j$ shown in Figure~\ref{fig:basischanges_b}, while fixing $\beta_i$ and $\alpha_j$. Here $\alpha'_i, \beta_i$ and $\alpha_j, \beta'_j$ are disjoint geometrically dual pairs contained in $\Sigma_2^2$ as well, but now $q(\alpha'_i)=q(\beta'_j)=0$. Thus, we see that any time we have two disjoint $(1,1)$ pairs, there is a diffeomorphism which replaces them with disjoint $(0,1)$ and $(1,0)$ pairs, and vice versa. 

We can now complete the proof of our claim. All we have initially is that we have a quadratic form which maps the monodromy curves to $1$, and in particular $q(\beta_1)=1$, since $\beta_1=B_1$. Suppose $q(\alpha_1)=0$. Since the genus of the surface $g \geq 3$, it follows from the above observations that, after a diffeomorphism supported away from $\alpha_1 \cup \beta_1$, we can get a $(1,0)$ pair $\alpha_j, \beta_j$. Under the inverse of the diffeomorphism we described in the previous paragraph, we can switch these two disjoint $(0,1)$ and $(1,0)$ pairs with two disjoint $(1,1)$ pairs, all the while fixing $\beta_1$. Relabeling the quadratic form corresponding to the new spin structure as $q$, we now have $q(\alpha_1)=q(\beta_1)=1$. Finally, after another diffeomorphism supported away from the new $\alpha_1 \cup \beta_1$, and  relabeling again the quadratic form for the new spin structure, we can assume that $\alpha_2, \beta_2$ is a $(1,0)$ pair and $q(\beta_3)=1$ as desired; see Figure~\ref{fig:basischanges_c}. This completes the proof of our claim.

Let $c_1, c_2, c_3, c_4, c_5$ be as in Figure~\ref{fig:K3embedding}. For $s\in \textrm{Spin}(\Sigma_9)$ we obtained above, and $q$ the  quadratic form corresponding to $s$, we easily see that now $q(c_i)=1$, so  \mbox{$t_{c_i} \in \M(\Sigma_9, s)$} for all $i$. In particular, $\phi=t_{c_1} t_{c_2} t_{c_3} t_{c_4} t_{c_5} \in \M(\Sigma_9, s)$, where one can easily verify that $\phi(c_i)=c_{i+1}$, for each $i=1, \ldots, 4$.  (Alternatively, for each $i=1, \ldots, 4$, we can take $\psi_i = t_{c_{i}} t_{c_{i+1}}  \in \M(\Sigma_9, s)$ so that $\psi_i(c_i)=c_{i+1}$.)  After  conjugating the positive factorization~\eqref{simplified} with the diffeomorphism $\varphi$ we described above, we can further conjugate the resulting positive factorization with powers of $\phi$, and get five positive factorizations \,  $Q_i:=t_{c_{i}} P_1^{\phi^{i-1} \varphi} = 1$, for $i=1, \ldots, 5$. (Recall that $B_1=c_1$.) 
Next, take the product factorization $(Q_1 Q_2 Q_3 Q_4 Q_5)^6 =1 $, which, after conjugating away the products of positive Dehn twists $P_1^{\phi_{i-1}\varphi}$ in $Q_i$ (i.e. after the corresponding Hurwitz moves), yields a positive factorization
\begin{equation}\label{5chainsetup}
(t_{c_1} t_{c_2} t_{c_3} t_{c_4} t_{c_5})^6  \, Q = 1 \,  
\end{equation}
in $\M(\Sigma_9)$, where $Q$ is the product of all the remaining Dehn twists conjugated away from $t_{c_i}$. This is in fact a factorization in  $\M(\Sigma_9, s)$. Because it is a product of signature zero relations, the signature of this last relation is also zero.

Now, by substituting the \emph{inverse} of the $5$--chain relation in the factorization~\eqref{5chainsetup}, or equivalently breeding with the achiral counter-part of the genus--$2$ pencil on the $\K$ surface given in Example~\ref{oddchain} (which we do so by embedding the relation into the subsurface $\Sigma_2^2$ that is a neighborhood of the chain $c_1, \ldots, c_5$), we get a new positive factorization
\begin{equation}\label{5chainresult}
R:=t_{D_1} t_{D_2} \, Q = 1 \, 
\end{equation}
in $\M(\Sigma_9)$. It is evident from Figure~\ref{fig:basischanges_c} that $D_1=\beta_3$ and $D_2$ is homologous to it, so $q(D_1)=q(D_2)=1$. Thus, this too is in fact a factorization in $\M(\Sigma_9, s)$, i.e. the quadratic form $q$ maps all the Dehn twist curves in the factorization~\eqref{5chainresult} to $1$. Following our heuristic, we have once again bred with an achiral pencil on a spin $4$--manifold. This new relation has signature $16$, since we had a signature zero factorization, and inserted the inverse $5$--chain relation, which has signature $16$ \cite{EndoNagami} (it is the signature of the $\K$--surface with the opposite orientation after all).

To be able to apply Theorem~\ref{SpinLF2} and equip the new fibration with a spin structure, we still need to confirm that it also has a primitive fiber class. While this  is indeed the case for the fibration corresponding to the factorization~\eqref{5chainsetup}, which is nothing but a twisted fiber sum of several copies of $(X,f)$, the substitution we have made above requires a new calculation for the pseudosection. There is an easier way to get what we want, which we will discuss next, in two distinct cases:

Suppose $\textrm{Arf}(s)=1$. Let $H=1$ be the positive factorization~\eqref{hyperellipticrelation} in Example~\ref{hyperelliptic} for a genus--$9$ fibration on $\CP \# 41 \CPb$. Recall that the unique quadratic form on $H_1(\Sigma_9; \Z_2)$ that maps all the Dehn twist curves in $H$ to $1$, also has Arf invariant $1$.  So, there is some diffeomorphism $\psi$ of $\Sigma_9$ sending the spin structure corresponding to that quadratic form to the spin structure $s$ we had. Now take the product positive factorization
\begin{equation}\label{5chainresultA}
R^6 \, (H^2)^\psi =1 
\end{equation}
in $\M(\Sigma_9)$, which is in fact a factorization in $\M(\Sigma_9, s)$. It also
has signature $6\cdot 16 +2 (-40)= 16$. Because the Dehn twist curves in $H$ already kill $\pi_1(\Sigma_9)$, by the particular case of Proposition~\ref{divisibility}, the fiber class of the Lefschetz fibration $(Z_1, h_1)$, prescribed by the positive factorization~\eqref{5chainresultA}, is primitive. So  $Z_1$ is spin by Theorem~\ref{SpinLF2}.

Suppose  $\textrm{Arf}(s)=0$. Consider the positive 
factorization~\eqref{yunrelation} in Example~\ref{yun}. For $n=8$ and $m=1$, this prescribes a genus--$9$ Lefschetz fibration on $(T^2 \times S^2) \# \, 32 \CPb$.
While there are four quadratic forms on $H_1(\Sigma_9; \Z_2)$ which map all the the Dehn twist curves in the factorization~\eqref{yunrelation} to $1$, they all have Arf invariant $0$. 
As explained in Example~\ref{yun}, there is a twisted fiber sum of two copies of this fibration, which gives a spin genus--$9$ Lefschetz fibration with a section on the knot surgered $E(8)_K$, where $K$ is any genus--$1$ fibered knot, say a trefoil. Let \,$Y=1$ be a positive factorization corresponding to this latter fibration. So it has signature $-64$, and since the fibration has a section, $\pi_1(\Sigma_9) \, / \, N = \pi_1(E(8)_K)=1$, where $N$ is the subgroup of $\pi_1(\Sigma_9)$ normally generated by the Dehn twist curves in $Y$. So the Dehn twist curves in the positive factorization $Y$ already kill $\pi_1(\Sigma_g)$. (In fact, here we chose $m=1$ so one can also see directly how to twist the fiber sum to get a spin fibration with trivial $\pi_1$ as desired.) The unique quadratic form on $H_1(\Sigma_9; \Z_2)$, which maps all the monodromy curves to $1$, necessarily has Arf invariant $0$. It follows that there is again some diffeomorphism $\psi$ of $\Sigma_9$ sending the spin structure corresponding to that quadratic form to the spin structure $s$ we had. Now take the product positive factorization
\begin{equation}\label{5chainresultB}
R^5 \, Y^{\psi} =1  \, ,
\end{equation}
in $\M(\Sigma_9)$, which is in fact a factorization in $\M(\Sigma_9, s)$. It, too
has signature $5\cdot 16 +(-64)= 16$. Because the Dehn twist curves in $Y$ already kill $\pi_1(\Sigma_9)$, once again, by the particular case of Proposition~\ref{divisibility} and Theorem~\ref{SpinLF2},  for $(Z_1, h_1)$ denoting the genus--$9$ Lefschetz fibration prescribed by the positive factorization~\eqref{5chainresultB}, we conclude that $Z_1$ is spin. 

Note that either one of the two cases, with $\textrm{Arf}(s)=1$ or $0$, may occur depending on the initial choice of \mbox{$s \in \textrm{Spin}(\Sigma_9)$} which yields the spin structure on $(X,f)$. (Indeed, the signature zero genus--$9$ Lefschetz fibration $(X,f)$ of Theorem~\ref{KeyLF} admits spin structures of both types.) Also note that, in either case, the vanishing cycles of $(Z_1, h_1)$ we constructed kill $\pi_1(\Sigma_9)$. 

To finish our proof, take a fiber sum of $k$ copies of $(Z_1, h_1)$ to produce a genus--$9$ Lefschetz fibration $(Z_k, h_k)$ with  $\sigma(Z_k)=16k$, for any prescribed $k \in \Z^+$. By Proposition~\ref{divisibility} and Theorem~\ref{SpinLF2}, all $Z_k$ are spin, and since the vanishing cycles of even one copy already kill $\pi_1(\Sigma_9)$,  all $Z_k$ are simply-connected. Negative signature examples are already realized by fiber sums of  Lefschetz fibrations on knot surgered \mbox{$\K$--surfaces} discussed in Example~\ref{yun}, for any genus--$4$ fibered knot. Finally, to get \emph{simply-connected}  examples of signature zero, one can simply lower the power of $R$ by one in the monodromy factorizations~\eqref{5chainresultA} and~\eqref{5chainresultB}.  For later reference, let $(Z_k, h_k)$ denote a fixed family of \emph{simply-connected} genus--$9$ Lefschetz fibrations with $Z_k$ spin and $\sigma(Z_k)=16k$, for  $k \in \Z$.

\smallskip
\noindent \underline{\textit{Higher genera}}:\, For any index $d$ subgroup of $\pi_1(X)$, there exists a $d$--fold cover $\pi_d\colon \tilde{X}(d) \to X$. The composition $\tilde{f}(d):=f \circ \pi_d \colon \tilde{X}(d) \to S^2$ yields a Lefschetz fibration on $\tilde{X}(d)$ of genus $g=8d+1$, and since the signature  is multiplicative under unbranched coverings, $(\tilde{X}(d), \tilde{f}(d))$ too is a signature zero Lefschetz fibration.  We can thus follow the same construction scheme we had above, now for genus $g=8d+1$ fibrations. 

For arbitrary signatures, take a twisted fiber sum of four copies of $(\tilde{X}(d), \tilde{f}(d))$ so that we get Dehn twist curves cobounding a subsurface $\Sigma_0^4$, and then make a lantern substitution to get a genus $g=8d+1$ Lefschetz fibration $(\tilde{X}_1(d), \tilde{f}_1(d))$ with signature $1$. Then, by taking fiber sums of copies of  $(\tilde{X}_1(d), \tilde{f}_1(d))$ and copies of any genus--$g$ Lefschetz fibration with a negative signature, we can build $(\tilde{X}_k(d), \tilde{f}_k(d))$ with $\sigma(\tilde{X}_k(d))=k$, for any given $k \in \Z$. Once again, picking the negative signature summands as simply-connected ones, like the ones in Example~\ref{hyperelliptic}, we can assume that $\tilde{X}_k(d)$ are all simply-connected. 

For spin examples, first note that  if $\pi \colon \tilde{Y} \to Y$ is a finite covering of a spin $4$--manifold $Y$, then $\tilde{Y}$ is also spin: The tangent bundle $T\tilde{Y}$ is isomorphic to the pull-back bundle $\pi^* (TY)$, and by functoriality, the second Stiefel-Whitney class of $\tilde{Y}$ is $w_2(T\tilde{Y})=\pi^*w_2(TY)$. So  for \mbox{$w_2(Y)=w_2(TY)=0$} in $H^2(Y; \Z_2)$, we have $w_2(\tilde{Y})=w_2(T\tilde{Y})=0$ in $H^2(\tilde{Y}; \Z_2)$ as well. Since $w_2$ is the only obstruction to the existence of a spin structure on an orientable manifold, the signature zero Lefschetz fibrations $(\tilde{X}(d), \tilde{f}(d))$ of genus $g=8d+1$ we described above are all spin. It follows that there is a spin structure on $\Sigma_g$ with a quadratic form that evaluates as $1$ on all the Dehn twist curves in a monodromy factorization of $(\tilde{X}(d), \tilde{f}(d))$. As the arguments we gave in the genus $g=9$ case applies just the same to any other genus $g \geq 3$, we can then build a twisted fiber sum of copies of $(\tilde{X}(d), \tilde{f}(d))$, with a monodromy factorization of the form $(t_{c_1} t_{c_2} t_{c_3} t_{c_4} t_{c_5})^6  \, Q = 1 \, $ in $\M(\Sigma_g, s)$, for some spin structure $s$. Once again, by inverse $5$--chain substitution, we get a new genus--$g$ Lefschetz fibration, with a  monodromy factorization in $\M(\Sigma_g, s)$. By fiber summing this fibration with simply-connected, spin genus--$g$ fibrations, which come with a quadratic form on $H_1(\Sigma_g; \Z_2)$ with the same Arf invariant as  $q$, we obtain a simply-connected, spin genus--$g$ Lefschetz fibration $(\tilde{Z}_1(d), \tilde{h}_1(d))$, with $\sigma(\tilde{Z}_1(d))=16$. For this, we observe that for any $g=8d+1$, there exists a spin genus--$g$ Lefschetz fibrations among the ones in Examples~\ref{yun}, where the quadratic form that evaluates as $1$ on all the monodromy curves has the desired Arf invariant. Finally, by taking further sums as explained in the genus--$9$ case, we can get simply-connected, spin genus--$g$ Lefschetz fibrations $(\tilde{Z}_k(d), \tilde{h}_k(d))$ with $\sigma(\tilde{Z}_k(d))=16k$, for any prescribed $k \in \Z^+$, and in turn examples with any signature $16k$, for $k \in \Z$. 

By the initial assumption on $\pi_1(X)$, we can take the covering degree $d$ to be arbitrarily large, and obtain the above examples with arbitrarily high genera. 

\smallskip
\noindent \underline{\textit{Other properties}}:\,  We equip all our examples with a Gompf-Thurston symplectic form so they become symplectic Lefschetz fibrations.
Set $(\tilde{X}_0(1), \tilde{f}_0(1))=(X_0, f_0)$ and $(\tilde{Z}_0(1), \tilde{h}_0(1))=(Z_0, h_0)$, which are  the simply-connected and signature zero examples for genus $g=9$ where the latter is spin. Further fiber sums with  copies of the simply-connected, signature zero genus $g=8d+1$ Lefschetz fibration $(\tilde{X}_0(d), \tilde{f}_0(d))$, or in the spin case, with copies of $(\tilde{Z}_0(d), \tilde{h}_0(d))$, we get an infinite family of simply-connected examples. Since fiber sums of Lefschetz fibrations of genus $g\geq 1$ are always minimal \cite{Usher} (also see \cite{BaykurMinimality}), these symplectic $4$--manifolds are minimal. 

This completes the proof of Theorem~A. $\QED$

\medskip
\begin{remark}
While all the examples of arbitrary signature Lefschetz fibrations over the $2$--sphere we have presented in this article are of fiber genus $g \equiv 1$ (mod $8$), we do not have any reasons to believe that the gaps in values of $g$ are essential. 
On the other hand, it is interesting to determine the smallest $g$ for a genus--$g$ Lefschetz fibration with positive signature, and even more so, with zero signature (as it should  be evident from our proof above, one can then generate examples with any other signatures). For the restricted class of hyperelliptic Lefschetz fibrations, one can derive upper bounds using Endo's signature formula \cite{Endo}, and in particular, any genus $g=1$ or $2$ fibration is known to have  negative signature; also see \cite{Ozbagci}. Therefore, the question really is: \

\smallskip
\noindent \textit{Are there genus--$g$ Lefschetz fibrations over the $2$--sphere with arbitrary signatures for $3 \leq g \leq 8$?}  \

\smallskip

Note that if one asks the analogous question for Lefschetz fibrations over the $2$--disk instead, \emph{with no restriction on their global monodromy}, then it is easier to generate positive signature examples, and there is a satisfactory answer that follows from the work of Cengel and Karakurt in  \cite{CengelKarakurt}, where we see that $g=1$ suffices in this case. 
\end{remark}

\begin{remark} \label{SignatureHistory}
It was conjectured by Stipsicz that all symplectic Lefschetz fibrations over the \mbox{$2$--sphere} had negative signatures  \cite{OzbagciThesis, StipsiczTalk}, and this constituted an open problem for over  20 years; see e.g.  \cite{OzbagciThesis, Ozbagci, StipsiczTalk}, \cite[Problem 6.3]{StipsiczSurvey}, and \cite[Problems 7.4]{KorkmazStipsicz}. Besides the lack of examples with positive signatures in the literature, as far as we know, not all negative values were known to be realized as the signature of Lefschetz fibrations over the $2$--sphere either ---especially for  fixed fiber genus. Since the signature is additive under fiber sums, the gaps in the latter case were essentially due to lack of examples with signatures close to zero. Prior to our work, the largest known signatures in the literature we know of had signature $\sigma=-4$ (realized by the examples of Matsumoto, Cadavid and Korkmaz for every even $g\geq 2$,  and by the examples of the first author in \cite{BaykurGenus3} for every odd $g \geq 3$), with the exception of $\sigma=-3$ in the $g=2$ case (realized by the examples of Xiao \cite{Xiao} and Baykur-Korkmaz \cite{BaykurKorkmaz}). 
\end{remark}

\begin{remark} \label{Holomorphic}
We expect most of our examples to be non-holomorphic. The main building blocks, and many non-simply-connected examples we can produce, have odd first Betti numbers, and thus their total spaces are not even homotopy equivalent to compact complex surfaces. It would be interesting to know to what extend an analogue of Theorem~A holds in the holomorphic category:

\smallskip
\noindent \textit{Are there holomorphic Lefschetz fibrations over $\CPo$ with arbitrary signatures?}  \

\smallskip
Many examples with various negative signatures appear in the works of Persson, Peter, Xiao among many others. It is claimed in \cite{Park} that some of the holomorphic fibrations on positive signature compact complex surfaces described in   \cite{PerssonEtal} are Lefschetz, but the authors' construction in the positive index case appears to always involve multiple fibers. (Existence of multiple fibers have no bearing on the rest of the arguments of \cite{Park}.) In fact the authors  successfully build their \emph{negative} signature examples  as Lefschetz fibrations, and use this to argue the simple-connectivity of their complex surfaces, whereas for their positive signature examples, they take a detour and use different arguments to kill the fundamental group \cite{PerssonEtal}.
\end{remark}


\smallskip
\section{Proofs of Corollary~C and Theorem~D}  

In this final section we focus on the signature zero case and apply our techniques to present novel constructions of symplectic $4$--manifolds homeomorphic but not diffeomorphic to \mbox{$\#_m (S^2 \times S^2)$,} the connected sum of $m$ copies of $S^2 \times S^2$. Here $m$ is necessarily odd since the holomorphic Euler characteristic of any almost complex \mbox{$4$--manifold} is an integer, implying that the Betti numbers of a symplectic $4$--manifold satisfies the equality \,$1-b_1+b_2^+ \equiv 0$ (mod 2). We will first produce such  examples as Lefschetz fibrations over the $2$--sphere and prove Corollary~C. We will then produce examples with smaller topology by applying symplectic surgeries to certain Lefschetz fibrations over the $2$--torus and prove Theorem~D.

\subsection{Exotic $\#_m  (S^2 \times S^2)$ as symplectic Lefschetz fibrations }  \

Any example of simply-connected, spin, signature zero symplectic Lefschetz fibration granted by Theorem~A  is necessarily homeomorphic to $\#_m (S^2 \times S^2)$ (for some large $m$), by Freedman  \cite{Freedman}. Whereas by the works of Taubes \cite{Taubes1, Taubes2}, no symplectic $4$--manifold with $b_2^+>1$ can be diffeomorphic to such a connected sum,  thus it would be an exotic  $\#_m  (S^2 \times S^2)$.  Corollary~C to our main theorem promises two refinements:  we can get explicit examples for $m=127$ and also for \emph{every} odd  $m \geq 415$. Although we are going to produce all our examples with fiber genus $g=9$, one can adopt the construction scheme below to generate higher genera examples as well (for different values of $m$), as we did so in the proof of Theorem~A.
 
\medskip
\noindent \textit{Proof of Corollary~C}. \,
Let $(X,f)$ be the signature zero  genus--$9$ Lefschetz fibration of Theorem~\ref{KeyLF}, equipped with the spin structure $\mathfrak{s}_0$. Recall that $\mathfrak{s}_0$ is induced by $s_0 \in \textrm{Spin}(\Sigma_9)$ with a quadratic form $q_0$ that evaluates as $1$ on all the basis elements $\alpha_i, \beta_i$ but $\beta_9$ given  in Figure~\ref{fig:Genus9_H1Basis}. 

For an easier presentation of our arguments, let us first note that there is a diffeomorphism of $\Sigma_9$, which fixes the spin structure $s_0$ and maps the bounding quadruple $(z_1, x_1, z_2, x_2)$ (see Figure~\ref{fig:Genus9_5thBoundingQuadruple}) in the monodromy factorization~\eqref{eq:Signature0LF} of $(X,f)$, to the quadruple $(\beta_1,\beta_1^\prime,\beta_5, \beta_5^\prime)$ shown in Figure~\ref{fig:killingpi1}(a). For instance, we can take the diffeomorphism
\begin{align*}
\phi_1= t_{6}^{-1} t_{8}^{-1} t_{7}^{-1} t_{6}^{-1} t_{\check{5}}t_{\hat{5}} t_{\check{4}}t_{\hat{4}} t_{\check{1}}t_{\hat{1}} t_{\check{5}}t_{\hat{5}} t_{\check{4}}t_{\hat{4}} t_{\check{3}}t_{\hat{3}} t_{\check{2}}t_{\hat{2}} t_{\check{1}}t_{\hat{1}}
\end{align*}
where the Dehn twist curves are as shown in Figure~\ref{fig:Genus9_SpinDiffeo}.
One can easily verify that $\phi_1 \in \M(\Sigma_9, s_0)$ by looking at its action on the symplectic basis elements  in Figure~\ref{fig:Genus9_H1Basis}. (Here not every Dehn twist we employed is in $\M(\Sigma_g, s_0)$, but their product $\phi_1$ is.) After a global conjugation with $\phi_1$ and Hurwitz moves, we get a monodromy factorization for $(X,f)$ of the form
\begin{equation} \label{monodwithquad}
t_{\beta_1} t_{\beta'_1} t_{\beta_5} t_{\beta'_5} \, P_1 = 1 \, \text{ in } \M(\Sigma_9) 
\end{equation}
where $P_1$ is the product of the remaining Dehn twists. This is in fact a factorization in $\M(\Sigma_9, s_0)$.

\begin{figure}[htbp]
	\centering
	\includegraphics[height=100pt]{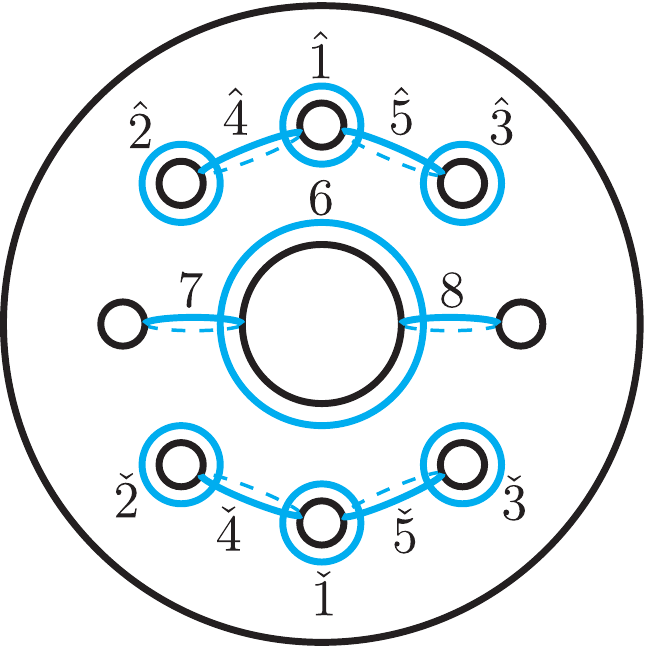}
	\caption{The Dehn twist curves for the spin diffeomorphism $\phi_1$.} 
	\label{fig:Genus9_SpinDiffeo}
\end{figure}

\noindent \underline{\textit{Simple-connectivity}}:
We will first show how to get simply-connected fibrations. For $\alpha_i, \beta_j,\beta'_j$ as in Figure~\ref{fig:killingpi1}(a), we have $q_0(\alpha_i)=q_0(\beta_j)=q_0(\beta'_j)=1$ for all $i=1, \ldots, 9$ and $j=1, \ldots, 8$. It follows that the Dehn twists $t_{\alpha_i}$, $t_{\beta_j}$, $t_{\beta'_j}$ are in $\M(\Sigma_9, s_0)$ for all $i$ and $j$ as above. Also note that the clockwise $(\frac{\pi}{4}$)--rotation $\rho$ about the center of the figure and the involution $\iota$ about the $y$--axis as illustrated in Figure~\ref{fig:killingpi1}(a) both preserve $s_0$. It is now easy to see that there are diffeomorphisms $\phi_k \in \M(\Sigma_9, s_0)$ for each $k=2, \ldots, 6$ which takes the curves $\beta_1, \beta'_1, \beta_5$ to the basis elements $\alpha_i, \beta_i$ as shown in \mbox{Figure~\ref{fig:killingpi1}(b1)--(b6).} For instance we can take
\begin{align*}
\phi_2 &=  t_{\alpha_7} t_{\beta_7}   t_{\beta'_7} t_{\alpha_7}  t_{\alpha_9} t_{\beta'_7}    t_{\beta'_2} t_{\alpha_9} \, \rho \\
\phi_3 &= \iota \phi_2 \\
\phi_4 &=  t_{\beta_1} t_{\alpha_1}  t_{\beta'_1} t_{\alpha_9}  t_{\beta_5} t_{\alpha_5} \\
\phi_5 &=   t_{\beta_7} t_{\alpha_7}  t_{\beta_6} t_{\alpha_6}  t_{\beta_2} t_{\alpha_2} \phi_2 \\
\phi_6 &=   t_{\beta_8} t_{\alpha_8}  t_{\beta_4} t_{\alpha_4}  t_{\beta_3} t_{\alpha_3} \phi_3 .
\end{align*}
Homotoping the curves $\alpha_i, \beta_i$, $i=1, \ldots, 9$ to share a common base point, and orienting them, we get a basis for $\pi_1(\Sigma_9)$. So the curves $\{\alpha_i, \beta_i\}_{i=1}^9$ normally generate $\pi_1(\Sigma_9)$. In turn, we see that $\beta_1, \beta'_1, \beta_5$, and their images under $\phi_k$, $k=2, \ldots 6$, all together normally generate $\pi_1(\Sigma_9)$. \linebreak (Here $\beta_9$ is normally generated by $\beta_1$ and $\beta'_1$.) 

\begin{figure}[htbp]
	\centering
	\begin{tabular}{cc}
	\multirow{3}{*}{
	\subfigure[ \label{fig:killingpi1_a}]
	{\includegraphics[height=185pt]{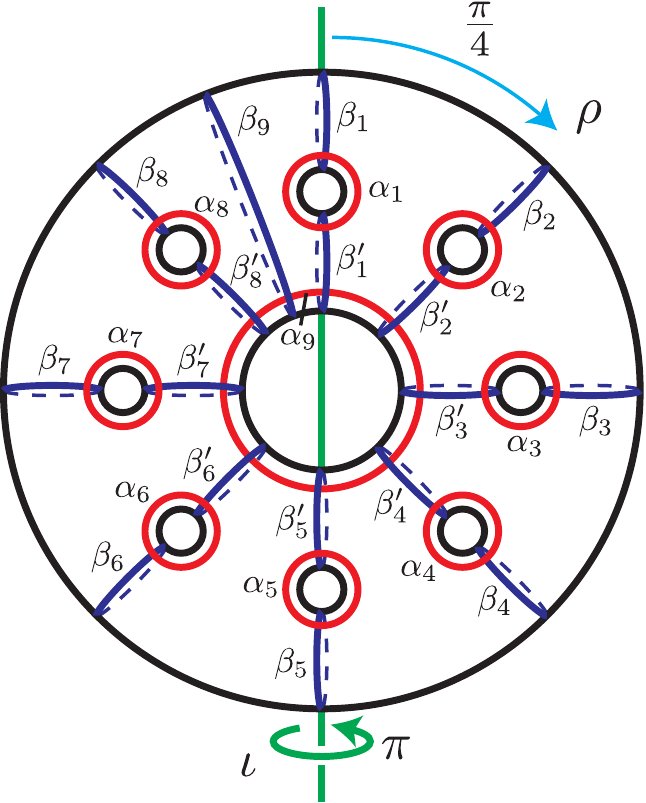}} 
	}
	& \\
	&
	\renewcommand*\thesubfigure{}
	\begin{minipage}{0.18\hsize}
		\subfigure[(b1) \label{fig:killingpi1_b1}]
		{\includegraphics[height=75pt]{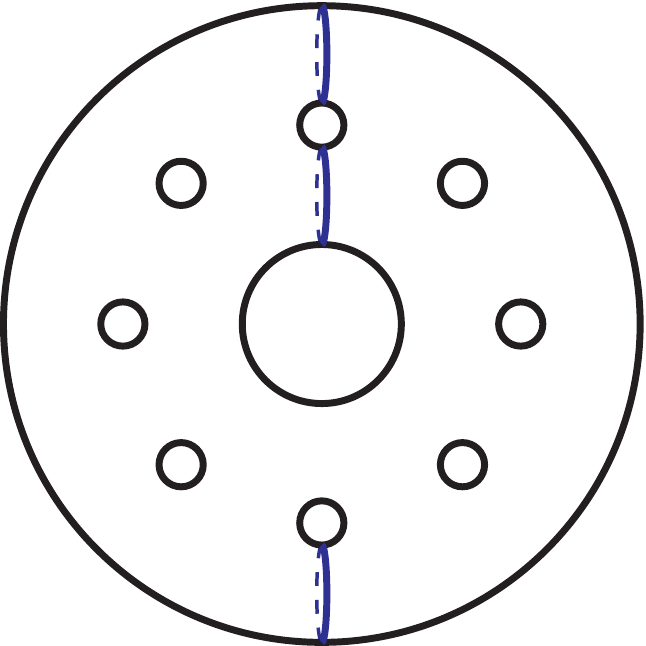}} 
	\end{minipage}
	\begin{minipage}{0.18\hsize}
		\subfigure[(b2) \label{fig:killingpi1_b2}]
		{\includegraphics[height=75pt]{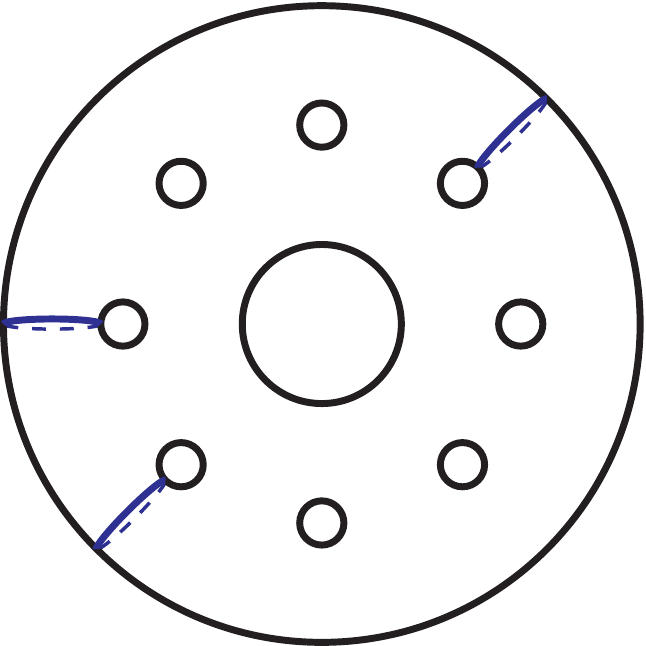}} 
	\end{minipage}
	\begin{minipage}{0.18\hsize}
		\subfigure[(b3) \label{fig:killingpi1_b3}]
		{\includegraphics[height=75pt]{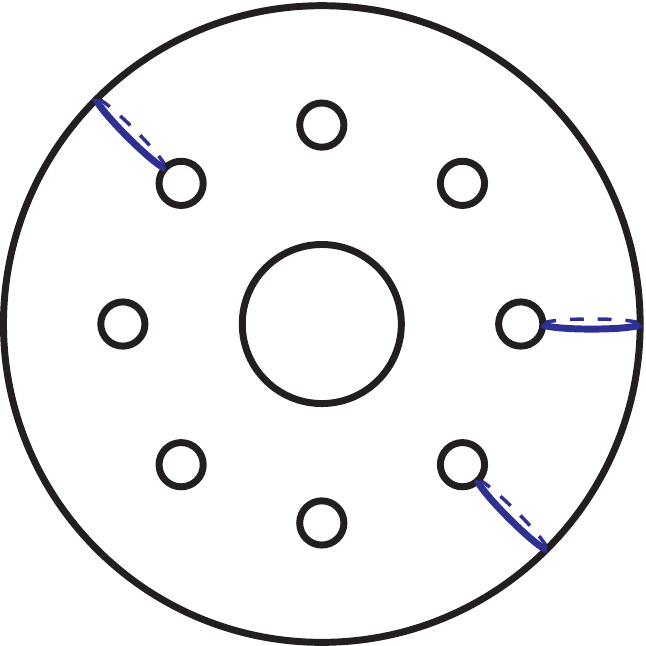}} 
	\end{minipage}\\
	&
	\renewcommand*\thesubfigure{}
	\begin{minipage}{0.18\hsize}
		\subfigure[(b4) \label{fig:killingpi1_b4}]
		{\includegraphics[height=75pt]{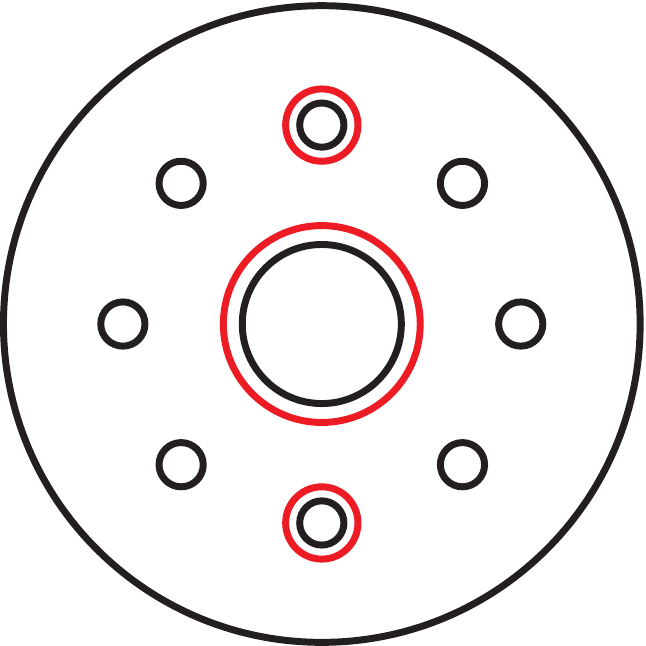}} 
	\end{minipage}
	\begin{minipage}{0.18\hsize}
		\subfigure[(b5) \label{fig:killingpi1_b5}]
		{\includegraphics[height=75pt]{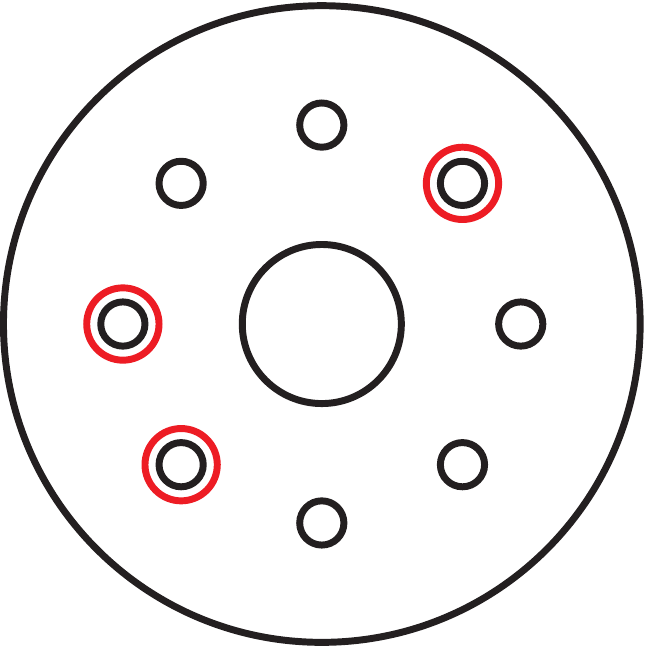}} 
	\end{minipage}
	\begin{minipage}{0.18\hsize}
		\subfigure[(b6) \label{fig:killingpi1_b6}]
		{\includegraphics[height=75pt]{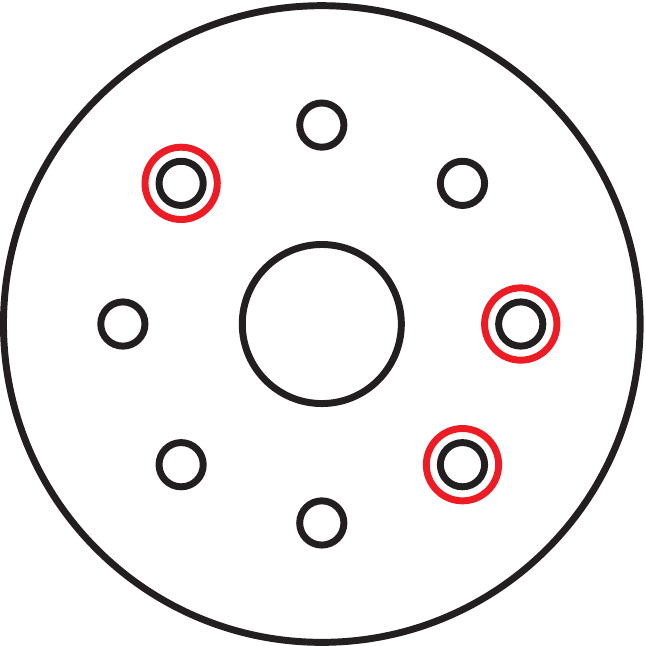}} 
	\end{minipage}
	\end{tabular}
	\caption{On the left: The curves $\alpha_i, \beta_i, \beta'_i$ and diffeomorphisms $\rho$ and $\iota$ of $\Sigma_9$. On the right: $\beta_1, \beta'_1, \beta_5$ and their images under the diffeomorphisms $\phi_k$.} 	
	\label{fig:killingpi1}
\end{figure}

Thus the product monodromy factorization in $ \M(\Sigma_9, s_0)$
\begin{equation} \label{trivialpi1}
t_{\beta_1} t_{\beta'_1} t_{\beta_5}  \, Q_1 \,
t_{\beta_2} t_{\beta_6} t_{\beta_7}  \, Q_2  \,
t_{\beta_3} t_{\beta_4} t_{\beta_8}  \, Q_3 \,
t_{\alpha_1} t_{\alpha_5} t_{\alpha_9}  \, Q_4 \,
t_{\alpha_2} t_{\alpha_6} t_{\alpha_7}  \, Q_5 \,
t_{\alpha_3} t_{\alpha_4} t_{\alpha_8}  \, Q_6 =1  \, ,
\end{equation}
which prescribes a twisted fiber sum of six copies of $(X,f)$, yields a simply-connected symplectic Lefschetz fibration $(Z,h)$ admitting a spin structure coming from $s_0 \in \textrm{Spin}(\Sigma_9)$ with quadratic form $q_0$. Here $Q_1= t_{\beta'_5} P_1$ and $Q_i= Q_1^{\phi_k}$ for $k=2, \ldots, 6$.  
A straightforward Euler characteristic calculation shows that the simply-connected spin, signature zero $4$--manifold $Z$ has the same Euler characteristic as $\#_{127}  (S^2 \times S^2)$. Since both are smoothable simply-connected $4$--manifolds, and have even intersection forms of the same rank, by Freedman, they are homeomorphic.

\medskip
\noindent \underline{\textit{Stable range}}:
Next we will show that we can get symplectic Lefschetz fibrations homeomorphic to $\#_m  (S^2 \times S^2)$ for any odd $m \geq 415$. 
Consider the factorization~\eqref{monodwithquad}, which is $t_{\beta_1} t_{\beta'_1} t_{\beta_5} t_{\beta'_5} \, P_1 = 1$, and its global conjugation by $\rho$, which reads
$(t_{\beta_1} t_{\beta'_1} t_{\beta_5} t_{\beta'_5} \, P_1)^{\rho} = t_{\beta_2} t_{\beta'_2} t_{\beta_6} t_{\beta'_6} \, (P_1)^{\rho} = 1$.
By taking a product of six copies of each, and then applying Hurwitz moves, we obtain
\begin{equation}
(t_{\beta_1} t_{\beta'_1} t_{\beta_6} t_{\beta'_6})^6 
(t_{\beta_2} t_{\beta'_2} t_{\beta_5} t_{\beta'_5})^6
R=1
\end{equation}
where $R$ is the product of the remaining Dehn twists.
This is a monodromy in $\M(\Sigma_9, s_0)$ for a twisted fiber sum of $12$ copies of $(X,f)$ that we will denote by $(Z', h')$.

In this new factorization, the quadruples $\beta_1, \beta'_1, \beta_6, \beta'_6$ and $\beta_2, \beta'_2, \beta_5, \beta'_5$ both bound copies of $\Sigma_2^2$ in $\Sigma_9$. 
We can thus embed two copies of the monodromy factorization~\eqref{eq:MatsumotoLP_IIA_MarkedPoint} (forgetting the marked point) we had for our signature zero spin genus--$2$ pencil, and cancel out the boundary-twists against the Dehn twists along these bounding quadruples. It is time to remember that the spin structure we described on this genus--$2$ pencil admits a quadratic form which evaluates as $1$ on each one of the homology basis elements we described in~\eqref{Genus2SpinStructure}. We can thus embed this relation so that the symplectic pairs are mapped to $\alpha_7, \beta_7, \alpha_8, \beta_8$ and to  $\alpha_3, \beta_3, \alpha_4, \beta_4$, which in turn guarantees that the new monodromy curves are all mapped to $1$ under the quadratic form $q_0$ we had. Repeating this for the other five subfactorizations, we can embed the genus--$2$ pencil monodromy a total number of $12$ times into the product monodromy we had for $(Z',h')$. Each time we breed with the genus--$2$ pencil, we get $4$ more vanishing cycles, and after $12$ breedings, we get  $48$ more.

Now to get a Lefschetz fibration homeomorphic to $\#_m  (S^2 \times S^2)$ for any given odd $m \geq 415$, write $m$ as $m=415+24n+2k$, where $n$ and $k$ are unique non-negative integers for $k < 12$. We first take  $(Z',f')$ and apply $k$ genus--$2$ breedings as described above, and then take the fiber sum of the resulting fibration with $(Z,f)$ and $n$ copies of $(X,f)$, using any gluing that preserves $s_0 \in \textrm{Spin}(\Sigma_9)$. (Untwisted fiber sums prescribed by a product of the monodromy factorizations we gave for these fibrations would do it.) No matter how the fundamental group changes after the genus--$2$ breedings, the extra fiber sum summand $(Z,f)$ ensures that the resulting fibration is simply-connected. Finally, calculating the Euler characteristic, we conclude that the simply-connected spin, signature zero total space of the genus--$9$ Lefschetz fibration we have is homeomorphic to $\#_m  (S^2 \times S^2)$.

This concludes the proof of Corollary~C. \qed

\smallskip
\begin{remark}
It is plausible that, with a more detailed study of the fundamental group of $(X,f)$ and that of twisted fiber sums, one can further improve Corollary~C to get smaller values of $m$. Since we will produce examples of exotic $\#_m (S^2 \times S^2)$ with much smaller topology in Theorem~D, here we content ourselves with examples we could get without getting bogged down with technical details. 
It should be easy to observe that, following a similar construction scheme one can also get minimal symplectic Lefschetz fibrations homeomorphic but not diffeomorphic to $\#_m (S^2 \, \twprod \, S^2)$. This can be achieved essentially with less effort, since we are no longer concerned about matching spin structures, and a single breeding with a non-spin pencil yields an odd intersection form. 
\end{remark}

\begin{remark} \label{KnottedSurfaces}
Involutions on genus--$1$ Lefschetz fibrations were used by Finashin, Kreck and Viro in \cite{FinashinKreckViro} to produce an infinite family of exotically knotted non-orientable surfaces in $S^4$; namely, a family of pairwise non-diffeomorphic surfaces, all ambiently homeomorphic to a standard embedding of $\#_{10} \RP$ with normal Euler number $16$. These were obtained as the fixed point sets of involutions on exotic elliptic surfaces, descending to the quotient, the standard $S^4$. This result was later improved by Finashin in \cite{Finashin}, who produced exotically knotted $\#_{6} \RP$ with normal Euler number $8$, and more recently by Havens in \cite{Havens}, who, by using Finashin's techniques, obtained further irreducible examples. The departing point of the latter works is once again involutions on genus--$1$ Lefschetz fibrations. In stark contrast, no examples of exotically knotted \emph{orientable} surfaces in $S^4$ are known, and there has been several attempts to show that there are no smooth knottings of the standardly embedded ones (that bound handlebodies), which is known as the \emph{Smooth Unknotting Conjecture}. Since any such examples necessarily come from an involution on an exotic $\#_m (S^2 \times S^2)$ with fixed point set homeomorphic to $\Sigma_m$, explicit examples as in Corollary~C, which one may attempt to build the desired involutions on, were sought for quite a while. Similarly, examples of exotic knottings of the standardly embedded non-orientable surface with Euler number $0$ would come from an involution on an exotic $\#_m (S^2 \, \twprod \, S^2)$. With the explicit fibration structure, the examples we built in this section come as prime candidates for this strategy. 
\end{remark}

\enlargethispage{0.15in}
\smallskip
\subsection{Smaller, but not fibered, exotic $\#_m  (S^2 \times S^2)$ }  \

To produce symplectic $4$--manifolds homeomorphic but not diffeomorphic to connected sums of smaller numbers of $S^2 \times S^2$, we will adopt an approach similar to the one employed by the first author and Korkmaz to build an exotic $\CP \# 4 \CPb$ in \cite[\S 5.3]{BaykurKorkmaz}:  First, we will extend our spin, signature zero Lefschetz fibration $(X,f)$ over the $2$--sphere to a  Lefschetz fibration $(X', f')$ over the $2$--torus.  While this will enlarge the fundamental group, now the symplectic $4$--manifold $(X', \omega')$ will have many homologically essential Lagrangian tori carrying the generators of $\pi_1(X')$. We will then apply Luttinger surgeries to these tori in the same fashion as in, for example \cite{FPS, BaldridgeKirk, AkhmedovParkOdd, AkhmedovBaykurPark, ABBKP}, to derive  a symplectic $4$--manifold with trivial fundamental group, which is homeomorphic to $\#_{23}  (S^2 \times S^2)$. We will then show how to produce an infinite family of such examples and how to obtain similar examples in the homeomorphism class of $\#_m  (S^2 \times S^2)$ for every odd $m \geq 23$.

\medskip
\noindent \textit{Proof of Theorem~D}. \,We noted in the previous subsection that the signature zero genus--$9$ Lefschetz fibration $(X,f)$ of Theorem~\ref{KeyLF} has a monodromy factorization of the form 
\mbox{$t_{\beta_1} t_{\beta'_1} t_{\beta_5}  \, Q_1 = 1$} in $\M(\Sigma_9)$, where $\beta_1, \beta'_1, \beta_5$ are as in Figure~\ref{fig:killingpi1}(a), and $Q_1$ is a product of the remaining positive Dehn twists. Given any disjoint, non-separating curves $A_1, A_2, A_3$ that are linearly independent in $H_1(\Sigma_9)$, there is a diffeomorphism of $\Sigma_9$ that takes $(\beta_1, \beta'_1, \beta_5)$ to $(A_1, A_2, A_3)$. Conjugating with this diffeomorphism, we thus get a monodromy factorization  of the form $t_{A_1} t_{A_2} t_{A_3} \, R =1$ in $\M(\Sigma_9, s)$ for $(X,f)$, where $R$ is a product of $45$ positive Dehn twists and $s$ is any spin structure on $\Sigma_9$ yielding a spin structure $\mathfrak{s}$ on $X$.

\smallskip
\noindent \textit{\underline{Symplectic $4$--manifolds homeomorphic to  
$\#_{23}  (S^2 \times S^2)$}}: 
For any $\mathcal{A}, \mathcal{B} \in \M(\Sigma_9)$ that commute with each other, we can extend the Lefschetz fibration $(X,f)$ over the $2$--sphere to a Lefschetz fibration $(X',f')$ over the $2$--torus with a monodromy factorization
\begin{equation} \label{LFoverT2}
[\mathcal{A}, \mathcal{B}] \ t_{A_1} t_{A_2} t_{A_3} \, R =1 \text{ in } \M(\Sigma_9) \, ,
\end{equation}
where the first term is the commutator of $\mathcal{A}$ and $\mathcal{B}$. Moreover, when $\mathcal{A}, \mathcal{B} \in \M(\Sigma_9, s)$, this becomes a factorization in $\M(\Sigma_9, s)$, and can be seen to prescribe a spin Lefschetz fibration over the $2$--torus. For our construction to follow, it will suffice to take $\mathcal{A}= \mathcal{B}=1$. In this case, the extension of the spin structure of $X$ to that of $X'$ can be easily seen by first viewing $X'$ as $X'= (X \setminus \nu(F)) \cup (\Sigma_9 \times \Sigma_1^1)$, where $\nu(F)$ is a fibered 
neighborhood of a regular fiber $F\cong \Sigma_9$ of $(X,f)$. Taking the product of $s \in \textrm{Spin}(\Sigma_9)$ with any $s' \in  \textrm{Spin}(\Sigma_1^1)$ we get a spin structure on $\Sigma_9 \times \Sigma_1^1$ inducing the same spin structure on its boundary as the restriction of the spin structure of $X$ to $\partial \nu(F) \cong \Sigma_9 \times S^1$.  Gluing these spin structures we get a spin structure $\mathfrak{s'}$ on $X'$. \linebreak In particular $X'$ has an even intersection form.

We have $\eu(X')=4(g-1)(h-1) + l = 4 \cdot 8 \cdot 0 +48=48$ (where $g, h$ are the fiber and base genera, $l$ is the number of Lefschetz critical points), and $\sigma(X')=\sigma(X)=0$ since the factorization~\eqref{LFoverT2} of $(X',f')$ is obtained from the monodromy factorization of $(X,f)$ by adding a trivial relation for the trivial commutator. 

We now set up our notation for the Luttinger surgeries and the fundamental group calculation, following the conventions in \cite{FPS, AkhmedovBaykurPark}. Consider the surface $\Sigma_9$ with its standard cell decomposition prescribed by a regular $36$--gon with vertex $x$, and with edges labeled as $\prod_{i=1}^9 a_i b_i a^{-1}_i b^{-1}_i$ as we go around the perimeter. Similarly take $\Sigma_1^1 = T^2 \setminus D$, where $D$ is an open disk, with its standard cell decomposition given by a rectangle with a hole,  with vertex $y$ and with edges labeled as $a b a^{-1} b^{-1}$. So $\{a_i, b_i\}_{i=1}^9$ generate $\pi_1(\Sigma_9)$ and $\{a, b\}$ generate $\pi_1(\Sigma_1^1)$ at the base points $x$ and $y$, respectively. Finally for any $c \in \{a_i, b_i, a, b\}$ let $c', c''$ denote the parallel copies of the curve $c$ on the same surface, as in \cite{FPS}[Figure~2]. By a slight abuse of notation, we will also denote any curve of the form $c \x y$ or $x \x c$ in $\Sigma_9 \x \Sigma_1^1$ by $c$.

Going forward, we take the Dehn twist curves $A_i$ in the factorization~\eqref{LFoverT2} to be isotopic to $a_i$, for $i=1, 2, 3$ and we assume $D$ above is an open disk in the base $T^2$ of the fibration \mbox{$f'\colon X' \to S^2$} containing all the critical values. As we added a trivial commutator, i.e. $\mathcal{A}= \mathcal{B}=1$, we have $\pi_1(X') \cong (\langle a, b \rangle  \times  \pi_1(\Sigma_g)) \, / \, N' \, $, where $N'$ is the subgroup normally generated by the Dehn twist curves in $\pi_1(\Sigma_9)$, together with an extra relation of the form $[a,b]=\mathcal{W}$ for some product of commutators $\mathcal{W} \in [\pi_1(\Sigma_9), \pi_1(\Sigma_9)]$. (The existence of $\mathcal{W}$ is implied by the existence of a pseudosection of $(X,f)$, and we would have $\mathcal{W}=1$ if $(X,f)$ had an honest section.) This can be most easily seen by applying the Seifert-van Kampen theorem to the decomposition $X'= (X \setminus \nu(F)) \cup (\Sigma_9 \times \Sigma_1^1)$. Nonetheless, for our fundamental group calculation below, it will suffice to just know that $N'$ contains relations induced by three Dehn twist curves $a_1, a_2, a_3$. 

The parallel transport of any $\alpha \in \{a'_i, b'_i, b''_i \}$ over any $\gamma \in \{a', b', b'' \}$ is a Lagrangian torus $T$ fibered over $\gamma$. Through the trivialization $X' \setminus (f'^{-1}(D)) \cong \Sigma_9 \times \Sigma_1^1$, we can view $T$ as a Lagrangian $\alpha \times \gamma$ with respect to a product symplectic form on $\Sigma_9 \times \Sigma_1^1$. Note that for $\nu(\alpha), \nu(\gamma)$ the normal neighborhoods of $\alpha, \gamma$ in $\Sigma_9, \Sigma_1^1$, the Weinstein  neighborhood of the Lagrangian torus $\alpha \times \gamma$ is $\nu(\alpha) \times \nu(\gamma)$. Encoding the surgery information by the triple $(T, \lambda, k)$, as in \cite{AkhmedovBaykurPark, FPS}, we claim that performing the following disjoint Luttinger surgeries in $X'$:
\[ (a'_1 \x b', b', 1), \, (b_2'' \x a', a', 1), (a_i' \x b', a_i', 1),  (b_j' \x b'', b_j', 1)  \text{ for }  i=4, \ldots, 9 \,\text{ and }j=1, \ldots, 9 \]
results in a simply-connected symplectic $4$--manifold $X''$. We take $x \times y$ as the base point for the fundamental group calculation. Per the choices we made here, we can invoke the work of Baldridge and Kirk in \cite{BaldridgeKirk} to deduce that $\pi_1(X'')$ has a presentation with generators $a_i, b_i, a, b$, for which the following relations hold (among several others we do not include here): 
\begin{equation*}
 a_1 =a _2 = a_3= 1 \, , 
 \end{equation*}
 \begin{equation*}
 \mu' \,  b = \mu \, a  = \mu_i \, a_i = \mu'_j \,  b_j =1, \text{ for } i=4, \ldots, 9 \text{ and }  j=1, \ldots, 9 \, ,
\end{equation*}
where the first three relations come from the vanishing cycles of our fibration, and   $\mu', \mu, \mu_i, \mu'_j$ in the second line are the meridians of the surgered Lagrangian tori, given by conjugates of commutators of the pairs $\{b_1, a\}$, $\{a_2, b\}$, $\{b_i, a\}$, $\{a_j, a\}$, respectively. (While we can write out the exact commutators following \cite{BaldridgeKirk, FPS}, there will be no need for these details for our calculation to follow.) Since $a_2=1$, the second commutator is trivial, so $a=1$ by the relation $\mu \, a =1$. Since all others are commutators of $a$, they too are trivial, which implies that $b=a_i=b_j=1$ for all $i=4, \ldots, 9$ and $j=1, \ldots 9$, by the the remaining surgery relations above. Now, because $a_1=a_2=a_3=1$, we see that all the generators of $\pi_1(X'')$ are trivial, so $X''$ is simply-connected.

As we obtained $X''$ from $X'$ via Luttinger surgeries, $X''$ admits a symplectic form $\omega''$ \cite{ADK}. Since these surgeries along tori do not change the Euler characteristic or the signature, we have $\eu(X'')=\eu(X')=48$ and $\sigma(X'')=\sigma(X')=0$. Moreover, each surgery is performed along a Lagrangian torus from a  pair of geometrically dual Lagrangian tori, which describe a hyperbolic pair in $H_2(X)$ with respect to the intersection form. These pairs can be seen to be all disjoint. (For example we can take the Lagrangian tori $b''_1 \x a'$, $a'_2 \x b'$, $b''_i \x a'$, $a''_j \x a''$ as the duals.) Since the intersection form $Q_{X'}$ is an extension of that of $Q_{X''}$ by such hyperbolic pairs, it follows that $X''$ is also even, and as $H_1(X'')$ has no \mbox{$2$--torsion,} we conclude that $X''$ is spin. By Freedman's celebrated result, $X''$ is homeomorphic to $\#_{23} \, (S^2 \x S^2)$.  

To obtain an infinite family of examples, first note that the Lagrangian torus $a'_3 \times b'$ and its dual (which we can take as $b''_3 \times a'$) in $X'$ are disjoint from all the other tori (and their dual tori) we surgered. So it descends to a homologically essential Lagrangian torus $T$ in $X''$, for $\omega''$ agrees with $\omega'$ away from the surgery tori \cite{ADK}. As shown by Gompf \cite{Gompf}, we can perturb $\omega''$ so that $T$ becomes a self-intersection zero symplectic torus. Importantly, $\pi_1(X'' \setminus T) = \pi_1(X)=1$, which follows from the fact that the meridian of $T$ in $X'$ was  a conjugate of a commutator of the pair $\{b_3, a\}$ that became trivial in $\pi_1(X'')$. Hence, we can perform Fintushel-Stern knot surgery \cite{FSKnotSurgery} to $X''$ along $T$,  using an infinite family of fibered knots $K_n$ with distinct Alexander polynomials, and produce an infinite family of symplectic $4$--manifolds $(X''_n, \omega''_n)$ which are pairwise non-diffeomorphic but are all homeomorphic to $X''$, and thus,  to $\#_{23} \, (S^2 \x S^2)$.

Lastly, observe that we can run the same construction using spin, signature zero genus $g=8d+1$ Lefschetz fibrations  $(\tilde{X}(d), \tilde{f}(d))$ we built in the proof of Theorem~A. In this case, we will have a similar list of Luttinger surgeries along the Lagrangian tori in the extended fibration $(\tilde{X}(d)', \tilde{f}(d)')$ over the $2$--torus, except now the indices $i$ and $j$ will run up to $8d+1$. This way we see that there are symplectic Lefschetz fibrations over the $2$--torus which are equivalent via Luttinger surgeries to symplectic $4$--manifolds homeomorphic to  $\#_{24d-1} (S^2 \times S^2)$, for any $d \in \Z^+$.

\enlargethispage{0.1in}
\smallskip
\noindent \textit{\underline{Stable range}}: 
As our observation above provides exotic $\#_{m} (S^2 \times S^2)$ only for $m \equiv 23$ (mod 24), we still need to show how to get examples for \emph{every} odd $m \geq 23$.  We will achieve this by taking symplectic fiber sums of $(X'', \omega'')$ above with copies of a small symplectic $4$--manifold we will quickly derive from \cite{FPS} as follows: Take $Y_0=\Sigma_2 \x \Sigma_2$ with a product symplectic form $\omega_0$. As before, let $a_i, b_i$ and $c_j, d_j$ denote the $\pi_1$ generators of the first and second copies of $\Sigma_2$ in $Y_0$, let $x$ and $y$ be the base points we take on them, and for any $c \in \{a_i, b_i, c_j, d_j\}$ let the parallel copies $c', c''$ of $c$ be described in the same fashion. Performing the following Luttinger surgeries in $Y_0$:\footnote{These surgeries are essentially the same as the ones employed in the construction of the homology $S^2 \x S^2$ in \cite{FPS}, except here we do not perform surgeries along the Lagrangian tori $a'_1 \x c'_1$ and $a'_2 \x c'_1$ (as we have different plans for them) and we simply took all the surgery coefficients to be $+1$ (since the effect of surgery coefficient $\pm 1$ on the $\pi_1$ calculation will be simply about killing a generator or its inverse).}
\[ (b'_1 \x c''_1, b_1', 1), \, (a'_2 \x c'_2, a'_2, 1), (b'_2 \x c''_2, b'_2 , 1),  
(a''_2 \x d'_1, d'_1, 1),  (a'_1  \x c'_2 , c'_2, 1), ( a''_1  \x d'_2 , d'_2, 1), \]
we obtain the desired symplectic $4$--manifold $Y$. Clearly $\eu(Y)=\eu(Y_0)=4$ and $\sigma(Y)=\sigma(Y_0)=0$. The following relations hold in $\pi_1(Y)$, based at $x \times y$, between the generators $a_i, b_i, c_j, d_j$:
\begin{equation*}
 \mu_1 \,  b_1 = \mu_2 \, a_2  = \mu_3 \, b_2 = \mu_4  \,  d_1 =\mu_5 \, c_2 = \mu_6 \, d_2 =1,  \, 
\end{equation*}
where $\mu_k$, for $k=1, \ldots, 6$, are the meridians of the surgered Lagrangian tori, given by conjugates of commutators of the pairs $\{a_1, d_1\}$, $\{b_2, d_2\}$, $\{a_2, d_2\}$, $\{b_2, c_1\}$,$\{b_1, d_2\}$,$\{b_1, c_2\}$, respectively. We claim that $a_1$ and $c_1$ normally generate $\pi_1(Y)$. To see this, add extra relations $a_1=c_1=1$. Then, the first and fourth commutators are trivial, so $b_1=d_1=1$ by the relations $\mu_1\, b_1=\mu_4\, d_4 = 1$. But $b_1=1$ implies that fifth and sixth commutators are trivial, and thus $c_2=d_2=1$ by the relations $\mu_5\, c_2 = \mu_6\, d_2=1$. Since now $d_2=1$,  second and  third commutators are trivial, so $a_2=b_2=1$ as well, by the corresponding relations $\mu_2\, a_2 = \mu_3\, b_2 =1$. Hence trivializing $a_1$ and $c_1$ kills all of $\pi_1(Y)$ as claimed.  If we let $T'$ and $T''$ denote the homologically essential Lagrangian tori in $Y$ descending from $a'_1 \x c'_1$ and $a'_2 \x c'_1$ in $Y_0$ (which, along with their geometric duals $b''_1 \x d'_1$ and $b'_2 \x d''_1$, are disjoint from the other surgered tori), their meridians $\mu'$ and $\mu''$ are conjugates of a commutator of $\{b_1, d_1\}$ and of $\{b_2, d_1\}$, respectively. It follows that $\pi_1(Y \setminus (T' \sqcup T''))$ is normally generated by $a_1$ and $c_1$ as well. 

Recall that by perturbing the symplectic form on $X''$  we got a self-intersection zero symplectic torus $T$ in $X''$ with $\pi_1(X'' \setminus T)=1$. Let us continue denoting the perturbed symplectic form on $X''$ as $\omega''$. Similarly, after perturbing the symplectic form, the Lagrangian tori $T'$ and $T''$ (which are homologically essential and independent) become symplectic in $(Y, \omega_Y)$.
We can thus take the symplectic fiber sum of $(X'', \omega'')$ and $(Y, \omega_Y)$ along $T$ and $T'$ to get $(X''_1, \omega''_1)$.

We have $\eu(X''_1)=\eu(X'')+\eu(Y) - 2 \eu(T^2)=48+4=52$ and $\sigma(X''_1)=\sigma(X'')+\sigma(Y)=0+0=0$. Since the image of the generators of $\pi_1(\partial (\nu T'))$ under the boundary inclusion map are $a_1, c_1$ and $\mu$,\linebreak and since $\pi_1(X'' \setminus T)=1$ and $\pi_1(Y \setminus T')$ is normally generated by $a_1$ and $c_1$, by applying the Seifert-van Kampen theorem to the decomposition $X''_1= (X'' \setminus \nu T) \cup (Y \setminus \nu T')$, we conclude that $\pi_1(X''_1)=1$. A quick way to see that $X''_1$ is spin is the following: Reversing the order of Luttinger surgeries and the symplectic fiber sum, which were performed along disjoint subsurfaces, we could obtain $X''_1$ by first taking a symplectic fiber sum of $(X'', \omega'')$ with $(Y_0, \omega_0)$ along $T$ and $a'_1 \x c'_1 $, which we can assume to be spin by taking a spin structure on the product $4$--manifold $Y_0$ whose restriction on the fiber sum region agrees with that of the restriction of the spin structure on $X''$ \cite{Gompf}. In particular we get a $4$--manifold with an even intersection form. But then, when we perform the Luttinger surgeries along the above tori contained in the $Y_0$ factor, the result is $X''_1$ and the intersection form only changes by removing the  hyperbolic pairs corresponding to the surgery tori and their duals. Therefore  $X''_1$ too has an even intersection form, and since $H_1(X''_1)$ has no \mbox{$2$--torsion,} $X''_1$ is spin. Hence $X''_1$ is a simply-connected spin symplectic $4$--manifold, which is homeomorphic to  $\#_{25} (S^2 \times S^2)$ by Freedman. Knot surgery along the other symplectic torus $T''$ that descends from $Y$ to $X''_1$ yields an infinite family of pairwise non-diffeomorphic symplectic $4$--manifolds in the same homeomorphism class. 

Now for  $k>1$, we can build a symplectic $4$--manifold $(X''_k, \omega''_k)$ by taking a symplectic fiber sum of $(X'', \omega'')$ and $k$ copies of $(Y, \omega_Y)$ by first fiber summing $(X'', \omega'')$ and $(Y, \omega_Y)$ along \mbox{$T$ and $T'$} as above,  then ---without performing the knot surgery-- fiber summing the resulting symplectic \mbox{$4$--manifold} $(X''_1, \omega''_1)$ with the next copy of $(Y, \omega_Y)$ along $T''$ (which descends to $X''_1$) and $T'$, and repeating the latter until we add all $k$ copies of $(Y, \omega_Y)$. At the very end of this procedure, we can perform  knot surgery along $T''$ coming from the very last copy of $Y$ to produce infinitely many pairwise non-diffeomorphic symplectic \mbox{$4$--manifolds} homeomorphic to $X''_k$ as before. \linebreak A straightforward calculation as above shows that the symplectic $4$--manifold $X''_k$ has $\pi_1(X''_k)=1$, $\eu(X''_k)= 48 +4k$ and $\sigma(X''_k)=0$. Thus $X''_k$ is homeomorphic to $\#_{23+2k} (S^2 \times S^2)$, for $k \in \Z^+$.
\qed.

\smallskip
\begin{remark}
Unlike the previous works  \cite{Park, APspin1, APspin2}, our construction of exotic $\#_{m}  (S^2 \times S^2)$ does not build on a  compact complex surface produced by algebraic geometers. The spin symplectic \mbox{$4$--manifold} $(X', \omega')$, to which we applied symplectic surgeries, has $b_1(X')=9$, and thus it cannot even be homotopy equivalent to a compact complex surface. Moreover, since there can be only finitely many deformation classes of simply-connected complex surfaces with the same Chern numbers ($c_1^2=2\eu + 3\sigma$ and $c_2=\eu$), all but  finitely many of our exotic symplectic $\#_{m}  (S^2 \times S^2)$ (for fixed $m$) are necessarily non-complex. Performing the knot surgeries we employed in our constructions, using \emph{non-fibered} knots instead with distinct Alexander polynomials, we also get infinitely many exotic  $\#_{m}  (S^2 \times S^2)$ (for fixed $m$) which do not admit symplectic structures \cite{FSKnotSurgery}. 
\end{remark}

\enlargethispage{0.2in}
\smallskip
\appendix
\section{Hurwitz equivalence for the genus--$2$ and genus--$3$ pencils}

Here we address the question of whether the signature zero, spin genus--$2$ and genus--$3$ pencils we constructed in Sections~\ref{SecGenus2} and~\ref{SecGenus3} are new additions to the literature. Many genus--$2$ pencils on ruled surfaces were obtained in \cite{Hamada}, and genus--$3$ pencils on symplectic Calabi-Yau surfaces with $b_1 >0$ in \cite{BaykurGenus3, HamadaHayano}. While our constructions are new, we are able to observe that the monodromy factorizations of the genus--$2$ and genus--$3$ pencils we constructed here are in fact Hurwitz equivalent to the monodromy factorizations of pencils in \cite{Hamada, BaykurGenus3}. Note that the breeding sequence we employed in the construction of our genus--$2$ and genus--$3$ pencils allowed us to carry out the essential calculation for a pseudosection of the key signature zero, spin genus--$9$ Lefschetz fibration we built out of them. However this is irrelevant for the Hurwitz equivalence of  pencil monodromies, and below we will  forget the marked points.

\subsection{The genus--$2$ pencil on $T^2 \times S^2$} \

In~\cite{Hamada}, the second author described several lifts of the monodromy factorization in $\M(\Sigma_2)$ for  Matsumoto's well-known genus--$2$ Lefschetz fibration  \cite{Matsumoto} to monodromy factorizations in $\M(\Sigma_2^4)$ for genus--$2$ pencils on ruled surfaces. Here we will observe that the genus--$2$ pencil we constructed in Section~\ref{SecGenus2} is isomorphic to one of these.

The pencil referred to as $W_{\mathrm{\rm II}A}$ in \cite{Hamada} has the monodromy factorization
\begin{align} \label{eq:MatsumotoLPWIIA}
\DT{B_{0,1}} \DT{B_{1,1}} \DT{B_{2,1}} \DT{C_1} \DT{B_{0,2}} \DT{B_{1,2}} \DT{B_{2,2}} \DT{C_2} = \DT{\delta_1} \DT{\delta_2} \DT{\delta_3} \DT{\delta_4},
\end{align}
where the curves are as shown in Figure~\ref{fig:MatsumotoSections}.
\begin{figure}[htbp]
	\centering
	\includegraphics[height=110pt]{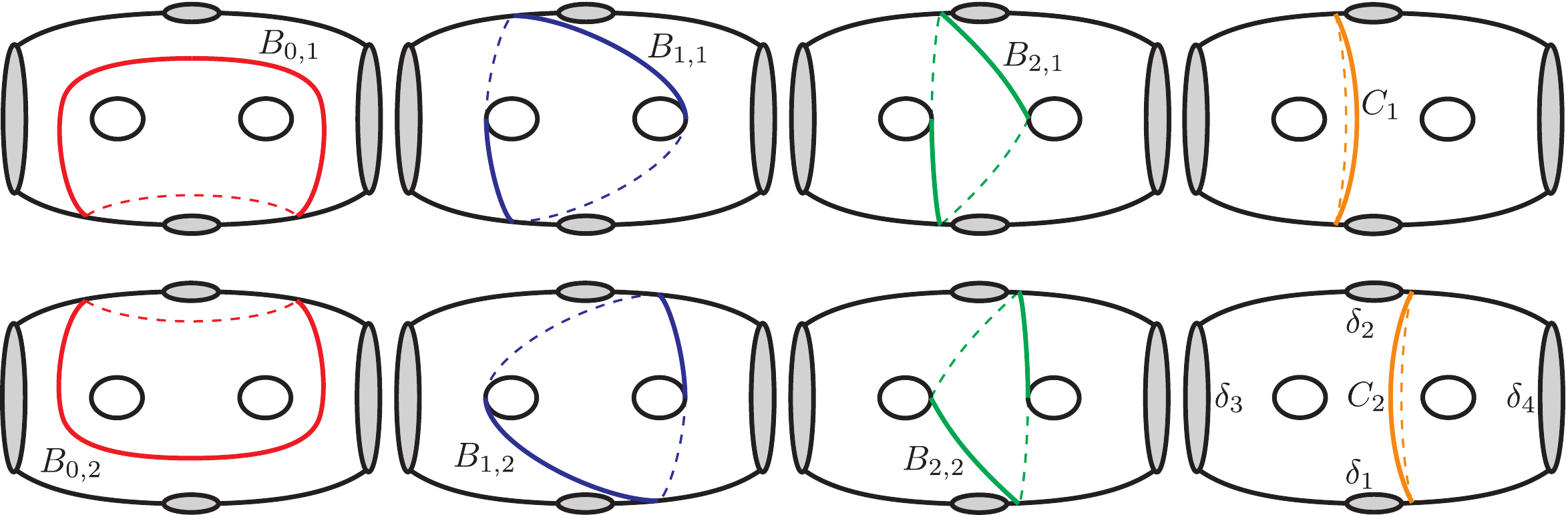}
	\caption{The monodromy curves for the pencil $W_{\mathrm{\rm II}A}$ of \cite{Hamada}.} \label{fig:MatsumotoSections}
\end{figure}

We perform the following Hurwitz moves to this factorization:
\begin{align*}
\DT{\delta_1} \DT{\delta_2} \DT{\delta_3} \DT{\delta_4} &=
\DT{B_{0,1}} \DT{B_{1,1}} \DT{B_{2,1}} \DT{C_1} \underline{\DT{B_{0,2}} \DT{B_{1,2}} \DT{B_{2,2}} \DT{C_2}} \\
&\sim \DT{B_{0,1}} \DT{B_{1,1}} \DT{B_{2,1}} \DT{C_1} \underline{\DT{B_{1,2}} \DT{B_{2,2}} \DT{C_2}} \DT{B_{0,2}^\prime} \\
&\sim \DT{B_{0,1}} \DT{B_{1,1}} \DT{B_{2,1}} \DT{C_1} \underline{\DT{B_{2,2}} \DT{C_2}} \DT{B_{1,2}^\prime} \DT{B_{0,2}^\prime} \\
&\sim \DT{B_{0,1}} \DT{B_{1,1}} \DT{B_{2,1}} \DT{C_1} \DT{C_2} \DT{B_{2,2}^\prime} \DT{B_{1,2}^\prime} \DT{B_{0,2}^\prime},
\intertext{where $B_{0,2}^\prime = \LDT{C_2} \LDT{B_{2,2}} \LDT{B_{1,2}}(B_{0,2})$,
$B_{1,2}^\prime = \LDT{C_2} \LDT{B_{2,2}}(B_{1,2})$,
$B_{2,2}^\prime = \LDT{C_2}(B_{2,2})$, 
and as it turns out that $B_{0,1}$ and $B_{0,2}^\prime$ are disjoint,}
&\sim \DT{B_{0,2}^\prime} \DT{B_{1,1}} \DT{B_{2,1}} \DT{C_1} \DT{C_2} \DT{B_{2,2}^\prime} \DT{B_{1,2}^\prime} \DT{B_{0,1}}.
\end{align*}
To compare with the monodromy factorization of our genus--$2$ pencil, push the boundary components $\delta_1$ and $\delta_2$ as shown in Figure~\ref{fig:Genus2Push_MarkedPoint}.
Then we recognize that the last expression exactly coincides with 
\begin{align*} 
\DT{B_0} \DT{B_1} \DT{B_2} \DT{C} \DT{C^\prime} \DT{B_2^\prime} \DT{B_1^\prime} \DT{B_0^\prime}
\end{align*}
with the curves in Figure~\ref{fig:MatsumotoSectionsIIA_MirrorSymmetrics_MarkedPoint},
which is the monodromy factorization of the genus--$2$ pencil we built in Section~\ref{SecGenus2}, obtained from the factorization~\eqref{eq:MatsumotoLP_IIA_MarkedPoint} by forgetting the marked point. 
(In fact, we can also show that the $6$--holed torus relation~\eqref{eq:6holed_torus_relation} we obtained while building our genus--$2$ pencil is Hurwitz equivalent to the $6$--holed torus relation of Korkmaz-Ozbagci~\cite{KorkmazOzbagci}.)

\enlargethispage{0.2in}
\subsection{The genus--$3$ pencil on $T^4$}  \

In \cite{BaykurGenus3}, the first author constructed symplectic genus--$3$ Lefschetz pencils in every rational homology type of symplectic Calabi-Yau surfaces with $b_1>0$. We will show that our signature zero, spin genus--$3$ pencil with monodromy factorization~\eqref{eq:Genus3LP_4BoundingPairs}
\begin{align} \label{eq:Genus3LP_4BoundingPairs_closed}
\DT{a} \DT{a^\prime} \DT{x} \DT{b} \DT{b^\prime} \DT{y} \DT{c} \DT{c^\prime} \DT{z} \DT{d} \DT{d^\prime} \DT{w} 
= \DT{\delta_1} \DT{\delta_2} \DT{\delta_3} \DT{\delta_4}
\end{align}
in $\M(\Sigma_3^4)$ (which we quote here without the marked point) is Hurwitz equivalent to the genus--$3$ pencil on a symplectic Calabi-Yau $4$--torus in~\cite{BaykurGenus3}.
After applying a cyclic permutation,
\begin{align*}
\DT{\delta_1} \DT{\delta_2} \DT{\delta_3} \DT{\delta_4} 
&=
\underline{\DT{d^\prime} \DT{w} \DT{a} \DT{a^\prime} \DT{x} \DT{b}} \cdot \underline{\DT{b^\prime} \DT{y} \DT{c} \DT{c^\prime} \DT{z} \DT{d}} \\
&\sim 
\underline{\DT{w} \DT{a} \DT{a^\prime} \DT{x} \DT{b}} \DT{A_2} \cdot 
\underline{\DT{y} \DT{c} \DT{c^\prime} \DT{z} \DT{d}} \DT{A_2^\prime}
\intertext{where $A_2 = \LDT{b} \LDT{x} \LDT{a^\prime} \LDT{a} \LDT{w}(d^\prime)$, $A_2^\prime = \LDT{d} \LDT{z} \LDT{c^\prime} \LDT{c} \LDT{y}(b^\prime)$,}
&\sim 
\DT{a} \underline{\DT{a^\prime} \DT{x} \DT{b}} \DT{A_1} \DT{A_2} \cdot 
\DT{c} \underline{\DT{c^\prime} \DT{z} \DT{d}} \DT{A_1^\prime} \DT{A_2^\prime}
\intertext{where $A_1 = \LDT{b} \LDT{x} \LDT{a^\prime} \LDT{a}(w)$, $A_1^\prime = \LDT{d} \LDT{z} \LDT{c^\prime} \LDT{c}(y)$,}
&\sim 
\DT{a} \DT{x} \DT{b} \DT{A_0} \DT{A_1} \DT{A_2} \cdot 
\DT{c} \DT{z} \DT{d} \DT{A_0^\prime} \DT{A_1^\prime} \DT{A_2^\prime}
\end{align*}
where $A_0 = \LDT{b} \LDT{x}(a^\prime)$, $A_0^\prime = \LDT{d} \LDT{z}(c^\prime)$.
By relabeling $B_0=a$, $B_1=x$, $B_2=b$, $B_0^\prime=c$, $B_1^\prime=z$, $B_2^\prime=d$, we obtain
\begin{align} \label{eq:genus3LPbyInanc}
\DT{B_0} \DT{B_1} \DT{B_2} \DT{A_0} \DT{A_1} \DT{A_2} \cdot 
\DT{B_0^\prime} \DT{B_1^\prime} \DT{B_2^\prime} \DT{A_0^\prime} \DT{A_1^\prime} \DT{A_2^\prime}
=
\DT{\delta_1} \DT{\delta_2} \DT{\delta_3} \DT{\delta_4} 
\end{align}
which, after suitably sliding the boundary components, coincides with the positive factorization $W=\DT{\delta_1} \DT{\delta_2} \DT{\delta_3} \DT{\delta_4}$ in~\cite{BaykurGenus3}. As shown in \cite{HamadaHayano}, by further Hurwitz moves and a global conjugation, one can see that the monodromy factorization of the latter pencil is equivalent to that of the holomorphic Lefschetz pencil on the standard $T^4$ by Smith \cite{SmithT4}. This array of arguments shows that the symplectic Calabi-Yau surface $Y$, which is the total space of our genus--$3$ pencil in Section~\ref{SecGenus3}, is in fact diffeomorphic to $T^4$.

\medskip

\end{document}